\documentclass[11pt]{amsart}

\usepackage{amssymb}
\usepackage{amsmath}
\usepackage[all,cmtip]{xy}
\usepackage{amsfonts}
\usepackage{mathrsfs}
\usepackage{latexsym}
\usepackage{graphicx}
\usepackage{amscd,amssymb,amsmath,amsbsy,amsthm}
\usepackage[colorlinks,plainpages,backref,urlcolor=blue]{hyperref}
\usepackage{verbatim}
\usepackage[normalem]{ulem}

\usepackage{tikz}
\usetikzlibrary{arrows,calc}
\tikzset{
>=stealth',
help lines/.style={dashed, thick},
axis/.style={<->},
important line/.style={thick},
connection/.style={thick, dotted},
}

\newcommand{\nc}{\newcommand}
\nc{\rnc}{\renewcommand}
\nc{\bb}[1]{{\mathbb #1}}
\nc{\bbA}{\bb{A}}\nc{\bbB}{\bb{B}}\nc{\bbC}{\bb{C}}\nc{\bbD}{\bb{D}}
\nc{\bbE}{\bb{E}}\nc{\bbF}{\bb{F}}\nc{\bbG}{\bb{G}}\nc{\bbH}{\bb{H}}
\nc{\bbI}{\bb{I}}\nc{\bbJ}{\bb{J}}\nc{\bbK}{\bb{K}}\nc{\bbL}{\bb{L}}
\nc{\bbM}{\bb{M}}\nc{\bbN}{\bb{N}}\nc{\bbO}{\bb{O}}\nc{\bbP}{\bb{P}}
\nc{\bbQ}{\bb{Q}}\nc{\bbR}{\bb{R}}\nc{\bbS}{\bb{S}}\nc{\bbT}{\bb{T}}
\nc{\bbU}{\bb{U}}\nc{\bbV}{\bb{V}}\nc{\bbW}{\bb{W}}\nc{\bbX}{\bb{X}}
\nc{\bbY}{\bb{Y}}\nc{\bbZ}{\bb{Z}}
\nc{\mbf}[1]{{\mathbf #1}}
\nc{\bfA}{\mbf{A}}\nc{\bfB}{\mbf{B}}\nc{\bfC}{\mbf{C}}\nc{\bfD}{\mbf{D}}
\nc{\bfE}{\mbf{E}}\nc{\bfF}{\mbf{F}}\nc{\bfG}{\mbf{G}}\nc{\bfH}{\mbf{H}}
\nc{\bfI}{\mbf{I}}\nc{\bfJ}{\mbf{J}}\nc{\bfK}{\mbf{K}}\nc{\bfL}{\mbf{L}}
\nc{\bfM}{\mbf{M}}\nc{\bfN}{\mbf{N}}\nc{\bfO}{\mbf{O}}\nc{\bfP}{\mbf{P}}
\nc{\bfQ}{\mbf{Q}}\nc{\bfR}{\mbf{R}}\nc{\bfS}{\mbf{S}}\nc{\bfT}{\mbf{T}}
\nc{\bfU}{\mbf{U}}\nc{\bfV}{\mbf{V}}\nc{\bfW}{\mbf{W}}\nc{\bfX}{\mbf{X}}
\nc{\bfY}{\mbf{Y}}\nc{\bfZ}{\mbf{Z}}
\nc{\bfa}{\mbf{a}}\nc{\bfb}{\mbf{b}}\nc{\bfc}{\mbf{c}}\nc{\bfd}{\mbf{d}}
\nc{\bfe}{\mbf{e}}\nc{\bff}{\mbf{f}}\nc{\bfg}{\mbf{g}}\nc{\bfh}{\mbf{h}}
\nc{\bfi}{\mbf{i}}\nc{\bfj}{\mbf{j}}\nc{\bfk}{\mbf{k}}\nc{\bfl}{\mbf{l}}
\nc{\bfm}{\mbf{m}}\nc{\bfn}{\mbf{n}}\nc{\bfo}{\mbf{o}}\nc{\bfp}{\mbf{p}}
\nc{\bfq}{\mbf{q}}\nc{\bfr}{\mbf{r}}\nc{\bfs}{\mbf{s}}\nc{\bft}{\mbf{t}}
\nc{\bfu}{\mbf{u}}\nc{\bfv}{\mbf{v}}\nc{\bfw}{\mbf{w}}\nc{\bfx}{\mbf{x}}
\nc{\bfy}{\mbf{y}}\nc{\bfz}{\mbf{z}}

\nc{\mcal}[1]{{\mathcal #1}}
\nc{\calA}{\mcal{A}}\nc{\calB}{\mcal{B}}\nc{\calC}{\mcal{C}}\nc{\calD}{\mcal{D}}
\nc{\calE}{\mcal{E}} \nc{\calF}{\mcal{F}}\nc{\calG}{\mcal{G}}\nc{\calH}{\mcal{H}}
\nc{\calI}{\mcal{I}}\nc{\calJ}{\mcal{J}}\nc{\calK}{\mcal{K}}\nc{\calL}{\mcal{L}}
\nc{\calM}{\mcal{M}}\nc{\calN}{\mcal{N}}\nc{\calO}{\mcal{O}}\nc{\calP}{\mcal{P}}
\nc{\calQ}{\mcal{Q}}\nc{\calR}{\mcal{R}}\nc{\calS}{\mcal{S}}\nc{\calT}{\mcal{T}}
\nc{\calU}{\mcal{U}}\nc{\calV}{\mcal{V}}\nc{\calW}{\mcal{W}}\nc{\calX}{\mcal{X}}
\nc{\calY}{\mcal{Y}}\nc{\calZ}{\mcal{Z}}
\nc{\fA}{\frak{A}}\nc{\fB}{\frak{B}}\nc{\fC}{\frak{C}} \nc{\fD}{\frak{D}}
\nc{\fE}{\frak{E}}\nc{\fF}{\frak{F}}\nc{\fG}{\frak{G}}\nc{\fH}{\frak{H}}
\nc{\fI}{\frak{I}}\nc{\fJ}{\frak{J}}\nc{\fK}{\frak{K}}\nc{\fL}{\frak{L}}
\nc{\fM}{\frak{M}}\nc{\fN}{\frak{N}}\nc{\fO}{\frak{O}}\nc{\fP}{\frak{P}}
\nc{\fQ}{\frak{Q}}\nc{\fR}{\frak{R}}\nc{\fS}{\frak{S}}\nc{\fT}{\frak{T}}
\nc{\fU}{\frak{U}}\nc{\fV}{\frak{V}}\nc{\fW}{\frak{W}}\nc{\fX}{\frak{X}}
\nc{\fY}{\frak{Y}}\nc{\fZ}{\frak{Z}}
\nc{\fa}{\frak{a}}\nc{\fb}{\frak{b}}\nc{\fc}{\frak{c}} \nc{\fd}{\frak{d}}
\nc{\fe}{\frak{e}}\nc{\fFf}{\frak{f}}\nc{\fg}{\frak{g}}\nc{\fh}{\frak{h}}
\nc{\fri}{\frak{i}}\nc{\fj}{\frak{j}}\nc{\fk}{\frak{k}}\nc{\fl}{\frak{l}}
\nc{\fm}{\frak{m}}\nc{\fn}{\frak{n}}\nc{\fo}{\frak{o}}\nc{\fp}{\frak{p}}
\nc{\fq}{\frak{q}}\nc{\fr}{\frak{r}}\nc{\fs}{\frak{s}}\nc{\ft}{\frak{t}}
\nc{\fu}{\frak{u}}\nc{\fv}{\frak{v}}\nc{\fw}{\frak{w}}\nc{\fx}{\frak{x}}
\nc{\fy}{\frak{y}}\nc{\fz}{\frak{z}}

\newtheorem{theorem}{Theorem}[section]
\newtheorem{lemma}[theorem]{Lemma}
\newtheorem{corollary}[theorem]{Corollary}
\newtheorem{prop}[theorem]{Proposition}

\theoremstyle{definition}
\newtheorem{definition}[theorem]{Definition}
\newtheorem{example}[theorem]{Example}
\newtheorem{remark}[theorem]{Remark}
\newtheorem{question}[theorem]{Question}
\newtheorem{problem}[theorem]{Problem}

\newtheorem{assumption}[theorem]{Assumption}

\newtheorem{thm}{Theorem}

\DeclareMathOperator{\rank}{rank} \DeclareMathOperator{\gr}{gr}
 
 \DeclareMathOperator{\id}{id}
 
\DeclareMathOperator{\ch}{ch} 
 \DeclareMathOperator{\GL}{GL}
\DeclareMathOperator{\Hom}{{Hom}} 
\DeclareMathOperator{\Ext}{{Ext}}
\DeclareMathOperator{\sHom}{{\mathscr{H}om}}

\DeclareMathOperator{\proj}{pr} 
\DeclareMathOperator{\pro}{pro} 
\DeclareMathOperator{\alg}{alg} 

\DeclareMathOperator{\Lie}{Lie}
\DeclareMathOperator{\Spec}{{Spec}} 
\DeclareMathOperator{\Aut}{Aut}
 \DeclareMathOperator{\End}{End}
\DeclareMathOperator{\sEnd}{{\mathscr{E}nd}}

\DeclareMathOperator{\Coh}{Coh}
\DeclareMathOperator{\Mod}{Mod\hbox{-}}
\DeclareMathOperator{\Gm}{\mathbb{G}_m}

\DeclareMathOperator{\Ell}{\mathcal{E}}
\DeclareMathOperator{\Pic}{Pic}

\DeclareMathOperator{\ad}{ad}

\DeclareMathOperator{\SL}{SL}
\DeclareMathOperator{\lcm}{lcm}
\DeclareMathOperator{\Proj}{Proj}

\newcommand{\catA}{\mathfrak{A}}
\DeclareMathOperator{\Rep}{Rep}
\DeclareMathOperator{\Res}{Res}
\DeclareMathOperator{\Ind}{Ind}

\newcommand{\surj}{\twoheadrightarrow}
\newcommand{\inj}{\hookrightarrow}

\def\angl#1{{\langle #1\rangle}}

\newcommand{\pt}{\text{pt}}
\newcommand{\Aff}{\bbA}

\newcommand{\PP}{\bbP}
\newcommand{\ZZ}{\bbZ}

\newcommand{\Z}{\bbZ}

\DeclareMathOperator{\sn}{sn}

\newcommand{\al}{\alpha}
\newcommand{\la}{\lambda}
\newcommand{\ga}{\gamma}

\newcommand{\UP}{\catA/W}
\newcommand{\La}{\Lambda}
\newcommand{\Dem}{\Delta}   % formal Demazure operator
\newcommand{\de}{\delta}

\newcommand{\varth}{\vartheta}
\DeclareMathOperator{\Sing}{Sing }
\newcommand{\wt}{\widetilde}

\DeclareMathOperator{\dyn}{{dyn}}

\newcount\cols
{\catcode`,=\active\catcode`|=\active
 \gdef\Young(#1){\hbox{$\vcenter
 {\mathcode`,="8000\mathcode`|="8000
  \def,{\global\advance\cols by 1 &}%
  \def|{\cr
        \multispan{\the\cols}\hrulefill\cr
        &\global\cols=2 }%
  \offinterlineskip\everycr{}\tabskip=0pt
  \dimen0=\ht\strutbox \advance\dimen0 by \dp\strutbox
  \halign
   {\vrule height \ht\strutbox depth \dp\strutbox##
    &&\hbox to \dimen0{\hss$##$\hss}\vrule\cr
    \noalign{\hrule}&\global\cols=2 #1\crcr
    \multispan{\the\cols}\hrulefill\cr%
   }
 }$}}
}

\title[Elliptic affine Hecke algebras]
{Representations of the elliptic affine Hecke algebras}
\date{\today}

\author[G.~Zhao]{Gufang~Zhao}
\address{Institut de Math\'ematiques de Jussieu, UMR 7586 du
CNRS, B\^atiment Sophie Germain, 75205 Paris Cedex 13, France}
\curraddr{School of Mathematics and Statistics,
University of Melbourne, Parkville VIC 3010, Australia}
\email{gufangz@unimelb.edu.au}

\author[C.~Zhong]{Changlong~Zhong}
\address{State University of New York at Albany, 1400 Washington Ave, ES 110, Albany, NY 12222}
\email{czhong@albany.edu}


\begin{document}

\begin{abstract}
We prove that  irreducible representations of the elliptic affine Hecke algebras  of Ginzburg, Kapranov, and Vasserot are in one-to-one correspondence with certain nilpotent Higgs bundles on the elliptic curve. 
The main tool we use is the equivariant elliptic cohomology  of the Steinberg variety of the Springer resolution.
As a by-product, we study representations at roots of unity in type-$A$. As another by-product, we define a version of elliptic Demazure-Lusztig operators with dynamical parameters that satisfy the braid relations. We discuss speculative indications of this correspondence in 4d $\calN=2$ gauge theory.
\end{abstract}
\maketitle
\tableofcontents

\section{Introduction}
The elliptic affine Hecke algebra  is defined  by Ginzburg, Kapranov, and Vasserot in \cite{GKV97}. Historically the study of this algebra and related objects (e.g., elliptic quantum groups) motivated the definition of equivariant elliptic cohomology of Grojnowski \cite{Gr},  Ginzburg, Kapranov, and Vasserot \cite{GKV95}, which in turn served as inspirations to fruitful endeavours in topology. To give a, far from being  complete, list of references that we are following, see \cite{And00,And03,Chen10,Gep,Gan}. This construction has been studied by \cite{Lur} in the framework of derived algebraic geometry.

In recent years the study of elliptic quantum groups has regained increasing attention from algebraic, geometric, and mathematico-physical aspects, see, e.g., \cite{AO, Konno, FRV}. In view of these recent developments, it is desirable to further investigate representations of  elliptic affine Hecke algebras, which is the main goal of the present paper.

In \cite{Gr} a construction of equivariant elliptic cohomology is given, and in \cite{GKV95}, an axiomatic definition of equivariant elliptic cohomology is given, together with a list of properties, and an example of a classical elliptic Schur algebra. In \cite{GKV97}, an elliptic affine Hecke algebra is defined algebraically. 

The goal of the present paper is not to give a new construction of an equivariant elliptic cohomology, nor another construction of an elliptic algebra. Instead, the paper contains 5 parts listed below.
\begin{enumerate}
\item Since Lurie's construction is not well-known in the representation theory community, we briefly review the account of Lurie on equivariant elliptic cohomology and along the way verify that it satisfies the properties on stalks expected in \cite{Gr} and some other properties axiomatized in \cite{GKV95}. No originality is claimed here.
\item We verify that the geometrically defined convolution algebra using equivariant elliptic cohomology of \cite{GKV95} is isomorphic to the algebraically defined algebra in \cite{GKV97}.
This isomorphism may have been expected by experts, and may have served as the motivation of the definition in \cite{GKV97}.
As no statements or proofs can be found in literature, even as a conjecture, this is a necessary step  for further works.
\item  Using (1) and (2) above, we prove that the irreducible representations of this elliptic algebra are parametrized by a set of nilpotent Higgs bundles. We note that this is the analogue of the Deligne-Langlands correspondence in the elliptic setting. The proof goes in two steps, the first
step follows from a standard method of  \cite{KL}. This standard method involves proving an analogue of  Kazhdan-Lusztig non-vanishing theorem, which is an intricate calculation.
The second step is to identify what the first step produces with the geometry of Hitchin systems. 
To the best of the authors' knowledge,  this fact has not been known  or expected  before the work of the present paper is done.
\item We discuss possible implications of the above correspondence. We raise a duality question  about a certain representation category of the elliptic affine Hecke algebra and a category of coherent sheaves of a Hitchin moduli space.
\item We study representations of the elliptic algebra at torsion parameters (roots of unity) in type-$A$, which is combinatorial in nature.
\end{enumerate}

From a topological and geometrical point of view, it is well-known that for a complex elliptic curve the principal bundles on it are parametrized by conjugacy classes of the loop group \cite{BG}, while semi-stable bundles are related to the elliptic cohomology of a point. The points (2) and (3) above can be considered as representation theoretical and elliptic cohomological parametrizations of semi-stable Higgs bundles on an elliptic curve, previously unknown.

Algebraically, the residue construction \cite{GKV97} gives a sheaf of algebras which is only locally free and does not contain known global sections satisfying the braid relations. In this paper
we define a version of elliptic Hecke algebra with dynamical parameters that is globally free as a sheaf of algebras, and prove that the Demazure-Lusztig operators satisfy the braid relations based on some algebraic properties of the recent construction of elliptic classes \cite{RW}.

Now we briefly summarize the main results of the present paper.
\subsection{Elliptic cohomology and convolution construction of the elliptic affine Hecke algebra} Part of the elliptic cohomological construction  of the elliptic affine Hecke algebra 
is done under the assumption of existence of integral equivariant elliptic cohomology \ref{assum:G-equiv}. However,
studying the representation theory of the elliptic affine Hecke algebra only involves equivariant elliptic cohomology for an compact abelian Lie group where the existence is proven by Lurie. We discuss in depth of the existence of the cohomological construction in \S~\ref{subsec:existence}.

Let $G$ be a compact Lie group and $E$ be an elliptic curve over a commutative ring $R$. Let $\catA_G$ be the moduli space of semi-simple, semi-stable, degree-0 principal $G^{\alg}$ bundles on the dual curve $E^\vee$, where $G^{\alg}$ is the corresponding split algebraic group of $G$ over $R$. If $G$ is connected with maximal torus $T$ and Weyl group $W$, then $\catA_G\cong \catA_T/W$. For any  $G$-space $X$, the $G$-equivariant elliptic cohomology of $X$, denoted by $\calE^*_G(X)$, is a  sheaf of $\bbZ$-graded algebras on $\catA_G$.

Let $Z$ be the Steinberg variety in the Springer resolution \S~\ref{subsec:DL_Conv}. In \cite{GKV95}, there is a convolution product constructed  on the sheaf $\Xi_{G\times S^1}(Z)$, which is certain twisted version of $\calE^0_{G\times S^1}(Z)$ on $\catA_{G\times S^1}$. 

%Let $\calB$ be the complete flag variety. Let $\calN$ be the nil-cone of the complexified Lie algebra of $G$, and let $\widetilde\calN$ be $T^*\calB$. There is a natural map  $\widetilde\calN\to \calN$ which is a resolution of singularity, called the Springer resolution.  There are natural actions of $G^{\bbC}\times\bbC^*$ on $\calN$ and $\widetilde\calN$, that make the Springer resolution equivariant. The fiber product $Z=\widetilde{\calN}\times_{\calN}\widetilde{\calN}$ is  called the Steinberg variety.  
On the other hand,  there is an elliptic affine Hecke algebra $\calH$ introduced in \cite{GKV97} associated to the root system of $G$ is also a sheaf of algebras on $\catA_{G\times S^1}$ with local sections subject to certain  residue vanishing conditions, which we recall in  Section~\ref{sec:ell}.  In \S~\ref{subsec:dyn} we define a new version of elliptic Hecke algebra with dynamical parameters. It satisfies the same residue-vanishing conditions. However, this version has global sections which satisfy braid relations.
\begin{thm}[Theorem~\ref{thm:main}]\label{thm:Intro_Main}
Assume $G$ is simple and simply-connected. There is  an isomorphism $\Upsilon:\calH\cong \Xi_{G\times S^1}(Z)$ of sheaves of algebras. 
\end{thm}
The proof we give is similar to that in \cite{CG} in the $K$-theory case, but there are new phenomena in the elliptic cohomology. Moreover,   via this isomorphism, we also determine explicit cohomology classes in $\Xi_{G\times S^1}(Z)$ that correspond to the elliptic Demazure-Lusztig operators in \cite{GKV97},

Note that there is  another construction, different from but somehow parallel to the one given by \cite{GKV97}, which provides a Hecke-type algebra associated to a root datum and  a one-dimensional  formal group law (see \cite{HMSZ, Zho, ZZ14}). In \cite{ZZ14}, it is shown that these formal affine Hecke algebras are isomorphic to the Borel equivariant cohomology of Steinberg varieties, and when the formal group law comes from an elliptic curve, they are  completions of the stalk of the algebra $\calH$ at the zero point (See Remark~\ref{Rmk:Formal}).

\subsection{Representations of the elliptic affine Hecke algebra over  complex numbers}
As the first application of Theorem~\ref{thm:Intro_Main}, we give a Deligne-Langlands type classification of irreducible representations of $\calH$. 
%Similar to the story of the classical affine Hecke algebra studied by  Kazhdan-Lusztig \cite{KL} and Ginzburg \cite{Gin85}, representations of convolution algebras can be studied by using the decomposition theorem.
With minimal amount of notations introduced, we summarize the classification as follows. 

Suppose $E$ is a complex elliptic curve. 
For any $t\in E$, let $\calL_t$ be the line bundle $\calO(\{t\}-\{0\})$. A $\calL_t$-values Higgs bundle is a pair $(P,x)$ where $P$ is a semi-simple semi-stable degree-0 principle $G$-bundle, and $x$ is a section of $\ad(P)\otimes\calL_t$. We say $(P,x)$ is nilpotent if the image of $x$ lands in the nilpotent cone at each fiber. Let $\Aut(P,x)$ be the group of automorphisms of $(P,x)$, and $C(P,x):=\Aut(P,x)/\Aut^0(P,x)$ the component group. Let $\bfB_{P,x}$ be the variety parametrizing all the Borel structures on $(P,x)$. Clearly, $C(P,x)$ acts on $H^*(\bfB_{P,x};\bbC)$. 
\begin{thm}[Corollary~\ref{cor:classify} and  Theorem~\ref{thm:Higgs}]\label{thm:intro_irreps}
Assume $t\in E$ is a non-torsion point. Let $\calH_t$ be the restriction  of $\calH$ to the subvariety $(\catA_T/W)\times\{t\}\subseteq (\catA_T/W)\times E$. 
\begin{enumerate}
\item For any nilpotent $\calL_t^{-1}$-valued Higgs bundle on $E^\vee$, and any irreducible $C(P,x)$-module $\chi$ with non-zero multiplicity in $H^*(\bfB_{P,x};\bbC)$, the multiplicity space $H^*(\bfB_{P,x};\bbC)_\chi$ is a module over $\calH_t$, which, moreover, has a simple top.
\item The set of irreducible $\calH_t$-modules is in one-to-one correspondence with $Higgs^{nil}_{t}(E^\vee)$, the set of triples $(P,x,\chi)$ as in (1). 
\end{enumerate}
\end{thm}
In the proof of Theorem~\ref{thm:intro_irreps}, we identify the variety $\bfB_{P,x}$ with the fixed locus in certain Springer fiber under a torus. Hence, the set $Higgs^{nil}_{t}(E^\vee)$ also has a description in terms of Springer theory, which can be found in details in \S~\ref{subsec:irrep}. 
Similar as in \cite{KL}, the technical part in the proof of Theorem~\ref{thm:intro_irreps}(2) is a non-vanishing statement, which we prove in \S~\ref{subsec:nonvan}.
We also give a Deligne-Langlands-Lusztig type character formula in \S~\ref{subsec:irrep}.

The set $Higgs^{nil}_{t}(E^\vee)$ also shows up in the study of double affine Hecke algebra \cite{BEG}, and its relation  with the present paper is briefly discussed in \S~\ref{sec:discussion} and is to be further investigated in future publications. 
When $G$ is of type-$A$, this set is the global nilpotent cone in the moduli space of  meromorphic Higgs bundles, which also plays a role in 
Seiberg-Witten integrable system \cite{DW} and 
geometric Langlands correspondence in the ramified case \cite{KW}.
We discuss how Theorem~\ref{thm:intro_irreps} provides a piece of evidence of a duality (Question~\ref{conj:KoszulDual}), and discuss various indications of this duality in $4d$ $\calN=2$ gauge theory.

\subsection{Representations in type-A}
Representations of the affine Hecke algebra of type-A at roots of unity have been studied by Ariki \cite{Arik} and  Grojnowski \cite{Gr94}.
In particular, they prove that the Grothendieck group of affine Hecke algebra of type-A at an $l$-th root of unity is isomorphic to the negative half of the affine quantum group of type $A_{l}$ \cite[Proposition~4.3]{Arik}.
They also prove that under this isomorphism the Lusztig canonical basis corresponds to the dual of the classes of the simple objects in the category of modules over  affine Hecke algebras at roots of unity.
This isomorphism has a categorical version, studied in \cite{BK, KhoLau, Rouq08}.
More precisely, it is well-known that, based on the work of Khovanov, Lauda and Rouqier, a certain category of modules over the quiver Hecke algebra categorifies the negative half of the affine quantum group of type-A. On the other hand, it is shown in \cite[Proposition~3.18]{Rouq08} that the affine Hecke algebra at a root of unity is Morita equivalent to a suitable quiver Hecke algebra. Hence, one obtains a categorification of the Ariki-isomorphism.

Using Khovanov, Lauda and Rouqier's work on the quiver Hecke algebra, we easily obtain the elliptic counterpart of Ariki's theorem and its categorified version. This is carried out in \S~\ref{sec:torsion_points}. We summarize the results here.

Assume $q_1$, $q_2\in S^1$ are two torsion points of order $n_1$ and $n_2$ respectively, $d=\lcm\{n_1,n_2\}$ and $l=n_1\cdot n_2/d$. Fix an isomorphism $E\cong S^1\times S^1$ as Lie groups. Let $S_t\subset E$ be the subset consisting of $z\in E$  such that $z$ is of the form $(q_1^u,q_2^v)$ for $u,v\in \bbZ$, and let $S_t^n\subseteq E^n$ be its product.
Let $\calH_n$ be the elliptic affine Hecke algebra of $\GL_n$ (or equivalently of $U_n$).
Let $\Mod_{t}\calH_n$ be the subcategory of finite dimensional $\calH_n-$modules that are set theoretically supported on $S_t^n$ when considered as coherent sheaves on $E^n$. 

Let $\Gamma_{d.l}$ be disjoint union of $l$-copies of the cyclic quiver with $d$ vertices. Let $H_n(\Gamma_{d,l})$ be the quiver Hecke algebra of $\Gamma_{d.l}$, and $\Mod_0H_n(\Gamma_{d,l})$  be certain  category of modules spelled out in details in \S~\ref{subsec:quiverHecke}.

\begin{thm}[Theorem~\ref{thm:ellHecke_quiverHecke} and Corollary~\ref{cor:ellHeck_Quant}]
\begin{enumerate}
\item There is an equivalence of abelian categories
$\Mod_0H_n(\Gamma_{d,l})\cong\Mod_{t}\calH_n$.
\item This isomorphism induces an isomorphism $\big(U^-(\widehat{\fs\fl_d})\big)^{\otimes l}\cong \bigoplus_nK(\Mod_{t}\calH_n)^*$.
\end{enumerate}
\end{thm}
Here $\widehat{\fs\fl_d}$ is the affine Lie algebra of type $A_{d-1}$, and $U^-(\widehat{\fs\fl_d})$ is the negative part of its enveloping algebra. Under the isomorphism in (2), we identify explicit auto-functors on $\bigoplus_nK(\Mod_{t}\calH_n)$ that correspond to the Chevalley basis in $U^-(\widehat{\fs\fl_d})^{\otimes l}$. 

There is also a graded version of this theorem, in which  the enveloping algebra is replaced by half of the quantum group. 
Under  this isomorphism, the Lusztig canonical basis corresponds to the dual basis of the classes of simple objects.

\subsection*{Notations}
We summarize the conventions of push-forwards and pull-backs we use, for the convenience of the readers.

For any map of schemes $f:X\to Y$ and quasi-coherent sheaf $\calG$ on $Y$, we use $f^*\calG$ to denote the inverse-image of $\calG$, and $f^{-1}g\in H^0(X,f^*\calG)$, the pull-back section  of $g\in H^0(Y,\calG)$. 

Let $p:X\to Y$ be a morphism between two topological  $G$-spaces. 
Vector bundles are always complex vector bundles, unless otherwise specified.
For any $G$-vector bundle $V$ on $Y$, its pullback on $X$ is denoted by $p^*V$. 
Taking equivariant (elliptic)	 cohomology, we get a map between two sheaves of algebras on the moduli scheme $\catA_G$, denoted by $p^{\sharp}:\Ell^*_G(Y)\to \Ell_G^*(X)$. 
The projection of the relative spectrum $\underline\Spec_{\catA_G}\Ell^0_G(X)\to \catA_G$ is denoted by $\pi_X^G$, or simply $\pi_X$ if $G$ is clear from the context.
The induced map on spectra by $p^{\sharp}$ is denoted by $p_\catA:\catA_G^X\to \catA_G^Y$. According to our convention above, the direct-image and inverse-image of quasi-coherent sheaves are denoted by $p_{\catA*}$ and $p_{\catA}^*$, respectively. 
If in addition, $p$ is proper, then there is a push-forward (Gysin map) in equivariant elliptic cohomology, $p_{\sharp}: \Theta(Tp)\to p_{\catA}^*\calO_{\catA_G^Y}$ of quasi-coherent sheaves on $\catA_G^X$, where $Tp$ is the normal bundle of $p$. By adjunction, we also have $p_{\catA*}\Theta(Tp)\to \calO_{\catA_G^Y}$, which is also  denoted by $p_\sharp$.

For any compact Lie group $G$ with maximal torus $T$, the natural projection $\catA_T\to \catA_G$ is denoted by $\pi$. For any map between compact Lie groups $\phi:H\to G$, the map induced on the moduli spaces is denoted by $\catA_\phi:\catA_H\to \catA_G$.

\subsection*{Acknowledgements}
Most of the ideas in this paper are rooted in \cite{GKV95}. Theorem~\ref{thm:Intro_Main}, which is the main theorem in our \S~\ref{sec:ellHecke} and \ref{sec:iso}, could in principle have appeared in {\it loc. cit.}. We would like to express our gratitude to  Eric Vasserot for numerous times of discussions and for sharing his insights into Section~\ref{sec:torsion_points}. We thank Marc Levine, Ivan Mirkovi\'c, Valerio Toledano Laredo, and Yaping Yang for helpful discussions and correspondences. We thank Eyal Markman for helps with  the proof of Theorem~\ref{thm:Higgs}. We are grateful to an anonymous referee for carefully reading the manuscript and making very helpful suggestions. During the preparation of this paper, the first named author was financially supported by Fondation Sciences Math\'ematiques de Paris and  Centre National de la Recherche Scientifique; The second named author was supported by Vladimir Chernousov and Stefan Gille, and also by PIMS. This project was conceived when both authors were visiting the Max-Planck-Institut f\"ur Mathematik in 2014. During the revision of the paper, the first named author was supported  by the Australian Research Council
via the award DE190101222.

\section{The elliptic group algebra}\label{sec:grp_alg}
In this section, we fix terminology and notations about line bundles on abelian varieties and theta-functions. From now until Section~\ref{sec:iso}, unless otherwise mentioned,  we always work over a commutative Noetherian ring $R$. 

\subsection{The theta functions}\label{subsec:theta}
Let $E$ be an elliptic curve over  a ring $R$. Let $0:\Spec R\to E$ be the zero section of the elliptic curve. 
Then the image of $0$ is a codimension one subvariety, denoted by $\{0\}$. Let $\calO(-\{0\})$ be the sheaf of ideals of this subvariety, which is a line bundle. Its dual $\calO(\{0\})=\calO(-\{0\})^{-1}$, as the line bundle associated to an effective Weil divisor, has a natural section $\vartheta$ dual to the embedding $\calO(-\{0\})\to  \calO_E	$. 
Let $inv:E\to E$ be the map sending any point to its additive inverse. Then $inv^*\calO(-\{0\})\cong\calO(-\{0\})$, and  $\vartheta$ is mapped to $-\vartheta$ via this  isomorphism. In this sense, we say that $\vartheta$ is an odd function.
Recall also that the natural section of an effective Weil divisor $\calO_E\to  \calO(\{0\})$ has the property that its derivative, which is a section of  $ \calO(\{0\})\otimes\Omega_E$, when restricted to the divisor itself $\{0\}$ becomes the identity under the natural isomorphism $\calO(\{0\})|_{\{0\}}\cong \Omega_E^\vee|_{\{0\}}$.

\begin{example}\label{ex:theta}
Let $E$ be the Tate elliptic curve over $R=\bbC(\!(q)\!)$, whose equation is given by \[y^2 +xy=x^3 +a_4(q)x+a_6(q)\] with 
$a_4(q) = -5s_3(q)$ and 
$a_6(q) = -5s_3(q) + 7s_5(q)$, such that $s_k(q) =  \sum_{n\geq1}\frac{n^kq^n}{1-q^n}$.
Then there is  the Jacobi-theta function
\[\vartheta(u)=u^{\frac{1}{2}}\prod_{s>0}(1-q^su)\prod_{s\geq0}(1-q^su^{-1})\frac{1}{2\pi i}\prod_{s>0}(1-q^s)^{-2}.\]
When $E$ is a complex elliptic curve, i.e., $E=\bbC/(\bbZ+\bbZ\tau)$, $\tau\in \bbC$,  then our section $\vartheta$ above can be chosen to be the Jacobi-theta function, evaluated at $u=e^{2\pi iz}$ and $q=e^{2\pi i\tau}$. Here $z$ is the coordinate of the universal cover $\bbC$.
Indeed, in both cases, the formula of theta function are written on a cover of the elliptic curve where the covering space is the projection $\bbC^*\to \bbC^*/q^\bbZ$. For the complex elliptic curve this can be made precise complex analytically, and for Tate curve this can be made sense in either rigid analytic geometry or toric geometry \cite[\S4.3]{Lur}. 
The line bundle  $\calO(\{0\})$ becomes trivial when lifted to the cover $\bbC^*$. Denote the lifting by $\tilde{L}\to \bbC^*$, which is endowed with a lift of the section $\tilde{s}:\bbC^*\to \tilde{L}$.
The trivialization 
\[\xymatrix{
\tilde{L}\ar@/^/[dr]\ar[rr]^{\cong}_{\phi}&&\bbC^*\times \bbC\ar[dl]\\
&\bbC^*\ar@/^/[ul]^{\tilde{s}}&
}\]
 can be uniquely fixed by the property that $\phi$ commutes with multiplication by $q^\bbZ$ and that the derivative of $\phi\circ\tilde{s}$ is 1 at $1\in \bbC^*$ \cite[p.38]{Sie}.
The above formula for the theta function are written so that it is a function on $\bbC^*$, that vanishes of order 1 at $q^\bbZ$, non-zero everywhere else, and the derivative at $1\in \bbC^*$ is equal to 1. 
\end{example}

Throughout this section, $T$ is a compact connected abelian (real) Lie group of rank $n$, i.e., $T\cong (S^1)^n$.  Let $\bbX^*(T)$ be the group of characters of $T$, and $\bbX_*(T)$ be its dual. Let $T^{\alg}\cong (\Gm)^n$ be the split algebraic torus over $R$ which classifies maps of abelian groups from $\bbX^*(T)$ to $\Gm$.  Let $\catA_T(E)$ be the $R$-scheme that classifies maps of abelian groups from $\bbX^*(T)$ to $E$, which is an abelian $R$-variety. When understood from the context, we simply write $\catA_T$. 
Equivalently, $\catA_T$ can be described as the moduli scheme of semi-simple, semi-stable, degree-0 principal $T^{\alg}$-bundles on the dual elliptic curve $E^\vee$.

We have a canonical isomorphism $\catA_T\cong E\otimes \bbX_*(T)$. 
Any character $\la\in \bbX^*(T)$ induces a group homomorphism $\chi_\lambda: \catA_T\to E$. 
The subvariety $\ker\chi_\lambda$ is a divisor of $\catA_T$, whose ideal sheaf is the line bundle $\calL_\la:=\calO(-\ker\chi_\lambda)\cong \chi_\lambda^*\calO(-\{0\})$ on $\catA_T$.
The natural section of the line bundle $\calO(-\ker\chi_\lambda)^{-1}$ is denoted by $\vartheta(\chi_\la)$, which   is equal to $\chi_\lambda^{-1}\vartheta$. 

%Fixing coordinates $T\cong (S^1)^n$, we get a basis for  $\bbX^*(T)$, say $\la_1,\cdots \la_n$.
The coordinates also induce an isomorphism  $\catA_T\cong E^n$.  
For any point $x\in \catA_T \cong E^n$, we can write $x_i\in E$ for its $i$-th coordinate, i.e., the image of $x$ via the $i$-th projection $E^n\to E$, $i=1,\dots,n$.
For any $\lambda=\sum_{i=1}^nn_i\la_i$, the morphism $\chi_\lambda:\catA_T \cong E^n\to E$ is given by  $x=(x_1,\dots,x_n)\mapsto \sum_{i=1}^nn_ix_i\in E$.
Then $\vartheta(\chi_\la)=\vartheta(\sum _{i=1}^nn_ix_i)$.

The map $\bbX^*(T)\to \Pic(\catA_T)$, $\lambda\mapsto\calL_\la$,  is not a group homomorphism. Instead, we have the following.
 
\begin{lemma}\label{lem:tensor_c1}
\begin{enumerate}
\item For any characters $\lambda_1$ and $\la_2$ of $T$, the following diagram commutes
 \[\xymatrix{
 \catA_T \ar[rr]^-{\chi_{\lambda_1}\times \chi_{\la_2}}\ar[drr]_-{\chi_{\lambda_1\otimes\la_2}}& &E\times E\ar[d]^{\mu}\\
 & &E
 }\]where $\mu:E\times E\to E$ is the addition of $E$. Let $AD$ be the subvariety of $E\times E$ consisting of $(x,-x)$ with $x\in E$, then $\calO(-\ker\chi_{\lambda_1\otimes \lambda_2})\cong (\chi_{\lambda_1}\times\chi_{\la_2})^*\calO_{E\times E}(-AD)$.
\item The line bundle $\calO(-\ker \chi_{\la^\vee})$ is canonically isomorphic to $\calO(-\ker \chi_\la)$. Via this isomorphism, $\vartheta(\chi_{\la^\vee})=-\vartheta(\chi_\la)$.
\end{enumerate}
\end{lemma}

Let $\Rep(T)$ be the free abelian group of virtual $T$-characters, i.e., $\Rep(T)\cong \ZZ[\bbX^*(T)]$.
% Following Lurie~\cite[\S~3.4]{Lur}, we define the $\Rep(T)$-graded sheaf of algebras on $\catA$. 
The set map $\chi:\bbX^*(T)\to \Pic(\catA_T)$ induces a homomorphism of abelian groups $\chi:\ZZ[\bbX^*(T)]\to \Pic(\catA_T)$, where $\ZZ[\bbX^*(T)]$ is considered as a free abelian group with the additive structure. The image is still denoted by $\calL_\la$ for $\la\in \bbZ[\bbX^*(T)]$.

%\begin{definition}\label{def:ext_S} Let $\hat\calS:=\bigoplus_{\lambda\in\ZZ[\bbX^*(T)]}\calL_\lambda$. Note that the group structure of $\Pic(\catA_T )$ makes $\hat\calS$ into  a sheaf of algebras on $\catA_T$.\end{definition}
%For any $\lambda\in \Pic(\catA_T)$, we write the corresponding line bundle on $\catA_T$ as $\calO(\lambda)$.
%In particular, for any $\lambda \in \ZZ[\bbX^*(T)]$, denote  \begin{equation}\label{def:ext_S}
%\calO(\chi_\lambda)=\calO(-\ker\chi_{\lambda}).\end{equation} The theta function $\varth(\chi_\la)$ is a section of $\calL_{\la}^{-1}$.

For any compact Lie group $G$ with a choice of maximal torus $T$, the Weyl group $W$ acts on $\bbX^*(T)$ and hence acts on $\catA_T$. That is, any group element $w\in W$ defines a map $w:\catA_T\to \catA_T$. Hence for any $\lambda\in \bbX^*(T)$ and any section $\sigma$ of $\calL_{\lambda}$ we have the pullback of line bundle  $w^*\calL_{\lambda}=\calL_{w^{-1}\lambda}$ and the pullback section $w^{-1}\sigma$. A choice of the set of simple roots determines a set of generators of $W$. For each simple root $\alpha$, the corresponding simple reflection is denoted by $s_\alpha$. A brief review of root systems will be given in \S~\ref{sec:ell}.

\subsection{Looijenga's ring of theta-functions}\label{subsec:Loo}
For any integer $n$, let $E^{(n)}$ be the $n$-th symmetric product $E^n/\fS_n$. 
The following lemma is clear.
\begin{lemma}\label{lem:divisors equal on E(n)}
Let $E^{(n-1)}\times E\to E^{(n)}\times E$ be the map given by the product of the symmetrization $E^{(n-1)}\times E\to E^{(n)}$ and the projection to the second factor $E^{(n-1)}\times E\to E$. The image of this map  is closed, and the ideal sheaf for the reduced subscheme structure is a divisor in $E^{(n)}\times E$, denoted by $D$. Moreover, letting $\iota:E^{(n)}\to E^{(n)}\times E$ be the embedding into the first factor, then 
\[\iota^*\calO(-D)\cong \calO(-C),\] where $C\subseteq E^{(n)}$ is the  divisor 
given by $\{\{x_1,\cdots,x_n\}\in E^{(n)}\mid x_i=0\hbox{ for some }i\}$.
\end{lemma}
\begin{proof}By definition, the map $E^{(n-1)}\times E\to E^{(n)}\times E$ is given by 
\[
(\{x_i\}, y)\mapsto (\{x_1,...,x_{n-1}, y\}, y).
\]
So the image coincides with the subset 
\[\{(\{x_i\}, y)\in E^{(n)}\times E~|~x_i=y \text{ for some }i\},\]
 which is closed of codimension 1. 

Now consider the Cartesian square
\[
\xymatrix{
Z\ar[d] \ar[r]&E^{(n-1)}\times E\ar[d]\\
E^{(n)}\ar[r]^\iota & E^{(n)}\times E.}
\]
Then by definition, the scheme $Z$ is given by 
\[
\{(\{x_i\}, \{y_j\}, z)\in E^{(n)}\times E^{(n-1)}\times E ~|~ \{x_i\}=\{y_j, z\} \text{ and }z=0\}. 
\]
So $Z\cong C$. 
\end{proof}

Recall that the general construction of symmetric powers of Deligne defines an isomorphism between $E^{(n)}$ and the Hilbert scheme of $n$-points on $E$ \cite[Theorem~6.3.9]{De}. In particular, the map $E^{(n)}\to S$ is smooth. This follows from a well-known fact that the Hilbert scheme of points on a smooth curve is smooth (see e.g., \cite[Proof of Proposition~6.3.9 b)]{De}), which in turn follows from the vanishing of obstruction \cite[XI 1]{SGA3}. 
Moreover, the line bundle $\calO(-\{0\})$ on $E$ defines a line bundle on  $E^{(n)}$. Explicitly, 
let $p_i:E^n\to E$ be the projection onto the $i$-th factor, and  $\pi:E^n\to E^{(n)}$ be the symmetrization map. The $\fS_n$-equivariance structure  on $\otimes_{i=1}^np_i^*\calO(-\{0\})$ defines the sheaf $\big(\pi_*(\otimes_{i=1}^np_i^*\calO(-\{0\}))\big)^{\fS_n}$ on $E^{(n)}$. This sheaf has been studied by Deligne  and is denoted as $TS_{ext}^n\calO(-\{0\})$ in \cite[(5.5.7)]{De}. A local calculation shows that $TS_{ext}^n$ transforms locally free sheaves to locally free sheaves \cite[(5.5.2.4)]{De}, the rank of which can be calculated from a fiber \cite[(5.5.8.1)]{De}. In particular,  $\big(\pi_*(\otimes_{i=1}^np_i^*\calO(-\{0\}))\big)^{\fS_n}$ is a line bundle. By definition, the section $\vartheta(x)$ of $\calO(-\{0\})^{-1}$ induces the section $\prod_{i=1}^n\vartheta(x_i)$ of $\pi_*\otimes_{i=1}^np_i^*\calO(-\{0\})^{-1}$, the vanishing divisor of which is $C$. As an effective divisor,  $\calO(-C)^{-1}$ has a natural section denoted by $\vartheta^{ U_n}$. Hence, there is an isomorphism $\calO(-C)^{-1}\cong \big(\pi_*\otimes_{i=1}^np_i^*\calO(-\{0\})^{-1}\big)^{\fS_n}$  under which the section $\vartheta^{ U_n}$ is identified with the section $\prod_{i=1}^n\vartheta(x_i)$.

Let $\rho:T\to  U_r$ be a representation of $T$. Up to conjugation the map lands in the maximal torus, and hence determines $r$ characters of $T$, denoted by $\lambda_1,\dots,\lambda_r$.  Let $\chi_\rho:E^n\to E^r$ be the product of the $\chi_{\lambda_i}$, and the composition with $E^r\to E^{(r)}$   is still  denoted by $\chi_\rho$. Let $\calL_\rho$  be the pull-back of the line bundle $\calO(-C)$ on $E^{(r)}$ via the map $\chi_\rho$, and  $\vartheta(\chi_\rho)$  be the section of $\calL_\rho^{-1}$ on $E^n$ which is the pull-back via $\chi_\rho$ of the section $\vartheta^{ U_r}$ of $\calO(-C)^{-1}$ on $E^{(r)}$.

The following lemma is an easy consequence of Lemma~\ref{lem:tensor_c1}. See also \cite[\S~1.8]{GKV95}. 
\begin{lemma}\label{lem: theta_add_tensor}
Let $\rho_1:T\to  U_{r_1}$ and $\rho_2:T\to  U_{r_2}$ be  two representations of $T$.
Let $\oplus:E^{(r_1)}\times E^{(r_2)}\to E^{(r_1+r_2)}$ be the symmetrization map, and let $\otimes:E^{(r_1)}\times E^{(r_2)}\to E^{(r_1r_2)}$ be the map given by $(\{x_1,\dots,x_{r_1}\},\{y_1,\dots,y_{r_2}\})\mapsto \{x_i+y_j\}$.
Then the following diagrams commute
\[\xymatrix{
E^n\ar[rr]^-{\chi_{\rho_1}\times\chi_{\rho_2}}\ar[drr]_-{\chi_{\rho_1\oplus\rho_2}}& & E^{(r_1)}\times E^{(r_2)}\ar[d]^{\oplus}\\
& & E^{(r_1+r_2)}
}\]
and
\[\xymatrix{
E^n\ar[rr]^-{\chi_{\rho_1}\times\chi_{\rho_2}}\ar[drr]_-{\chi_{\rho_1\otimes\rho_2}} &  &E^{(r_1)}\times E^{(r_2)}\ar[d]^{\otimes}\\
 &  &E^{(r_1r_2)}.
}
\]
Moreover, $\vartheta(\chi_{\rho_1\oplus\rho_2})=\vartheta(\chi_{\rho_1})\otimes\vartheta(\chi_{\rho_2})$ as sections of $\calL_{\rho_1\oplus\rho_2}^{-1}\cong \calL_{\rho_1}^{-1}\otimes\calL_{\rho_2}^{-1}$.
\end{lemma}

%The following is a direct consequence of Definition~\ref{def:ext_S} and Lemma~\ref{lem: theta_add_tensor}. \begin{corollary} For any $T$-representation $\rho:T\to  U_r$, the theta function $ \vartheta(\chi_\rho)$ is a section of the sheaf $\hat\calS$ on  $\catA_T$. \end{corollary} \begin{proof} For each character $\lambda$ of $T$, we know $\vartheta(\chi_\lambda)$ is a section of a direct summand of $\hat\calS$. Any representation $\rho:T\to  U_r$ decomposes into  a direct sum of characters of $T$. By Lemma~\ref{lem: theta_add_tensor}, $\vartheta(\chi_\rho)$ is the tensor of the theta-functions associated to these characters. Hence, $\vartheta(\chi_\rho)$ is a section of $\hat\calS$. \end{proof}

\section{Lurie's approach to $T$-equivariant cohomology}\label{sec:Lurie}
This section is a very brief reminder of Lurie's survey \cite{Lur}, adjusted to our purpose. We recall Lurie's existence theorem. We show that the Atiyah-Bott-Berlin-Vergne localization theorem and the Atiyah-Segal completion theorem are formal consequences of Lurie's theorem. These are folklore facts which are not in the literature. We also spell out the definition of Chern character in Lurie's theory, a version of which will be used  later in this paper.

\subsection{Derived schemes and equivariant cohomology theory}
\label{subsec:dersch}
Recall that an $E_\infty$-ring spectrum $A$ is a refinement of
a multiplicative cohomology. That is, it is a spectrum endowed with certain multiplicative structure. For any topological space $X$, we consider the infinite suspension of $X$ (disjoint union with a base point) as a spectrum. 
In the homotopy category of spectra, there is a suspension functor denoted by $\Sigma$, and the set of morphisms are denoted by $[-,-]$, i.e., the set of homotopy classes of maps of spectra.   
The $E_\infty$-ring spectrum $A$ then defines a well-behaved contravariant functor from the category of topological
spaces to the category of graded commutative rings, sending any $X$ to $A^*(X):=\bigoplus_{i\in\bbZ}[X,\Sigma^iA]$.  
Such a multiplicative cohomology also has relative and reduced versions $A^*(X,Y)$ and $A^*(X,x_0)$ for  pairs $(X,Y)$ and pointed spaces $(X,x_0)$. For simplicity, we denote $A^*(X,x_0)$ by $\tilde{A}(X)$. 
Recall that every commutative
differential graded algebra defines via the Eilenberg-MacLane construction an $E_\infty$-
ring spectrum,  and when working rationally, every $E_\infty$-ring spectrum arises essentially
in this manner. We refer to  \cite[\S~2.1]{Lur} for an introduction and to \cite[Section 7.1]{HighAlg} for a full-blown treatment.

\begin{definition}
We call an $E_\infty$-ring spectrum $A$   even if  $A^{2i-1}(\pt)=0$ for any integer $i$, and  periodic if it is endowed with an invertible element $p\in \pi_2A$.
\end{definition}

If $A$ is periodic, then for any topological space $X$, multiplication by $p\in \pi_2A$ induces isomorphisms $A^i(X)\cong A^{i+2}(X)$ for all $i$.

Let $A$ be an $E_\infty$-ring spectrum. Then $\pi_0A$ is a commutative ring and $\pi_n(A)$ is a $\pi_0A$-module. The  Zariski spectrum  $\Spec A$ of $A$ is a pair $(\Spec(\pi_0A), \calO_{\Spec A})$, where $\Spec(\pi_0 A)$ is the classical scheme, and $\calO_{\Spec A}$ is a certain structure sheaf of $E_\infty$-ring spectra. As commented in \cite[\S~2.2]{Lur}, there is an $\infty$-category of $E_\infty$-ring spectra, hence there is a well-defined notion of sheaves of $E_\infty$-ring spectra. A  derived scheme $X^{der}$ is a topological space endowed with a sheaf of $E_\infty$-ring spectra $\calO_{X^{der}}$, which is locally isomorphic to the Zariski spectrum $\Spec A$ of some $E_\infty$-ring spectrum $A$ (see \cite[\S~2.2]{Lur}). The underlying classical scheme of $X^{der}$ is denoted by $X$. A morphism of derived schemes $f:X^{der}\to Y^{der}$ is a morphism of schemes $X\to Y$ together with a morphism of sheaves of $E_\infty$-ring spectra $\calO_{Y^{der}}\to f_*\calO_{X^{der}}$. 
Most of the notions from scheme theory are defined in the derived setting as well, such as  fiber product, flatness  of  morphisms, quasi-coherent sheaves and coherent sheaves. For example, a flat morphism locally is given by a map $A\to B$ of $E_\infty$-ring spectra such that $\pi_0B$ is a flat $\pi_0A$-module in the sense of commutative algebra, and that for each $n$ the induced map $\pi_n A\otimes_{\pi_0A}\pi_0M\to\pi_nM$ is an isomorphism. 
There is also a notion of derived Deligne-Mumford stack \cite[\S~1.4.4]{SAG}, of which an example is the classifying stack of oriented derived elliptic curves, to be briefly reviewed in \S~\ref{sec:G_equ_ell}.

If there is a given morphism of derived schemes $f:X^{der}\to Y^{der}$, we say that  $X^{der}$ is over $Y^{der}$, or $X^{der}$ is a $Y^{der}$-scheme. When $f$ is an affine morphism,  $f_*\calO_{X^{der}}$ as a quasi-coherent sheaf on $Y^{der}$ recovers $X^{der}$, hence in this case we simply write $\calO_{X^{der}}$ instead of $f_*\calO_{X^{der}}$. A map between two $Y^{der}$-schemes is a commutative diagram
 \[\begin{xymatrix}{
X^{der}\ar[dr]\ar[rr]&&X^{'der}\ar[dl]\\
&Y^{der}.&
}\end{xymatrix}\]
The collection of maps $\Hom_{Y^{der}}(X^{der},X^{'der})$ form a simplicial set (rather than a set) and hence has the homotopy type of a topological space. 
For an $E_\infty$-ring spectrum $A$, a commutative $A$-group is a flat $A$-scheme $\bbG^{der}$ endowed with the structure of an abelian group-object in the category of $A$-schemes. In particular, its topological functor of points, which a priori lands in the category of topological spaces, factors through the category of topological abelian groups.  For simplicity, we skip the superscript $der$ for $\bbG$ when understood from the context. Examples includes the additive group over certain $E_\infty$-rings, the multiplicative group, and a derived elliptic curve. 
For any $A$-scheme $S$, denote $\bbG^{der}(S)=\Hom_A(S,\bbG^{der})$.

\begin{definition}\cite[Definition 7.2.1]{Ell2}
For an $E_\infty$-ring spectrum $A$ and a commutative $A$-group $\bbG^{der}$, a  pre-orientation of  $\bbG^{der}$ is a pointed map of spaces  $S^2\to \bbG^{der}(A)$, where the base point of the target is the zero section.
\end{definition}
% This is the data of a map $S^2\to \bbG^{der}(A)$ which sends the base point of $S^2$ to  the zero section of $\bbG^{der}(A)$ \cite[\S~3]{Lur}.
The homotopy classes of  pre-orientations are classified by $\pi_2(\bbG^{der}(A))$.
From now on, we always assume $\bbG^{der}$ is pre-oriented. 

 In this section we consider the equivariant cohomology with respect a compact abelian Lie group $T$. Let $\catA^{der}_T$ be the derived $A$-scheme that classifies maps of abelian groups from $\bbX^*(T)$ to $\bbG$. We denote its underlying classical scheme  by $\catA_T$.  For example, when $T=(S^1)^n$, then $\catA^{der}_T\cong\bbG^n$. 
For any homomorphism $f:T'\to T$ between compact abelian Lie groups, we have a map $f^*:\bbX^*(T)\to\bbX^*(T')$ between the weight groups and hence a map $\catA_f:\catA^{der}_{T'}\to \catA^{der}_{T}$. If $f:T'\to T$ is a closed embedding, then $\catA_f$ is an embedding.  Following the convention of affine morphisms to simplify notations,  when there is no confusion, we write $\calO_{\catA^{der}_{T'}}$ for the sheaf  $\catA_{f*}\calO_{\catA_{T'}^{der}}$ over $\catA_{T}^{der}$, and think of $\catA_{T'}^{der}$ as a closed subscheme of $\catA_T^{der}$.

Let  $T$ be a compact abelian Lie group.
We say  $X$ is a finite $T$-space if the following property is satisfied: there is a stratification \[\emptyset=X_0\subseteq X_1\subseteq\cdots\subseteq X_n=X\] with $X_{i+1}=X_i\coprod_{(T/T_i)\times S^{k-1}}((T/T_i)\times D^k)$, where $T_i$ are some closed subgroups of $T$, and $D^k$ are $k$-disks with boundary $S^{k-1}$.
Here the attaching map  $(T/T_i)\times S^{k-1}\to X_i$ is $T$-equivariant. 
\begin{example}\label{ex:attaching_fixed_loci}
An example of such a map is that we have a subset $ T/T'\times \pt\subseteq X_i$ and when attaching the cell $T/T_i\times D$ to $X_i$ for some disk $D$ and $T_i\subsetneq T'$,  the map $(T/T_i)\times S^{k-1}\to X_i$ sends $S^{k-1}$ to $\pt$, and  the factor $T/T_i$ to $T/T'$ via the natural quotient map. Note that the $T'$-fixed  loci on $ T/T_i\times D$ is empty although the attaching map sends $(T/T_i)\times S^{k-1}$ to a $T'$-fixed locus in $X_i$.
\end{example}

\begin{remark}
We need to apply equivariant cohomology theory to the case when $X$ is a vector bundle on a compact algebraic variety. In this case, $X$ itself does not have a finite $T$-cell decomposition as above. We therefore consider $D(X)$ the associated disk bundle, that is, replacing each fiber of the vector bundle by the unit disk in the fiber. So  $D(X)$ and $X$ are homotopy equivalent. We make the convention that $\calF_T(X)$ is understood to be $\calF_T(D(X))$. 
\end{remark}
\begin{theorem}\cite[Theorem~3.2]{Lur}\label{thm:lur_able}
Let $A$ be an $E_\infty$-ring spectrum. Let $\bbG$ be a pre-oriented commutative $A$-group. Then, 
for each compact abelian Lie group $T$, there exists a contravariant functor $\calF_T$  from the category of finite $T$-spaces to the {$\infty$}-category of quasi-coherent sheaves on $\catA^{der}_T$, satisfying the following properties:
\begin{enumerate}
\item It maps $T$-equivariant homotopy equivalences to equivalences of quasi-coherent sheaves;
\item For any fixed $T$, the functor $\calF_T$ maps finite homotopy colimits of $T$-spaces to homotopy limits of quasi-coherent sheaves;
\item $\calF_T(\pt)=\calO_{\catA^{der}_T}$;
\item For any closed embedding of subgroups $\phi:T'\subseteq T$, and finite $T'$-space $X'$, define $X=(X'\times T)/T'$ with the induced $T$-action. Then $\calF_{T}(X)=\catA_{\phi*}\calF_{T'}(X')$ where $\catA_\phi$ is the natural embedding $\catA^{der}_{T'}\to \catA^{der}_{T}$ induced by $\phi$. 
\end{enumerate}
\end{theorem}
%\gufang{Is $\calF_T(X)$ flat over $A$ for any $X$? This should be the restatement that the equivariant elliptic cohomology is periodic. If so, $\pi_*\calF_T(X)$ is determined by $\pi_0\calF_T(X)$.}

We refer the readers to 
\cite[Proposition~3.1, Remark~3.15]{Lur} for the relation between a pre-orientation on $\bbG^{der}$ and the identification $\calF_T(T/T')= \calO^{der}_{\catA_{T'}}$ as sheaves on $\catA_{T}^{der}$. The latter will be used frequently.

One observes that from the proof of \cite[Theorem~3.2]{Lur},  $\calF_T(X)$ is automatically a sheaf of $E_\infty$-ring spectra  on $\catA^{der}_T$. We  define $\calF_T^*(X)=\bigoplus_{i\in \bbZ}\calF_T^i(X)$, with $\calF^i_T(X)=\pi_{-i}(\calF_T(X))$, then $\calF_T^*(X)$ is  a sheaf of graded commutative algebras on the (classical) scheme $\catA_T$. 

%Note that as $\calF_T(X)$ is a quasi-coherent sheaf on $\catA_T^{der}$, each $\pi_i(\calF_T(X))$ is a well-defined quasi-coherent sheaf on $\catA_T$ in the classical sense.

For a pointed finite $T$-space $X$, we also consider the reduced version $\widetilde{\calF}_T(X)$ which is defined to be the homotopy fiber of the map $\calF_T(X)\to \calF_T(\pt)$ induced by the inclusion of the base point $\pt\to X$. 

Let $\Omega$ be the sheaf of relative K\"ahler differentials on the underlying scheme $\bbG$ of $\bbG^{der}$ along the structure map $p:\bbG \to \Spec(\pi_0A)$, and let $\omega$ be  the pull-back $0^*\Omega$ along  the zero section $0:\Spec(\pi_0A)\to \bbG$. Following \cite[\S~3]{Lur}, we will identify $\omega$ with the $\pi_0A$-module of global sections of $\omega$.  
A pre-orientation of $\bbG^{der}$, considered as a base-point preserving map $S^2\to \bbG^{der}(A)$, gives rise to a  morphism of $\pi_0A$-modules $\beta:\omega\to\pi_2A$ \cite[p19]{Lur}. We say a pre-orientation is an {\it orientation} if $\bbG$ is smooth of relative dimension 1 over $\Spec\pi_0A$, and $\beta$ induces isomorphisms $\pi_nA\otimes_{\pi_0A}\omega\to \pi_{n+2}A$ for any $n$.
The flatness of $\bbG^{der}$ over $A$ then gives an isomorphism \[p^*\omega\otimes_{\pi_0A}\pi_0\calO_{\bbG^{der}}\cong \pi_2A\otimes_{\pi_0A}\pi_0\calO_{\bbG^{der}}\cong\pi_2\calO_{\bbG^{der}}.\]

%In particular, the non-equivariant theory $\calF_{\{1\}}$ is an oriented cohomology theory in the classical sense. \gufang{elaborate}

Let $X$ be a finite $T$-space, and let $V$ be a $T$-vector bundle on $X$.  Let $Th(V)$ be the Thom space, i.e., the quotient of the disk bundle $D(V)$ by the sphere bundle $S(V)$. The following is the Thom isomorphism theorem.
\begin{prop}\cite[Proposition~3.2]{Lur} \label{prop:equivThom}
\label{thm:Thom} Assume $\bbG$ is oriented.
Let $V$ be a finite dimensional unitary representation of $T$. Then 
\begin{enumerate}
\item the quasi-coherent sheaf $\widetilde{\calF}_T(Th(V))$ is a line bundle on $\catA_T^{der}$;
\item for any finite $T$-space $X$, the natural map $\widetilde{\calF}_T(Th(V))\otimes\calF_T(X)\to \widetilde{\calF}_T( Th(V)\wedge X_+)$ is an equivalence.
\end{enumerate}
\end{prop}

\begin{example}\cite[\S~3.4]{Lur}\label{ex:Thom_circ}
For simplicity, assume $A$ is even and periodic.
Let $T=S^1$, and let $V$ be the 1-dimensional representation on which $T$ acts by rotation. In this case, we have $\catA_T^{der}\cong \bbG$. Also $\widetilde\calF_T(Th(V))$ is the homotopy fiber of the map $\calF_T(DV)\to \calF_T(SV).$ Since $DV$ is contractible, $\widetilde\calF_T(Th(V))$ is the homotopy fiber of the following map of quasi-coherent sheaves  on $\catA_T^{der}$:
\[\calO_{\catA_T^{der}}\to \calF_T(SV),\]
with $\calF_T(SV)$  homotopy equivalent to the structure sheaf of the zero section of $\catA_T^{der}$. Therefore, we can identify  $\widetilde\calF_T(Th(V))$ with the ideal sheaf of the zero section of $\bbG$. Via this identification, the map $\widetilde\calF_T(Th(V))\to\calO_{\catA_T^{der}}$ corresponds to the natural section of the line bundle $\widetilde\calF_T(Th(V))^{-1}$ that vanishes of order one at the zero section of $\catA_T^{der}$.  

\end{example}

In the rest of \S~\ref{sec:Lurie}, we deduce some formal consequences of Theorem~\ref{thm:lur_able}.

\begin{lemma}\label{lem:change_of_group} Assume there is a pre-orientation on $\bbG$. 
Let $f:T'\to T$ be a surjective map of  compact abelian Lie groups. Let $X$ be a finite $T$-space, considered as a finite $T'$-space via the map $f$. Then, $\catA_f^*\wt\calF_{T}(X)\simeq\wt\calF_{T'}(X)$.
\end{lemma}
Note that in general the pullback of coherent sheaves $\catA_f^*$ is understood in the derived sense. In particular, it preserves finite homotopy limits. Although we will not explore this, in this case the map $\catA_f$ can be shown to be flat, hence,  the homotopy groups of  $\catA_f^*$ are equal to the classical (i.e., non-derived) pullbacks of homotopy groups. 

For a topological space $X$, denote by $X_+$ the disjoint union of $X$ and a point, making it a pointed topological space. If $X$ is a finite $T$-space, with the trivial action at the base point, $X_+$ is also a finite $T$-space.
\begin{proof}
Assume first that  $X=T/\Gamma\times D$ for a disc $D$ and some closed subgroup $\Gamma<T$. Let $\Gamma'=f^{-1}(\Gamma)$ be the  closed subgroup of $T'$. Then by Theorem \ref{thm:lur_able}.(4),  $\calF_{T}(X)\simeq \calO_{\catA_{\Gamma}^{der}}$ and $\calF_{T'}(X)\simeq\calO_{\catA_{\Gamma'}^{der}}$. By definition  of $\catA_f$, we also have $\catA_f^*\calO_{\catA^{der}_{\Gamma}}\simeq \calO_{\catA^{der}_{\Gamma'}}$. So the conclusion holds in this case. Now assume the conclusion holds for $X'$, and suppose $X$ is obtained from $X'$ by gluing a cell, i.e., for some $\Gamma<T$, we have
\[
X/X'\cong (T/\Gamma\times D^n)/(T/\Gamma\times S^{n-1})\cong T/\Gamma_+\wedge S^n. 
\]
Again, denote $\Gamma'=f^{-1}(\Gamma)$, then $T/\Gamma\cong T'/\Gamma'$. 
There is a cofiber sequence
\[
\xymatrix{\wt\calF_{T}(T/\Gamma_+\wedge S^n) \ar[r] & \wt\calF_{T}(X)\ar[r] & \wt\calF_{T}(X').}
\]
Apply  $\catA_f^*$ on this sequence. Since $\catA_f^*$ preserves finite homotopy limits, so we have a cofiber sequence
\[
\xymatrix{\catA_f^*\wt\calF_{T}(T/\Gamma_+\wedge S^n) \ar[r] & \catA_f^*\wt\calF_{T}(X)\ar[r] & \catA_f^*\wt\calF_{T}(X').}
\]
Comparing with the cofiber sequence
\[
\xymatrix{\wt\calF_{T'}(T'/\Gamma'_+\wedge S^n) \ar[r] & \wt\calF_{T'}(X)\ar[r] & \wt\calF_{T'}(X'),}
\]
we see that the induction hypothesis together with $\catA_f^*\wt\calF_T(T/\Gamma_+\wedge S^n)\cong \wt\calF_{T'}(T'/\Gamma'_+\wedge S^n)$ implies that $\catA_F^*\wt\calF_T(X)\cong \wt\calF_{T'}(X)$. The proof is finished. 

%\cha{$\catA_f^*(\calF_{T}(X))\simeq\calF_{T'}(X)$, and assume $Y$ is the homotopy cofiber of a glueing map $ T/\Gamma\times S\to X\coprod ( T/\Gamma\times D)$, then Theorem~\ref{thm:lur_able}.(2) implies that $\calF_{T}(Y)$ is the homotopy fiber of the map  $\calF_T(X\coprod(T/\Gamma\times D))\to \calF_T( T/\Gamma\times S)$, and similarly $\calF_{T'}(Y)$ is the homotopy fiber of the map  $\calF_{T'}(X\coprod (T'/\Gamma'\times D))\to \calF_{T'}( T'/\Gamma'\times S)$. Note that $T/\Gamma\cong T'/\Gamma'$, and  $\catA_f^*$ preserves homotopy fibers. Hence, by induction hypothesis,  $\catA_f^*(\calF_{T}(Y))\simeq\calF_{T'}(Y)$.}
\end{proof}

\subsection{Localization}
First, we show that it follows  from Lurie's results that the $T$-equivariant cohomology theory satisfies the Atiyah-Bott-Berline-Vergne localization theorem, and consequently  the stalks of the sheaf $\calF_T$ can be identified.  The proof presented here is kindly suggested to us by an anonymous  referee.

For any non-trivial character $\chi: T\to S^1$, let $T_\chi$ be the kernel of $\chi$. We call the subscheme  $\catA_{T_\chi}^{der}$ in $\catA^{der}_T$ the weight hyperplane corresponding to $\chi$. Note that such hyperplanes  could be disconnected. Moreover, for any closed subgroup $\Gamma<T$,  lifting any non-trivial character $\chi$ of $T/\Gamma$ to $T$ gives a non-trivial character of $T$, whose corresponding  weight hyperplane contains     $\catA^{der}_{\Gamma}$. Therefore, each $\catA_{\Gamma}^{der}$ is contained in some weight hyperplane. In particular,  a (more commonly used) weaker version of the following lemma can be phrased as $i_\Gamma^\sharp$ is an equivalence outside a finite union of weight hyperplanes.

\begin{lemma}\label{lem:localization}
Let $X$ be a finite $T$-space and $\Gamma< T$ be a closed subgroup. Denote the embedding $i_\Gamma:X^\Gamma\to X$.  Then there exist  finitely many closed subgroups $\Gamma_i<T, i\in I$ such that   $\Gamma\not<\Gamma_i$, and  the induced pull-back $i_\Gamma^\sharp:\widetilde\calF_T(X)\to \widetilde\calF_T(X^\Gamma)$ is an equivalence  on the open subset $(\cup_{i\in I}\catA_{\Gamma_i}^{der})^c$. 
\end{lemma}
From the proof below, the subgroups $\Gamma_i$ can be explicitly determined as those showing up in a cell-decomposition of $X$. 

\begin{proof}Let $\{\Gamma_i\}$ be the collection of isotropy groups of cells not containing $\Gamma$. By induction over the number of cells, with the empty space as base case, we can assume that the statement holds for a sub-$\Gamma$-complex $X'$ of $X$. Now suppose $X$ is obtained from $X'$ by gluing a cell, i.e., 
\[
X/X'\cong (T/\Delta\times D^n)/(T/\Delta\times S^{n-1})\cong T/\Delta_+\wedge S^n, 
\]
where $\Delta<T$. Consider the diagram of cofiber sequences:
\[
\xymatrix{\wt\calF_T(T/\Delta_+\wedge S^n)\ar[r] \ar[d]^g & \wt \calF_T(X)\ar[d]\ar[r] & \wt \calF_T(X')\ar[d]\\
\wt\calF_T((T/\Delta_+\wedge S^n)^\Gamma)\ar[r] & \wt \calF_T(X^\Gamma) \ar[r] & \wt \calF_T((X')^\Gamma). 
}
\]
The five lemma implies that it suffices to show the morphism  $g$ is an equivalence.   Note that $(T/\Delta_+\wedge S^n)^\Gamma$ equals $T/\Delta_+ \wedge S^n$ if $\Gamma\subset \Delta$  and equals a point otherwise. So it suffices to show that $g$ is an equivalence when $\Gamma\subset \Delta$. By Theorem \ref{thm:lur_able}.(4), the inclusion $j:\Delta\to T$ induces $\wt\calF_T(T/\Delta_+\wedge S^n)\cong \catA_{j*}\wt\calF_\Delta(S^n)$. So the morphism $g$ coincides with the morphism $\catA_{j*}\wt\calF_\Delta(S^n)\to \wt\calF_T(\pt)\cong 0$, which is clearly an equivalence away from the subscheme $\catA_\Delta^{der}$ of $\catA_T^{der}$. 
\end{proof}

For any $a\in \catA_T$, let 
\begin{equation}\label{eq:fixeda}T(a):=\bigcap_{a\in \catA_{T'}}T', \quad i_a: X^{T(a)}\to X.
\end{equation}
\begin{theorem} \label{thm:stalk}
Taking the stalks at the point $a\in \catA_T$, the map $i^\sharp_a$ induces an equivalence $\widetilde\calF_T(X)_a\to \widetilde\calF_T(X^{T(a)})_a$.
\end{theorem}
\begin{proof}From   Lemma \ref{lem:localization}, there exist closed subgroups $\Gamma_i,i\in I$ such that $\Gamma_i\not> T(a)$, and $i^\sharp_{a}$ is an equivalence on $U:=(\cup_{i\in I}\catA_{\Gamma_i})^c$. By definition of $T(a)$, for any closed subgroup $\Gamma<T$,  $a\in\catA_{\Gamma}\subseteq\catA_T$ if and only if  $T(a)< \Gamma$. Hence, for any $i\in I$,  $a\not\in \catA_{\Gamma_i}$. Therefore,  $a\in U$ and hence taking the stalks at $a$, $i^\sharp_{T(a)}$ is an equivalence.
\end{proof}
\begin{remark} As mentioned in the proof,  for $a$ to be contained in the hyperplane $\catA_{\Gamma_i}\subseteq\catA_T$, the subgroup $\Gamma_i$ has to contain $T(a)$. Therefore, if $a\in \catA_T$ is the zero section,  it belongs to $\catA_{\Gamma}$ for any closed subgroup $\Gamma<T$, then $T(a)=0$, which implies that $X=X^{T(a)}$. In the other extreme case of Theorem \ref{thm:stalk}, if for any $\Gamma_i$ appearing in a cell decomposition of $X$, $a\not\in \catA_{\Gamma_i}$,  then $a$  belongs to the open subset in  Lemma \ref{lem:localization} (with $\Gamma$ replaced by $T$). Therefore, the stalks of $\wt\calF_T(X)$ and $\wt\calF_T(X^{T})$ at $a$ are equivalent.
\end{remark}

As in \cite[\S~2.2]{Gan}, we can calculate the stalks at every point. Let  $a\in \catA_T$ be an $A$-point, i.e., it defines a map $a:\bbX^*(T)\to \bbG(A)$. For any derived  $A$-scheme $S$, the map $a$ extends to a map $a:\bbX^*(T)\to \bbG(S)$. Since $\bbG$ is an  $A$-group,  $\Hom(\bbX^*(T),\bbG(S))$ has the structure of a (topological) abelian group. In particular, translation by $a$ induces an automorphism $t_a$ of $\Hom(\bbX^*(T),\bbG(S))$, which in turn amounts to an automorphism $t_a$ of  $\catA^{der}_T$, still called the translation by $a$. For any finite $T$-space $X$, the translation $t_a$ induces a map  on stalks $\wt\calF_T(X)_a\to \wt\calF_T(X)_0$. 

\begin{corollary}\label{cor: stalks}
For any $a\in \catA_T$, we have $\wt\calF_T(X)_a\simeq \wt\calF_T(X^{T(a)})_0$.
\end{corollary}
\begin{proof}
It suffices to show that the map  \[\wt\calF_T(X^{T(a)})_a\to\wt \calF_T(X^{T(a)})_0\] induced by $t_a$ is an equivalence. Let $\phi:T\to T/T(a)$. 
We have $\wt\calF_T(X^{T(a)})\simeq \catA_{\phi}^*\wt\calF_{T/T(a)}(X^{T(a)})$ by Lemma~\ref{lem:change_of_group}. Naturally $\catA_{\phi}=\catA_{\phi}\circ t_a$, hence the statement follows.
\end{proof}

\subsection{Completion} In this subsection we assume $T$ is a torus. 
Let $\{E_T^n\mid n\in\bbN\}$ be the Borel construction of classifying spaces, that is, a system of finite $T$-spaces such that the $T$-actions are free, and each $E^n_T$ is contractible in $E^N_T$ for $N$ large enough. $E^n_{T+}$ is the space obtained by adding a base point to $E^n_T$.  The following theorem is a generalization of the Atiyah-Segal completion theorem. It is true for any periodic ring spectrum  $A$. For simplicity, we assume further that $A$ is even.

\begin{prop}\label{prop:completion}
Let $A$ be even and periodic, and $I$ be the sheaf of ideals on $\catA^{der}_T$ corresponding to the zero section $0\in \catA^{der}_T$. If $\wt\calF_T(X)$ is  coherent, then the natural map $\wt\calF^*_T(X)\to \wt\calF^*_T( X\wedge E_{T+}^n)$ induces an isomorphism \[\wt\calF^*_T(X)_I^\wedge\cong \underset{\leftarrow}\lim \wt\calF_T^*(X\wedge E_{T+}^n),\] where  the left hand side is the completion with respect to the $I$-adic topology.
\end{prop}
\begin{proof}
We follow the original proof in \cite[\S~3]{AS}. First we prove it in the special case when $T=S^1$.  Let $E^n_T=SV^n$ where $V^n$ is an complex vector space of dimension $n$ with $S^1$-action by rotation, and $SV^n$ is the unit sphere in $V^n$. Let $DV^n$ be the unit disk. Then there is a long exact sequence of coherent sheaves on $\catA_T:$ 
\[\cdots\to \wt\calF^i_T(X\wedge Th(V^n)_+)\to \wt\calF^i_T(X\wedge DV^n_+)\to \wt\calF^i_T(X\wedge SV^n_+)\to \wt\calF^{i+1}_T(X\wedge Th(V^n)_+)\to \cdots .\]
By  Proposition~\ref{thm:Thom} and Example~\ref{ex:Thom_circ}, we have $\wt\calF_T(X\wedge Th(V^n)_+)=\wt\calF_T(X)\otimes \calL^n$, and the second map in the sequence above can be identified with $\wt\calF^i_T(X)\otimes\calL^n\to \wt\calF^i_T(X)$ given by the natural section $\vartheta^{-n}$ of the line bundle $\calL^{-n}$ on $\catA_T^{der}$. Note that  the image of $\vartheta^{-1}:\calL\to \calO_{\catA_T^{der}}$ defines an ideal sheaf $I$, so the cokernel of $\vartheta^{-n}:\wt\calF^i_T(X)\otimes\calL^n\to \wt\calF^i_T(X)$ is isomorphic to  $\wt\calF^i_T(X)/(I^n)$. 

For any $k>0$, denote the kernel of $\vartheta^{-k}:\wt\calF^{i+1}_T(X)\otimes \calL^n\to \wt\calF^{i+1}\otimes \calL^{n-k}$ by $\wt\calF^{i+1}_{n,k}$. 
Then we have a short exact sequence
\[
0\to \wt\calF^i_T(X)/(I^n)\to \wt\calF^i_T(X\wedge SV^n_+)\to \wt\calF^{i+1}_{n,n}\to 0.
\] 
To get the conclusion, we need to show that, there exists a positive number $K$ and a map $\beta$ in the following commutative diagram
\begin{equation}
\label{eq:comp1}
\xymatrix{ 0\ar[r] & \wt\calF^i_T(X)/(I^{n+K})\ar[r]\ar[d] & \wt\calF^i_T(X\wedge SV^{n+K}_+)\ar[r]\ar[d] \ar@{-->}[dl]^{\beta} & \wt\calF^{i+1}_{n+K,n+K}\ar[d]^{\vartheta^{-K}}\ar[r] & 0\\
0\ar[r] & \wt\calF^i_T(X)/(I^{n})\ar[r] & \wt\calF^i_T(X\wedge SV^{n}_+)\ar[r] & \wt\calF^{i+1}_{n,n} \ar[r] & 0.}
\end{equation}

Note that by definition, there is a chain of subsheaves in $\wt\calF^{i+1}_T(X)\otimes \calL^n$:
\begin{equation}\label{eq:comp3}\cdots\subset\wt\calF^{i+1}_{n,k}\subset \wt\calF^{i+1}_{n,k+1}\subset \cdots.
\end{equation}

{\it Claim}: There exists an integer $K$ such that for  any $k\ge K$, we have $\wt\calF_{m,k}^{i+1}=\wt\calF_{m,K}^{i+1}$ for any $m\ge n$. 

Indeed, as  $\wt\calF^{i+1}_T(X)\otimes \calL^n$ is a coherent sheaf, the chain \eqref{eq:comp3} stabilizes by the Noetherian property.  That is,  there exists  $K$ so that for any $k\ge K$, we have  $\wt\calF^{i+1}_{n,k}= \wt\calF^{i+1}_{n,K}$.
Note that  tensoring with $\calL$ is an exact functor, hence $\wt\calF^{i+1}_{n+1,k}\subseteq \wt\calF^{i+1}_T(X)\otimes \calL^{n+1}$ is obtained by applying $\_\otimes\calL$ to $\wt\calF^{i+1}_{n,k}\subseteq \wt\calF^{i+1}_T(X)\otimes \calL^{n}$. In particular,  $\wt\calF^{i+1}_{n,k}=\wt\calF^{i+1}_{n,K}$  implies that  $\wt\calF^{i+1}_{n+1,k}=\wt\calF^{i+1}_{n+1,K}$. Inductively, we see that  for any $m\ge n$, we have  $\wt\calF_{m,k}^{i+1}=\wt\calF_{m,K}^{i+1}$ for $k\ge K$. This proves the claim.  

Consider the following commutative diagram:
\[
\xymatrix{\wt\calF^{i+1}_T(X)\otimes \calL^{n+K}\ar[rr]^-{\vartheta^{-n-K}}\ar[d]^{\vartheta^{-K}}& & \wt\calF_T^{i+1}(X)\\
\wt\calF_T^{i+1}(X)\otimes\calL^{n}\ar[urr]^{\vartheta^{-n}}.}
\]
We then have an exact sequence 
\begin{equation}\label{eq:comp2}
0\to \wt\calF^{i+1}_{n+K,K}\overset{\psi}\longrightarrow \wt\calF^{i+1}_{n+K,n+K}\overset{\vartheta^{-K}}\longrightarrow \wt\calF^{i+1}_{n,n}.
\end{equation}
By the claim above, the map $\psi$ in \eqref{eq:comp2} is an equality, which  implies that the map $\vartheta^{-K}$ in \eqref{eq:comp2} (which is also the map $\vartheta^{-K}$ in \eqref{eq:comp1}) is 0.  Therefore, the map $\beta$ exists.

By the induction argument in \cite[\S~3, Step~2]{AS}, the  completion theorem for general torus $T$ reduces to the case when $T=S^1$. This finishes the proof.
\end{proof}

\subsection{The Chern character}
Recall that an orientation on the $A$-group $\bbG$ makes the cohomology theory  defined by the spectrum $A$ into an oriented cohomology theory in the classical sense.  Equivalently, it implies that  for any proper map $f:X\to Y$ between smooth complex manifolds of relative dimension $d$, there is a push-forward
\[f_A:A^{*+d}(X)\to A^{*}(Y)\]
which is a morphism of $A^*(Y)$-modules. Moreover, there is an associated formal group law $F(u,v)\in A^*(\pt)[\![u,v]\!]$ determined by 
\[
c_1^A(\calL_1\otimes \calL_2)=F(c_1^A(\calL_1), c_1^A(\calL_2)), ~\text{where }\calL_1, \calL_2 \text{ are line bundles over } X,
\]
and  $c_1^A$ is the first Chern class in the cohomology theory  $A$. 

In the remaining part of this section, assume $A^*(\pt)$ is a $\bbQ$-algebra. Then the formal group law associated to $A$ has an exponential $\fl_A(t)\in A^*(\pt)[\![t]\!]$ that is characterized by $F(u,v)=\fl_A(\fl_A^{-1}(u)+\fl_A^{-1}(v))$.
This exponential induces a natural isomorphism between $A^*$ and $H^*$ as functors to the category of graded commutative rings, denoted by $\ch^{A}$, called the non-equivariant Chern character. It is determined by the following property: for any compact smooth manifold $X$, the ring homomorphism $\ch^{A}:A^*(X)\to H^*(X;A^*(\pt))$ sends $c_1^A(\calL)$ to $\fl_A(c_1^H(\calL))$ for any line bundle $\calL$ over $X$.

Moreover, there is a Riemann-Roch type theorem. For any line bundle $\calL$ on a smooth manifold $X$, define $Td_A(\calL)=\frac{c_1^H(\calL)}{\fl_A(c_1^H(\calL))}\in H^*(X;A^*(\pt))$, or more generally, for any rank-$n$ vector bundle $V$ on $X$ with Chern roots $\{x_1,\dots,x_n\}$, define $Td_A(V)=\prod_i\frac{x_i}{\fl_A(x_i)}$.
\begin{theorem} \label{thm:Riemann-Roch}
Assume that $A^*(pt)$ is a $\bbQ$-algebra.
For any proper map $f:X\to Y$ between smooth complex manifolds, let $Tf$ be the relative tangent bundle. We have, for any  $\al\in A^*(X)$,
\[\ch^A (f_A(\alpha))=f_H(\alpha\cdot Td_A(Tf)).\]
\end{theorem}
Here $Tf=f^*TY-TX$ which is a virtual vector bundle on $X$. Note that the Todd class of a virtual vector bundle is well-defined. In the present setting this theorem can be found, for example, in \cite{Dy}. We also note that an algebraic version of theorem is  \cite[Theorem 2.5.4]{P04}, whose proof largely carries in the topological case as well. 

For any $a\in \catA_T$, we define the localized Chern character at $a$, denoted by $\ch^A_a$, to be the following composition 
\[\xymatrix@C=1em{
\calF_T^*(X)_a\ar[r]_-{i^\sharp}^-{\cong}&\calF^*_T(X^{T(a)})_a\ar[r]^-{\cong}_-{t_a}&\calF^*_T(X^{T(a)})_0\ar[r]&\underset{\leftarrow}\lim \calF^*(X^{T(a)}\times_TE^n_T)\ar[r]^-{\cong}_-{\ch^A} &\underset{\leftarrow}\lim H^*(X^{T(a)}\times_TE^n_T;A^*(\pt))
}\]
where the third map induces an isomorphism on completion by Proposition \ref{prop:completion}.

Summarizing the discussions above, we have the following 
\begin{corollary} Assume that $A^*(\pt)$ is a $\bbQ$-algebra, and let $\ch^A_a:\calF^*_T(X)_a\to H^*_T(X^{T(a)};A^*(\pt))$ be the localized Chern character as above. 
Then it induces an isomorphism on the completions $\calF^*_T(X)^\wedge_a\cong H^*_{T}(X^{T(a)};A^*(\pt))^\wedge$ at the augmentation ideal of $H^*_T(\pt; A^*(\pt))$.
\end{corollary}

\section{Equivariant elliptic cohomology theory}
\label{sec:G_equ_ell}
In this section, we collect some basic notions, constructions and properties of equivariant elliptic cohomology from \cite{Lur} and \cite{GKV95}. We do not claim originality in this section. Nevertheless, in \S~\ref{subsec:Quillen}, we show a push-forward formula in equivariant elliptic cohomology which previously was not known in the present generality, although a form of it can be found in \cite{Gan}.

\subsection{Construction of oriented derived elliptic curves}\label{subsec:abelian}
Recall the following theorem of Lurie.

\begin{theorem}\label{prop:derived_Ell_Curv}\cite[Theorem~4.1]{Lur}
There is a derived Deligne-Mumford stack $\calM_{1,1}^{der}$, whose underlying classical Deligne-Mumford stack is $\calM_{1,1}$, such that for every $E_\infty$-ring spectrum $A$, there is a natural homotopy equivalence between $\Hom(\Spec A,\calM_{1,1}^{der})$ and the classifying space for the topological category of oriented elliptic curves over $A$.
\end{theorem}
That is to say, any map of derived stacks $S\to \calM_{1,1}^{der}$ gives rise to an oriented derived elliptic curve on $S$.
Although we do not use this, the uniqueness of the oriented derived elliptic curve on $S$ obtained this way can be made precise by saying that, in the case when $S$ is affine, the connected components of $\Hom(S,\calM_{1,1}^{der})$ are naturally in bijection to the isomorphism classes of oriented elliptic curves on $S$.

A more precise statement of the theorem as well as the details can be found in \cite[Proposition~7.2.10, Remark~7.3.2]{Ell2}, which are in turn partially based on \cite{AV}. The derived Deligne-Mumford stack  $\calM_{1,1}^{der}$ is denoted by $\calM_{Ell}^{or}$ in \cite{Ell2}.
In practice, this theorem gives many examples of derived elliptic curves over $E_\infty$-ring spectra or in general derived Deligne-Mumford stacks. We give two examples below which are relevant for our purposes. 

The first example is a reformulation of the theorem. Roughly speaking, this is the description of the universal family of oriented derived elliptic curves on $\calM_{1,1}^{der}$.
\begin{example}\label{ex:etale_family}
	Let $S=\Spec A$ be a classical Noetherian affine scheme, and $E\to S$ be a family of elliptic curves so that the map $\phi:S\to \calM_{1,1}$ is \'etale. Then, by a theorem due to Goerss, Hopkins, and Miller, reformulated as \cite[Theorem~1.1]{Lur}, 
the ring	$A$ naturally refines to an even
$E_\infty$-ring spectrum $A^{der}$ with $\pi_0A^{der} = A$. Let  $S^{der} = \Spec A^{der}$, then 
	 the map $\phi:S\to \calM_{1,1}$ comes from a map of derived stacks $S^{der}\to \calM_{1,1}^{der}$ by taking $\pi_0$. Therefore, the theorem above endows $S^{der}$ with an oriented derived elliptic curve $E^{der}\to S^{der}$. The underlying map of classical stacks $\phi$ is given by the classical elliptic curve $E\to S$, hence the underlying classical scheme of  $E^{der}$ is $E$. 
\end{example}
However, it is worth mentioning that the proof of Theorem~\ref{prop:derived_Ell_Curv} given by Lurie \cite{Lur} does not explicitly use the theorem of Goerss, Hopkins, and Miller.

The following example is related to the``dynamical parameters" in \S~\ref{subsec:dyn}. See also \cite[Example~7.1]{YZ}.
\begin{example}
	Let $\calM_{1,2}$ be the (open) Deligne-Mumford (classical) moduli stack of genus 1 curves with two marked points. Then, there is a universal family of elliptic curves on $\calM_{1,2}$ endowed with a non-trivial line bundle coming from the second marked point.
	Let $\bbE^{der}\to \calM_{1,1}^{der}$ be the universal oriented derived elliptic curve. Recall that the underlying classical stack is the universal classical elliptic curve $\bbE\to \calM_{1,1}$. The complement of the zero-section $e: \calM_{1,1}\to \bbE$, denoted by $\bbE_o$,  is the moduil stack of genus 1 curves with two marked points. Therefore, we refer to the derived stack $\bbE^{der}_o$, i.e., the underlying stack  $\bbE_o$ endowed with the sheaf of $E_\infty$-ring spectra obtained from $\bbE^{der}$,  simply as $\calM_{1,2}^{der}$. The natural map $\bbE^{der}_o\to  \calM_{1,1}^{der}$ is then a map of derived stacks. Therefore, the same argument as Example~\ref{ex:etale_family} then implies that there is an oriented derived  elliptic curve on $\calM_{1,2}^{der}$ whose underlying classical elliptic curve is the universal elliptic curve with a non-trivial line bundle. 
\end{example}

The following remark is kindly pointed out to us by an anonymous referee. 
\begin{remark}
Let $S$ be a derived scheme,  $E$ be a strict elliptic curve on $S$. This induces
 a map  $S\to \calM_{1,1}^s$ to the classifying stack of strict elliptic curves. By \cite[Proposition~7.2.5 and Proposition~ 7.2.10]{Ell2}, the natural map $\calM_{1,1}^{der}\to \calM_{1,1}^s$ is an affine morphism. Hence, fiber product gives a derived scheme $S^{der}$ with a natural affine morphism $S^{der}\to S$, together with an oriented elliptic curve $E^{der}$ on $S^{der}$.
Note however,  if $S\to \calM_{1,1}^s$ is not flat, $\pi_0$ of the map $S^{der}\to S$  does  necessarily induces an isomorphism between the underlying classical schemes.
For example, if $S$ is a
classical Noetherian affine scheme, by \cite[Proposition~1.5.6]{AV}, $E$ is isomorphic to the underlying abelian variety of a strict abelian variety. 
However, the fiber product $S^{der}$ can be shown to be empty. Indeed, a non-vacuous  example has to be obtained from a derived scheme $S$ with a non-zero map $\omega\otimes\pi_0(S)\to \pi_2(S)$.
\end{remark}

\subsection{Equivariant elliptic cohomology theory}
For any connected compact Lie group $G$, let $G^{\alg}$ be the corresponding split  reductive algebraic  group with maximal algebraic torus $T^{\alg}\supset T$. There is a derived scheme $\catA_G^{der}$ functorial in $G$ \cite[\S~5.1]{Lur}, so that  $\catA_T^{der}$ for an abelian $T$ agrees with the one from \S~\ref{subsec:abelian}, and for the product of groups $G\times H$ we have $\catA_{G\times H}^{der}\simeq \catA_G^{der}\times\catA_H^{der}$. For example, if $G=T$ is a torus of rank $n$, then $\catA_T^{der}=\bbE^n$. 
For connected $G$ the underlying  classical scheme $\catA_G$ is related to the moduli space of semi-simple, semi-stable,  degree 0 principal $G^{\alg}$-bundles on $E^\vee$. The precise relation is not clear from {\it loc. cit.}, nevertheless, 
for the constructions below, we require the following properties, which we discuss here. 
 \begin{enumerate}
 \item 
 In \S~\ref{subsec:Higgs}, we assume that for a complex elliptic curve $E$, $\catA_G$ is equal to the aforementioned moduli space. It is well-known that the latter is in turn isomorphic to $\catA_T/W$  \cite[pg.4]{FMW} and \cite[Theorem~4.16]{Laz}.
 \item For the Thom isomorphism theorem \S~\ref{subsec:Thom}, we assume that $\catA_{U_n}$ is isomorphic to $E^{(n)}$. In Example~\ref{ex:Un-bundle} below we give a detailed discussion about how the latter is related to vector bundles. 
\end{enumerate}
Here for a split reductive algebraic group $G$ with a maximal torus $T$,  a principal $G$-bundle is semi-simple if the structure group has a reduction to $T$. 
A principal $\GL_n$-bundle, or equivalently, a vector bundle, is semi-stable if it is slope-semi-stable, where the  slope is $\frac{\deg}{\rank}$. That is, there is no subbundle with greater slope. Recall that a vector bundle on $E^\vee$ is flat as a coherent sheaf over $R$, and hence both $\deg$ and $\rank$ can be defined fiber-wise and are independent of the choice of fiber \cite[Theorem~9.9]{Hart}. 
For a closed subgroup $G\subseteq \GL_n$, a principal $G$-bundle is semi-stable if the induced  vector bundle is slope-semi-stable.

More generally than (1) above, let $S$ be   a finite-type integral complex scheme, and let $E$ be a projective family of elliptic curves. Let $M$ be the  moduli functor sending any scheme $B\to S$ to the groupoid of $G$-bundles on $E\times_S B$ satisfying the semi-simplicity, degree-0, and semi-stable conditions. Then, there is a natural transformation from $M$ to $\catA_T/W$ sending the family to a reduction to $T$.   Then,  this natural transformation induces a bijection on complex points \cite[pg.4]{FMW} and \cite[Theorem~4.16]{Laz}. Hence  $\catA_T/W$ is the coarse moduli space in this sense. 

Note that in both cases above ($S$ is a finite-type integral complex scheme, or $S$ is a general Noetherian scheme and $G=U_n$), for a closed subgroup of equal rank $H\subseteq G$, the map $\catA_H\to \catA_G$ is flat.   Indeed, let $T\subseteq H\subseteq G$ be a maximal torus, and let $W_H$ and $W_G$ be the respective Weyl groups. They are both finite groups generated by reflections, and $W_H$ is a subgroup of $W_G$. The map $\catA_H\to \catA_G$ in both cases above is the projection $\catA_T/W_H\to \catA_T/W_G$. In the case of complex numbers, this is a finite surjective projective morphism, and hence to show that it is flat it suffices to show that for any field $k$ with $\Spec k\to \catA_T/W_G$, the fiber is a  length-$|W_G|/|W_H|$ finite $k$-scheme \cite[Theorem~II9.9]{Hart}.  For this,  it again suffices to show that for each $\Spec k\to S=\Spec R$ the base-change $(\catA_T/W_H)_k\to (\catA_T/W_G)_k$ is a flat map of this degree. 
It follows from the Chevalley-Shephard-Todd theorem that both quotients $(\catA_T/W_H)_k$ and  $(\catA_T/W_G)_k$ are smooth. In the case of $U_n$,  $W_G$ and $W_H$ are both products of symmetric groups, hence $(\catA_T/W_H)$ and  $(\catA_T/W_G)$ are both products of Hilbert schemes of points on a smooth curve and hence are smooth, similar as in the discussion after  Lemma~\ref{lem:divisors equal on E(n)}.
%Indeed,  locally elementary symmetric polynomials of a local uniformizer generate the ring of invariants and are algebraically independent similar as in \cite[Example 6.3.18]{De}.  
Now in both cases the flatness follows from the miracle flatness, and the degree can be calculated off diagonally.

\begin{example}\label{ex:Un-bundle}
As an illustration we show that for any elliptic curve $E$ over an Noetherian ring $R$,  for any $S\to \Spec R$, let 
 $\catA_{U_n}(S)$ be the set of isomorphism classes  of  semi-simple semi-stable degree-0 principal $G=\GL_n$-bundles. 
First note that the natural map $\catA_{T}(S)=E^n(S)\to \catA_{U_n}(S)$  given by induction from a $T$-bundle to a $G$-bundle factors through the quotient $E^{(n)}(S)$. 
This is because the action of $W=S_n$ is given by conjugation by $N_T(G)\subseteq G$. By definition of semi-simplicity, the map is 
a surjective map. 
We claim this map is also   injective. Without loss of generality, we may assume $S$ is irreducible.
We note that any $T$-bundle on $E_S$, the base-change of $E$ to $S$, is of the form $\calL_1\oplus\cdots\oplus\calL_n$, with semi-stability implies that each $\calL_i$ is of degree-0. Let $\calL$ be any degree-0 line bundle, a non-zero map $\calL\to \calL_1\oplus\cdots\oplus\calL_n$, when composed with a projection to $\calL_i$ is non-zero for some $i$. There is an  open dense subscheme of $S$ on which the cokernel is flat. By degree consideration, this cokernel is trivial and hence we have an isomorphism $\calL_i\cong \calL$. 
If $\calL_1'\oplus\cdots\oplus\calL_n'\to \calL_1\oplus\cdots\oplus\calL_n$ is an isomorphism of vector bundles, by induction on $n$ we obtain that on some open dense subscheme of $S$ each $\calL_i'$ is isomorphic to some $\calL_j$ and hence the two $T$-bundles are in the same $S_n$-orbit. By definition of an elliptic curve, the morphism $E\to \Spec R$ is separated and hence the diagonal $E^{(n)}\subseteq E^{(n)}\times E^{(n)}$ is closed. Therefore, the two map $S\to  E^{(n)}$ agreeing on an open dense subscheme have to agree on $S$.
\end{example}

We consider $\bbE$ and $A$ satisfying the following assumption.
\begin{assumption}\label{assum:G-equiv}
For each connected compact Lie group  $G$, there is a  contravariant functor $\Ell_G^*$ from the category of finite $G$-spaces to the category of quasi-coherent sheaves of graded  algebras on $\catA_G$, satisfying the following properties:
\begin{enumerate}
\item  $\Ell_G^*$ maps homotopy equivalences to isomorphisms of quasi-coherent sheaves.
\item $\Ell_G^*$ extends to a functor from the category of pairs of $G$-finite spaces to  the category of graded quasi-coherent sheaves.
\item $\Ell_G^*$ sends disjoint unions into products, and associated to  the mapping cone of any map, there is the usual long exact sequence of sheaves.
\item When $G=T$ is a compact abelian Lie group, then  $\Ell^*_T$ coincides with homotopy groups of $\Ell_T$ from Theorem~\ref{thm:lur_able}.
\item\label{prop:(3)} Let $\psi:H\to G$ be a group homomorphism. Let $\catA_\psi:\catA_H\to \catA_G$ be the induced morphism.
Then we have a natural transformation  $\catA_\psi^*(\calE_G^*(-))\longrightarrow \calE_H^*(-)$, which yields an isomorphism for a finite $G$-space $X$ if $\calE_G^*(X)$ is flat. 
\item\label{prop:ell_induction} For an embedding of compact Lie groups $\phi: H\inj G$, and  finite $H$-space $X$, the following composition 
\[\Ind_H^G:\calE_G^*((X\times G)/H)\to \catA_{\phi*}\calE_H^*((X\times G)/H)\to \catA_{\phi*}\calE_H^*(X)\]
is an isomorphism. Here the second map is induced by the embedding of  $H$-spaces $X\to (X\times G)/H, x\mapsto (x,e_G)$. 
\end{enumerate}
\end{assumption}

\begin{remark}\label{rmk:assump_ell_satisf}
\begin{enumerate} 
\item The existence of the functors $\Ell_G$, for  compact (non-abelian) Lie group $G$, from the category of finite $G$-spaces to the category of quasi-coherent sheaves on $\catA_G^{der}$ is announced in \cite[\S~5.1]{Lur}.  On the level of global sections, the functor preserves homotopy equivalences and maps  homotopy colimits of $G$-spaces to homotopy limits of quasi-coherent sheaves \cite[Proposition~3.3]{Lur}. It also satisfies property \eqref{prop:ell_induction}. 

\item It follows from Assumption~\ref{assum:G-equiv}\eqref{prop:(3)} that $\Ell_G^0(\pt)=\calO_{\catA_G}$. Indeed, one can consider the map $\psi:G\to \{e\}$, then $\Ell_G^0(\pt)=\catA_{\psi}^*(\Ell^0(\pt))=\catA_{\psi}^*(\calO_{S})=\calO_{\catA_G}$.

\item If $\pi_0A$ is a finitely generated $\bbC$-algebra, for any connected compact Lie group $G$ with torsion free fundamental group we have.  For any finite  $G$-space $X$, we  define $\Ell_G^*(X)$ to be the $W$-invariants in $\pi_*\Ell_T^*(X)$ where $\pi:\catA_T\to \catA_G=\catA_T/W$ is the quotient map. 
Recall that in this case any embedding of compact Lie groups of equal rank induces a flat map $\catA_\psi:\catA_H\to \catA_G$.
Then, the assignment sending a finite $G$-space $X$ to $\Ell_G^*(X)$ satisfies Assumption~\ref{assum:G-equiv}(1)-(5). When $X=\pt$, Assumption~\ref{assum:G-equiv}\eqref{prop:ell_induction} is shown in \cite[Theorem~4.6]{Gan}. If furthermore K\"unneth property is satisfied, then Assumption~\ref{assum:G-equiv}\eqref{prop:ell_induction} holds for any $X$ \cite[Proposition~5.1]{Gan}.
\end{enumerate}
\end{remark}

It is reasonable to expect that Assumption~\ref{assum:G-equiv} is satisfied for a large class of oriented derived elliptic curve, but  this statement is not in the literature. Proving this is beyond the scope of this paper.  However, in the setting where we apply equivariant elliptic cohomology outside of the abelian case \S~\ref{sec:repn},  $\pi_0A$ is $\bbC$-algebra and $G$ is either a product of simply-connected groups or a product of the unitary groups, so that $\pi_1$ is torsion free. Therefore, following Remark~\ref{rmk:assump_ell_satisf}(3),  Assumption~\ref{assum:G-equiv} is always satisfied. See \S~\ref{subsec:existence} below for a detailed remark about the existence of the integral form as well as the implications in representations studied in this paper.

As  shown in \cite{GKV95},  Assumption~\ref{assum:G-equiv}.\eqref{prop:ell_induction}, which is a property about induction, implies the following property of change of groups. (See also the proof of \cite[Proposition~5.1]{Gan}.)
\begin{corollary}\label{cor: ell_change_group}
Under Assumption~\ref{assum:G-equiv}, let $X$ be a finite $G$-space, and let $K$ be a normal subgroup of $G$ such that  $K$ acts on $X$  freely. Denote the quotient  $G\to G/K$ by $\phi$, then we have an isomorphism
\[{\catA}_{\phi*}\Ell_G^*(X)\cong \Ell_{G/K}^*(X/K).
\] 
\end{corollary}
\begin{proof}
By cellular induction, it suffices to assume that $X=G/H$ for some closed subgroup $H<G$. The assumption that $K$ acts freely on $X$ implies that the composition $\psi:H\to G\to G/K$ is an embedding. In particular, $H\cap K=\{1\}$. We have $(G/H)/K\cong (G/K)/H$. Hence,   by Assumption~\ref{assum:G-equiv}(\ref{prop:ell_induction}), we obtain \[\Ell_{G/K}^*((G/H)/K)\cong \Ell_{G/K}^*((G/K)/H)\cong {\catA}_{\psi*}\Ell_H^*(\pt).\] Applying Assumption~\ref{assum:G-equiv}(\ref{prop:ell_induction}) again to the embedding $H<G$ and the $H$-space $\pt$, we conclude the proof.
\end{proof}

Define \[\catA_G^X=\Spec_{\catA_G} \calE^0_G(X),\]
which is a scheme over $\catA_G$. We have  the structure morphism  
\[\pi_X^G:\catA_G^X\to\catA_G,
\]
which we   simply denote by $\pi_X$ if $G$ is understood from the context. By Assumption \ref{assum:G-equiv}.\eqref{prop:(3)}, if $\psi:H\to G$ is a group homomorphism and $X$ is a finite $G$-space, then it induces a morphism of schemes $\catA^X_\psi:\catA_H^X\to \catA_G^X$.

\subsection{The GKV-classifying maps}\label{subsec:GKV_classifying}
We recall the GKV-classifying map defined in \cite[(1.6)]{GKV95}. Let $X$ be a finite $G$-space. For any $G$-vector bundle $V$ on $X$ of rank $n$, let $Fr\to X$ be the associated frame bundle. The group $G\times U_n$ acts on $Fr$, so  $\catA^{Fr}_{G\times U_n}$ is a scheme over $\catA_{G\times U_n}$, which in turn is a scheme over $\catA_{U_n}\cong E^{(n)}$ (see Example~\ref{ex:Un-bundle}). By Corollary~\ref{cor: ell_change_group}, there is an isomorphism $\catA^{Fr}_{G\times U_n}\cong \catA_G^X$. These maps fit into the following commutative diagram
\[\xymatrix{
\catA^{Fr}_{G\times U_n}\ar[r]\ar[d]^{\cong}&\catA_{G\times U_n}\ar[r]\ar[d]&\catA_{U_n}\\
\catA^X_G\ar[r]
&\catA_G
}.\]
The induced map $c_V^G: \catA_G^X\to \catA_{U_n}$ is called the {\it GKV-classifying map}, which we simply denote by $c_V$ if $G$ is understood. 

As in \cite[(1.9)]{GKV95}, one can define the characteristic classes in equivariant elliptic cohomology as follows. Let $P<U_n$ be the parabolic subgroup such that $U_n/P=\PP^{n-1}$. Then the projectivization $\PP(V)$ is isomorphic to $Fr/P$.  By Corollary~\ref{cor: ell_change_group}, there is an isomorphism $\catA_{G\times P}^{Fr}\cong \catA_{G}^{\PP(V)}$, which is a scheme over $\catA_P\cong\catA_{U_{n-1}\times S^1}\cong E^{(n-1)}\times E$. 

\begin{lemma}\label{lem:GKV_classify}
We have a commutative diagram
\[\xymatrix{
\catA_{G}^{\PP(V)}\ar[r]^-{\cong}\ar[d]&\catA^{Fr}_{G\times P}\ar[r]\ar[d]&\catA_P\ar[d]\ar[r]^-{\cong}&E^{(n-1)}\times E\ar[d]\\
\catA_G^X\ar[r]^-{\cong}&\catA^{Fr}_{G\times U_n}\ar[r]^{c_V}&\catA_{U_n}\ar[r]^-{\cong}&E^{(n)},
}\]
where the right vertical map is the symmetrization defined in Section~\ref{sec:grp_alg}. Moreover, 
 the middle square is  Cartesian.
%\item the inclusion $\catA_{G}^{\PP(V)}\cong \catA_G^X\times_{E^{(n)}} (E^{(n-1)}\times E)\subseteq \catA_G^X\times E$ is the embedding of a divisor. \end{enumerate}
\end{lemma}
We denote the map $\catA_{G}^{\PP(V)}\to E^{(n-1)}\times E$  by $c_{\PP(V)}$, and  refer to it as the {\it GKV projective classifying map}.
\begin{proof}
This follows from Assumption~\ref{assum:G-equiv}\eqref{prop:(3)} and the fact that the map $\catA_{P}\to \catA_{U_n}$ is flat. 
%For Claim (2), note that the natural map $\catA_G^X\times_{E^{(n)}} (E^{(n-1)}\times E)\subseteq \catA_G^X\times E$ is the pull-back of. The later  is a well-defined embedding of a divisor, hence, so is the former.
\end{proof}
\begin{example}
For simplicity, we assume $R=k$ is an algebraically closed field. 
Let $G=\{1\}$, and  $X=\PP^1$ with the tautological line bundle $L=\calO(-0)$ on it, so that the total space, with the zero section removed, is isomorphic to $\Aff^2-\{0\}$. 
Below we find the GKV-classifying map $c_{L}:\catA_1^{\PP^1}\to \catA_{S^1}\cong E$, for different $S^1$-actions on $L=\calO(-0)$. 

Let $S^1$ act on $\Aff^2-\{0\}$ by scaling, so that $L$ is an $S^1$-line bundle with respect to the trivial $S^1$-action on $\PP^1$. 
By construction, $Fr=\Aff^2-\{0\}$ with the above mentioned $U_1=S^1$-action. 
We have $\catA_1^{\PP^1}\cong \catA_{S^1}^{\Aff^2-\{0\}}=\Spec_E(\Ell^0_{S^1}(\Aff^2-\{0\}))$, with the map $c_{L}$ being the structure map to $E$. By Corollary~\ref{cor: stalks}, the stalk of $\Ell^0_{S^1}(\Aff^2-\{0\})$ is zero away from the origin, and is 2-dimensional at the origin. Thank to the algebraic-closedness of $k$, as a coherent sheaf, it is either $k_0^{\oplus 2}$ or $\calO_{E,0}/m_0^2$ where $m_0$ is the maximal ideal of the local ring $\calO_{E,0}$. To determine which one it is, we consider the inclusion $\Aff^2-\{0\}\subseteq \Aff^2$ which is $S^1$-equivariant and hence induces a map $\calO_E\cong\Ell^0_{S^1}(\Aff^2) \to\Ell^0_{S^1}(\Aff^2-\{0\})$ which is surjective. Therefore, $\Ell^0_{S^1}(\Aff^2-\{0\})\cong\calO_{E,0}/m_0^2$. Taking $\Spec_E$, we obtain that the map $c_{L}$ identifies $\catA_1^{\PP^1}$ as the first order nilpotent thickening of $\catA_{S^1}\cong E$ at the origin.
\end{example}

\subsection{The Thom isomorphism theorem and Chern classes}\label{subsec:Thom}

Recall that the natural map $\oplus\times \id:E^{(n-1)}\times E\to E^{(n)}\times E$ is a divisor. We have the map $\catA_{U_n}\to\catA_{U_n}\times E, x\mapsto (x,0)$.
Let $\calO(-0)$ be the line bundle on $\catA_{U_n}$ which is the pullback of the ideal sheaf of the above divisor. 
\begin{definition} \cite[\S~2.1]{GKV95} Let $X$ be a finite $G$-space, and $V$ be a $G$-vector bundle. 
%By Lemma \ref{lem:GKV_classify}, the subset $\catA_{G}^{\PP(V)}\subset\catA_G^X\times E$ is a divisor. Let $\calO(-\catA_{G}^{\PP(V)})$ be its ideal sheaf.
Define $\Theta_G(V)$ to be the line bundle on $\catA_G^X$ that  is the   the pull-back of $\calO(-0)$  along the map $c_V$. If $G$ is understood, we simply abbreviate $\Theta_G(V)$ as  $\Theta(V)$. 
\end{definition}

\begin{example}
When $G=U_n$, $X=\pt$, and the vector bundle $V=\xi_n$ is the standard $n$-dimensional representation of $U_n$, then by Lemma~\ref{lem:divisors equal on E(n)}, $\Theta(\xi_n)=\calO(-C)$. Moreover, $\vartheta^{U_n}=\prod_{i=1}^n\vartheta(x_i)$ is the natural section of $\Theta(\xi_n)^{-1}$ as in \S~\ref{subsec:Loo}.
\end{example}

\begin{prop}\label{prop:GKVClass_general}
Let $X$ be a finite  $G$-space.
\begin{enumerate}
\item For any rank-$n$ $G$-vector bundle $V$ on $X$,  we have $\Theta_G(V)\cong c_V^*\Theta_{U_n}(\xi_n)$, and under this isomorphism,   $c_V^{-1}\vartheta^{U_n}$ is identified with the natural section of $\Theta_G(V)^{-1}$.

\item For any two $G$-vector bundles $V_1$ and $V_2$ on $X$, we have a natural isomorphism \[\Theta(V_1\oplus V_2)\cong \Theta(V_1)\otimes \Theta(V_2).\]

\item The assignment sending a $G$-vector bundle $V$ on $X$ to the line bundle $\Theta(V)$ on $\catA^X_G$ extends to a group homomorphism $\Theta: K_G(X)\to \Pic(\catA^X_G)$.
\end{enumerate}

\end{prop}
\begin{proof}
Part (1) follows from the definition. Part (2) follows directly from part (1) and Lemma~\ref{lem: theta_add_tensor}. Part (3) follows from part (2). 
\end{proof}

Recall that similar to Proposition~\ref{prop:equivThom}, for any $G$-vector bundle $V$ on a finite $G$-space $X$, the Thom space $Th(V)=D(V)/S(V)$ of the vector bundle is again a finite $G$-space over $X$. The following is the Thom isomorphism in  equivariant elliptic cohomology.
\begin{theorem}\cite[(2.1.3)]{GKV95}\label{thm:thom}
Let  $V$ be a $G$-vector bundle on $X$. There is a canonical isomorphism $\pi_{X*}\Theta(V)\cong \wt\Ell^0_G(Th(V))$ making the following diagram commutative
\[\xymatrix{
\pi_{X*}\Theta(V)\ar[r]\ar[d]^{\cong}&\pi_{X*}\calO_{\catA_G^X}\ar@{=}[d]\\
\wt{\Ell}^0_G(Th(V))\ar[r]&\calE^0_G(X).
}\]
Here  $\wt{\Ell}^0_G(Th(V))\to \calE^0_G(X)$  is the pull-back via the embedding $X\to Th(V)$.
\end{theorem}

We have the following property about change of groups.
\begin{corollary}\label{cor:Thom_Restrict}
Let $X$ be a finite $G$-space, with a $G$-vector bundle $V$ on $X$. Let $\psi:H\to G$ be a group homomorphism, then there is an isomorphism $\Theta_H(V)\cong ({\catA}_\psi^X)^*\Theta_G(V)$ of line bundles on $\catA_H^X$. 
\end{corollary}
\begin{proof}
By Assumption \ref{assum:G-equiv}.\eqref{prop:(3)}, we have the commutative diagram 
\[\xymatrix{
\catA_H^X\ar[r]\ar[d]_{\cong}&\catA_G^X\ar[d]_{\cong}\\
\catA_{H\times U_n}^{Fr}\ar[r]\ar[d]&\catA_{G\times U_n}^{Fr}\ar[d]\\
\catA_{H\times U_n}\ar@{=}[r]&\catA_{H\times U_n}
}\]
By definition of $c_V$ and the above diagram, the composition of $\catA_\psi^X:\catA_H^X\to\catA_G^X$ with $c^G_V:\catA_G^X\to \catA_{U_n}$ is equal to $c^H_V$. Now the statement follows from Proposition~\ref{prop:GKVClass_general}(1).
\end{proof}

\begin{example}
When $G=T$ is a torus and $X=\pt$,  then a vector bundle $V$ of rank $n$ over $X$ is a representation of $T$, which defines a map $\rho:T\to U_n$ with virtual character $\la_V\in\bbZ[\bbX^*(T)]$. Then the map $c_V:\catA_T\to \catA_{U_n}$ coincides with $\catA_\rho:\catA_T\to \catA_{U_n}\cong E^{(n)}$.  So $\Theta(V)\cong \wt\Ell^0_T(Th(V))$ is equal to 	 $\calO({-\ker \chi_{\la_V}})$ defined in Section~\ref{sec:grp_alg}. 
\end{example}

By Proposition~\ref{prop:GKVClass_general}(3), $\Theta(V)$ is well-defined for any virtual vector bundle. For any    proper 
morphism $q:X\to Y$ between two smooth complex	 $G$-manifolds, we define $Tq$ to be the virtual vector bundle $q^*TY-TX$ on $X$, so that $\Theta(Tq)$ is well-defined, which for simplicity we also denote by  $\Theta(q)=\Theta(Tq)$, called the Thom bundle of $q$.
Following \cite[\S~2.3]{GKV95}), there exists an (elliptic cohomology) push-forward   morphism $q_{\sharp}: \Theta(q)\to q_{\catA}^{-1}\Ell^0_G(Y)$ of sheaves. Equivalently, it is determined by the morphism $q_{\catA*}\Theta(q)\to \Ell^0_G(Y)$	 of sheaves of $\Ell^0_G(Y)$-modules, still denoted by $q_{\sharp}$. 
The functor sending an equivariant proper map to the pushforward morphism is determined by the following two properties. 
\begin{enumerate}
\item For any $G$-equivariant regular embedding $q:X\to Y$. We denote the homotopy cofiber of the inclusion map $Y\setminus X \inj Y$ by $Th_Y(X)$.  We have $N_XY=Tq$, and $Th(Tq)$ is homotopy equivalent to $Th_Y(X)$. Then the map  $\pi_{X*}(q_\sharp): \pi_{X*}\Theta(q)\to \Ell_G^0(Y)$, as a map of sheaves over $\catA_T$,  is the composition  of the Thom isomorphism 
$\pi_{X*}\Theta(q)\cong \Ell^0_G(Th(Tq))$ and  the pull-back 
$\Ell^0_G(Th(Tq))\to \Ell^0_G(Y)$. 
\item Similar to the abelian case Proposition~\ref{prop:equivThom}, we have  the natural map $X\times V\to Th(V)\wedge X_+$, with   $\Ell_G^*(Th(V)\wedge X_+)\cong 
\Theta(V)\otimes \Ell^*_G(X)$. 
\end{enumerate}
As claimed above, for any proper map $q:X\to Y$, the pushforward $\Theta(q)\to \Ell^0_G(Y)$ is determined by the above two properties via 
 the usual Pontryagin-Thom construction. Indeed,  we can find a representation $V$ so that we have an equivariant embedding $\tilde{q}:X\to V\times Y$. Now notice that we have the isomorphism $Th_Y(Y\times V)\cong Th(V)\wedge Y_+$, where $Y_+$ is  $Y$ with an added base-point so that the $G$-action is trivial on the base-point (see. e.g., \cite[\S~23.5]{May} for the non-equivariant case). Notice also that $\tilde{q}$ induces a map $X\to Th_Y(Y\times V)=Sph(V\times Y)/Y_\infty$ where $Sph(V\times Y)$ is a compactification of each fiber into a sphere and $Y_\infty$ is the section at infinity, and that we have the map $Y\times V\setminus X\to  Th_Y(Y\times V)\setminus X$ induced by the inclusion $Y\times V\to Th_Y(Y\times V)$, and hence $Th_X(Y\times V)\cong Th_X(Th_Y(Y\times V))$. Composing the isomorphism  $\Ell_G(Th_X(V\times Y))\cong \Theta(V\oplus Tq)\cong \Theta(V)\otimes\Theta(Tq)$ with  the isomorphism $\Ell_G(Th_X(V\times Y))\cong \Ell_G(Th_X(Th_Y(Y\times V)))$, the pullback $\Ell_G(Th_X(Th_Y(Y\times V)))\to \Ell_G(Th_Y(Y\times V))$,  and the isomorphism $\Ell_G(Th_Y(Y\times V))\cong \Ell_G(Th(V)\wedge Y_+)\cong \Ell(Y)\otimes \Theta(V)$, then applying $-\otimes\Theta(V)^{-1}$, we finally obtain $\Theta(q)\to \Ell^0_G(Y)$.

Now we list a few useful properties of  $\Theta(q)$ as well as the pushforward map $q_{\sharp}: q_{\catA*}\Theta(q)\to \Ell^0_G(Y)$. 
\begin{prop}\label{prop:push}
\begin{enumerate}
\item Let $V$ be a $G$-equivariant vector bundle on $Y$, then, $q_{\catA}^*\Theta(V)\cong \Theta(q^*V)$ \cite[Remark~2.4.4]{GKV95}.
\item The map $\Theta: K_G(X)\to \Pic(\catA_G^X)$ is functorial in $X$. In particular, let $q:X\to Y$ and $p:Y\to Z$ be equivariant maps, we have $q_\catA^*\Theta(p)\cong \Theta(q^*Tp)$, and that  $\Theta(p\circ q)=\Theta(q)\otimes q_\catA^*\Theta(p)$ as coherent sheaves on $\catA_G^X$.
\item The map $q_{\sharp}:  q_{\catA*}\Theta(q)\to \Ell^0_G(Y)$	is a morphism of sheaves of $\Ell^0_G(Y)$-modules \cite[\S~2.3]{GKV95}.
\item The correspondence sending a $G$-equivariant proper map $q:X\to Y$ to $\pi_{X*}\Theta(q)$ is functorial, sending a commutative diagram 
\[\begin{xymatrix}{
X\ar[rr]|-{\phi}\ar[dr]_q&&X'\ar[dl]^{q'}\\
&Y
}\end{xymatrix}\]
to the obviously defined $\phi_\sharp:\pi_{X*}\Theta(q)\to \pi_{X'*}\Theta(q')$ \cite[(2.4)]{GKV97}. 
\item For a $G$-equivariant embedding of a complex variety into a smooth complex variety $i:S\to M$ which is a  $G$-subcomplex, as an abuse of notations we write $\Ell_T^0(M,M\setminus S)$ as $\Theta(i)$. It is a sheaf of modules on $\catA_G^S$. Then, we have a well-defined map $\pi_{S*}\Theta(i)\to \Ell^0_G(M)$,  which as an abuse of notations is denoted by $i_\sharp$ \cite[(2.2)]{GKV95}.
\item Given  the following Cartesian diagram of $G$-varieties with  $X$ and $X'$ smooth:
\begin{equation*}\xymatrix@R=1.5em{
Y' \ar@{^{(}->}[r]^{i_{Y'}} \ar[d]_{f} & X' \ar[d]^{g}\\
Y\ar@{^{(}->}[r]^{i_Y} & X, 
}
\end{equation*}
 there is a well-defined pullback construction $g^\sharp: \pi_{Y*}\Theta(i_Y) \to \pi_{Y'*}\Theta(i_{Y'})$ \cite[\S~7.3]{YZ}.   Moreover, the following diagram commutes \cite[(24)]{YZ} \begin{equation*}
\xymatrix@R=1.3em{
\pi_{Y*}\Theta(i_Y) \ar[r]^{i_{Y\sharp}}  \ar[d]_{f^{\sharp}} & \Ell_G(X) \ar[d]^{g^\sharp}\\
\pi_{Y'*}\Theta(i_{Y'})\ar[r]^{i_{Y'\sharp}} & \Ell_G(X').
}\end{equation*}
\item Assume furthermore from part (5) that $f$ is proper, define 
\[\Theta(f):= \Theta(i_{Y'})\otimes \mathcal{H}om(f^* \Theta(i_Y), i_{Y'}^* \Theta(g)).\] Then $ \Theta(f)$ is independent of the smooth embeddings, and the properties (1)-(5) above extend to this case \cite[2.5.2]{GKV95}. 
\end{enumerate}
\end{prop}
\begin{remark}\label{rmk:ellBorelMoore}
\begin{enumerate}
\item Properties (1) and (4) above are stated for complex manifolds in \cite{GKV95}. Nevertheless, using the convention from (7), they also hold in the singular case \cite[2.5.2]{GKV95}. Hence in the above we state the general case. 
\item In what follows, for a singular $G$-variety $S$ we avoid considering $\catA_G^S$. Instead, we fix an embedding $i:S\to M$ into a  complex manifold and consider  $\Theta(i)$ instead. This should be thought of as the elliptic 
  Borel-Moore homology  \cite[\S~2.2]{GKV95}.
\item 
The property (3) is the analogue of the projection formula, and the property (6) is the analogue of base-change property. We also note that the analogue of the projective bundle formula is given by Lemma~\ref{lem:GKV_classify}. Together with the homotopy invariance property, these are all the main ingredients in defining a convolution algebra as in \cite{ZZ14}.  
\item 
Let $q:X\to Y$ be a proper smooth morphism of complex $G$-varieties, then the relative normal bundle $Tq$ is a well-defined vector bundle on $X$ without smoothness assumption on either $X$ or $Y$. In this case the definition of $\Theta(q)$ from (7)  agrees with $\Theta(Tq)$.
\end{enumerate}
\end{remark}

As the push-forward is defined up to a twist by an explicit line bundle, we have two different notions of Chern classes in equivariant elliptic cohomology, coming from the orientation of the elliptic curve and the Thom isomorphism, respectively. In the terminology of \cite{GKV95}, they are called the Chern classes and the Euler class, respectively.

Let $p:E\to S$ be the
structure map of the elliptic curve. 
\begin{definition}\cite[(1.8.1)]{GKV95}\label{def:loc_coord}
A local coordinate on $E$ is a rational section $\fl$ of the line bundle $p^*\omega^{-1}$ that is regular in a neighbourhood of the zero-section $0:S\to E$ and vanishes on the image of $0$. Moreover, the  following condition should be satisfied:   the differential $d(\fl)$,  which is a rational section of $p^*\omega^{-1}\otimes\Omega$, goes to the identity under the isomorphism  $(p^*\omega^{-1}\otimes\Omega)|_{0}\cong\sHom_S(\omega^{-1},\omega^{-1})$.\end{definition}
We refer the interested readers to \cite[(1.8.3.2)]{GKV95} for a discussion of how a choice of local coordinate is related to the formal group law. Nevertheless, we use the following property. By the discussion after Theorem~\ref{thm:lur_able} (which is a brief review of \cite[p19]{Lur}), 
we have an  isomorphism coming from the orientation
$\Ell_G^0(\pt)\otimes \omega^{-i}\cong \Ell^{2i}_G(\pt)$.

Let $f$ be any rational section of $p^*\omega^{-1}$  on $E$  (not necessarily  a local coordinate of $E$). For any $i\geq 0$, let  $e_i$ be the $i$-th elementary symmetric function, and  $e_i(f)(x_1,\dots,x_n):=e_i(f(x_1),\dots,f(x_n))$ is a section of $p_n^*(\omega)^{-i}$ on $E^{(n)}$ with the structure map $p_n:E^{(n)}\to S$.
For any $G$-vector bundle $V$ on a finite $G$-space $X$ of rank $n$, define the $i$-th  $f$-Chern class $c^f_i(V)$ of $V$ to be $c_V^{-1}(e_i(f))$ as a rational  section of $\Ell_G^0(X)\otimes \omega^{-i}\cong \Ell^{2i}_G(X)$. A choice of a local coordinate provides a trivialisation of the line bunle $\omega$ and $\omega^i$ for any $i$, and hence of $p^*\omega^i$.  In what follows, we frequently treat a section of $\omega^i$ as a function.

On the other hand, define the  Euler class
\[e(V)=c_V^{-1}(\vartheta^{U_n})\in H^0(\catA_G^X,\Theta(V)^{-1})\cong H^0(\catA_G^X,c_V^{*}\Theta_{U_n}(\xi_n)^{-1}).\] 
It follows from the Thom isomorphism theorem that $e(V)$ coincides to  $z_{\sharp}:\Theta(V) \to z_{\catA}^{-1}\Ell^0_G(V)\cong \Ell^0_G(X) $ induced by the zero section $z:X\to V$, since one has  $V\cong Tz$.

%\begin{example}
%If the $G$-action on $X$ is free, then $e(V)$ is the usual theta function. On the other hand, if the $G$-action on $X$ is trivial, then $e(V)=c_1^{\fl}(V)$ where $(\Pi, \fl)$ is a local coordinate of $E$  in the sense of Definition \ref{def:loc_coord}.
%\end{example}

\begin{lemma}[\cite{GKV95}, (2.9.2)]\label{lem:loc_Eul}
Let $T$ be a torus, and let $a\in \catA_T$ be an $R$-point. Let $T(a)<T$ be as in  \eqref{eq:fixeda}, and let 
\[
i_{a\sharp}:\pi_{X^{T(a)}*}\Theta(T(X^{T(a)}))^{{-1}}\to \pi_{X*}\Theta(TX)^{{-1}}
\]
be the push-forward induced by the inclusion $i_a:X^{T(a)}\to X$.
Then, $i_{a\sharp}$ is given by multiplication by $e(T_{X^{T(a)}}X)$. Moreover, $i_{a\sharp}$ is compatible with the push-forward induced by the  inclusion  $T(a)\to T(a')$ for $a, a'\in \catA_T$, and is invertible on a Zariski open subset of $\catA_T$ containing $a$.
\end{lemma}

\subsection{The Quillen-Weyl-Kac formula}\label{subsec:Quillen}
To illustrate the Thom bundles and the push-forwards, we prove the formula for push-forward in equivariant elliptic cohomology from a projective bundle, which will be used later.
Let $V\to X$ be a rank-$n$ $G$-vector bundle. Let $p:\PP(V)\to X$ be the projection. 
Recall from Lemma \ref{lem:GKV_classify} that the following diagram \begin{equation}\label{eqn:diagram_projbundle}
\xymatrix{
\catA_{G}^{\PP(V)}\ar[r]^{c_{\PP(V)}}\ar[d]&\catA_P\ar[d]^-{\pi} \\
\catA_G^X\ar[r]^{c_V}&\catA_{U_n}. 
}\end{equation}
is a Cartesian square.
We describe the Thom bundle $\Theta(p)$  as 
follows. Consider the map
\[c_{Tp}:\catA_P\cong E^{(n-1)}\times E\to \catA_{U_{n-1}}\cong E^{(n-1)}, \quad 
((y_1,\dots,y_{n-1}),y_n)\mapsto 
(y_1-y_n,\dots,y_{n-1}-y_n).\]
 The pullback of the  natural section $\prod_{i=1}^{n-1}\vartheta(x_i)$ of the 
 line bundle $\Theta(\xi_{n-1})^{-1}$ on $\catA_{U_{n-1}}\cong E^{(n-1)}$ is 
 equal to $\prod_{i=1}^{n-1}\vartheta(y_i-y_n)$,
 as a section of $c_{Tp}^*\Theta(\xi_{n-1})^{-1}$ on $\catA_{P}$. The Euler 
 class $e(Tp)$ is the section $c^{-1}_{\PP(V)}(\prod_{i=1}^{n-1}\vartheta(y_i-y_n))$ of the line bundle $\Theta(p)^{-1}\cong 
 c^{*}_{\PP(V)}c_{Tp}^*\Theta(\xi_{n-1})^{-1}$. 

Pushing-forward the structure sheaves in 	\eqref{eqn:diagram_projbundle},  $\pi_*c_{\bbP(V)*}\calE ^0_G(\PP(V))$, $\pi_* \calE ^0_G(X)$, and $c_{V*}\calO_{E^{(n-1)}\times E}$ as coherent sheaves on $\catA_{U_n}\cong E^{(n)}$. 
The morphisms $\pi$, $c_V$, and $c_{\bbP(V)}$ are all affine morphisms. When there is no confusion, we omit $\pi_*c_{\bbP(V)*}$, $\pi_* $, and $c_{V*}$ from the notations of the above sheaves following the convention of  \S~\ref{subsec:dersch}.

\begin{lemma}\label{lem:proj_fiber}
Let $X$ be a finite $G$-space with a $G$-vector bundle $V$ of rank $n$. As coherent sheaves on $\catA_{U_n}\cong E^{(n)}$,  we have the isomorphism \[\calE ^0_G(\PP(V))\cong   \calE ^0_G(X)\otimes_{\calO_{\catA_{U_n}}}\calO_{E^{(n-1)}\times E}.\]
\end{lemma} 
\begin{proof}
	Note that the map $\catA_P\cong E^{(n-1)}\times E\to E^{(n)}\cong \catA_{U_n}$ is a finite map, hence affine. The base change of an affine morphism is affine. Hence, by definition of base change of affine schemes, we get the isomorphism  $\calE ^0_G(\PP(V))\cong c_V^*\pi_*(\calO_{E^{(n-1)}\times E})$ as coherent sheaves on $\catA_G^X$. Then using the projection formula we get that  $c_{V*}\calE ^0_G(\PP(V))\cong c_{V*}\calE ^0_G(X)\otimes_{\calO_{\catA_{U_n}}} \pi_*(\calO_{E^{(n-1)}\times E})$ as claimed.
\end{proof}
We comment that the flatness of the map $E^{(n-1)}\times E\to E^{(n)}$ should be thought of as the splitting principle in elliptic cohomology. 

We analyse $\pi_*\calO_{E^{(n-1)}\times E}$ as a sheaf of $\calO_{E^{(n)}}$-modules. The map $E^{(n-1)}\times E\to E^{(n)}\times E$ is injective, and hence locally on $U\subseteq E^{(n)}$ a section of $\pi_*\calO_{E^{(n-1)}\times E}$ is determined uniquely by a function on an open subset of $E^{(n)}\times E$.  Note that the map $E^{(n-1)}\times E\to E^{(n)}\times E$ is identity  on the $E$-component, and hence by our convention the coordinate of $E$ is denoted by $y_n$. We write a section as $f(y_n)$ to emphasize its dependence on $y_n$. In this sense, locally a section $f$ of $\pi_*\calO_{E^{(n-1)}\times E}$ is a function on  $E$ valued in $\calO_{E^{(n)}}(U)$. 

Thanks to Lemma~\ref{lem:proj_fiber} above, $c_{V*}\calE ^0_G(\PP(V))$ as a sheaf of $c_{V*}\calE ^0_G(X)$-modules is $c_{V*}\calE ^0_G(X)\otimes_{\calO_{\catA_{U_n}}} \pi_*(\calO_{E^{(n-1)}\times E})$. 
Any local section $f$ of $c_{V*}\calE ^0_G(\PP(V))$ on $U\subseteq E^{(n)}$ is uniquely determined as a function on an open subset of $\catA_G^X\times E$, and consequently  can be written as a function on  the $E$-factor with values in $c_{V*}\calE ^0_G(X)(U)$.
Using this description of $c_{V*}\calE ^0_G(\PP(V))$, we write a local section as $f(y_n)$. Any local section $f$ on $U\subseteq E^{(n)}$ is uniquely determined as a function on an open subset of $\catA_G^X\times E$, and consequently  can be written as a function on  the $E$-factor with values in $c_{V*}\calE ^0_G(X)(U)$.
Similarly, for a coherent sheave $\calM$ on $\catA_P$,  the pullback to  $\catA_G^{\bbP(V)}$ is so that $c_{V*}c^*_{\bbP(V)}\calM\cong c_{V*}\calE ^0_G(X)\otimes_{\calO_{\catA_{U_n}}} \pi_*\calM$. Any local section $f$ on $U\subseteq E^{(n)}$ is uniquely determined as a function on an open subset of $\catA_G^X\times E$, and consequently  can be written as a function on  the $E$-factor with values in $c_{V*}\calM(U)$.
Taking $\calM$ to be $c_{Tp}^{*}\Theta(\xi_{n-1})$, its pullback to $\catA_G{\bbP(V)}$ is $\Theta(p)$. It is customary to compose with the inverse map of $E$ and write $f(-y_n)$. 
We have the following Quillen-Weyl-Kac formula, calculating $p_{\sharp}f(-y_n)$.
\begin{prop}\label{prop:Quillen-Weyl-Kac}
Let $V\to X$ be a rank-$n$ $G$-vector bundle with $p:\PP(V)\to X$  the projection.
The push-forward $p_{\sharp }:p_{\catA*}\Theta(p)\to \calE ^0_G(X)$ is given by
\begin{equation}\label{Quill-formula}
\sum_{i=1}^{n-1}(i,n)\frac{f(-y_n)}{\prod_{j=1}^{n-1}c^{-1}_{\PP(V)}\vartheta(y_j-y_n)}
\end{equation}
for any local section $f$  of $p_{\catA*}\Theta(p)$.
Here $(i,n)\in \fS_n$ is the transposition switching the variables $y_i$ and $y_n$. 
\end{prop}
The proof is standard, using splitting principle, similar to that of \cite[Theorem~5.30]{Vish}, based on functoriality of $()_\sharp$ with respect to composition of proper maps, and the fact that for a closed embedding the pushforward is given by the Euler class. For the convenience of the readers, we sketch it below. 
\begin{proof}
Passing to a projective bundle and then an affine bundle if necessary, we may assume without loss of generality that $V:=\calL_1\oplus\cdots\oplus\calL_n$. The proposition is clear when $n=1$. We argue by induction. Let $I$ be an non-empty subset of $\{1,\dots,n\}$, and let $P_I:=\bbP(\oplus_{i\in I}\calL_i)$ with $u_I: P_I\to \bbP(V)$ the natural embedding and $p_I:P_I\to X$ the projection. Hence we have $p_I=u_I\circ p$ and $p_{I\sharp}=p_\sharp \circ u_{I\sharp}$. By Theorem~\ref{thm:thom} and the discussion in the paragraph before Proposition~\ref{prop:push}, $u_{I\sharp}$ is multiplication by $\prod_{i\in I}e(\calO(1)\otimes\calL_i)$ (see the paragraph before Lemma~\ref{lem:loc_Eul}). Then, the same calculation as that of \cite[Proof of Proposition~5.29]{Vish} gives the formula for $n$.
\end{proof}

The push-forward formula from a projective bundle in a general oriented cohomology theory is given by Quillen, which takes the above form with the theta-functions replaced by the formal group law of the cohomology theory. In $K$-theory, Quillen's formula when applied to line bundles on a flag variety  yields the Weyl character formula for representations of the general linear group. 
In the elliptic case, applying Quillen's formula above  to  line bundles on a flag variety, the domain of the pushforward map, when taking global sections, can be identified as Looijenga's space of theta-functions, and hence has a preferred choice of basis, the Looijenga basis \cite[Definition~8.5]{Gan}. For each such basis element, the right hand side of the above formula,  plugged in the specific form of theta-functions \eqref{ex:theta}, is equal to the 
 Kac character formula for representations of the affine Lie algebras \cite[Corollary 8.6]{Gan}. See also \cite[11.4]{And00} for the presentation of Kac character formula in this form. 

An illustrating example of this formula will be given in Lemma~\ref{Lem:Quillen P1}.

Let $G=\hbox{SU}_2\subset G^{\bbC}=\hbox{PGL}_2(\bbC)$ act on $\Aff^2$ in the natural way. 
The only positive root of $G$ is denoted by $\alpha$, and the fundamental weight by $\omega$. We have $\alpha=2\omega$, and the standard 
$T= S^1$-action on $\Aff^2$ has two weights, $-\alpha/2$ and $\alpha/2$.
We identify $\PP^1$ with $\PP(\Aff^2)$ endowed with the above actions.  The projection $\PP^1\to\pt$ is denoted by $p$. Let $p_i:\PP^1\times\PP^1\to\PP^1$ be the $i$-th projection, for $i=1,2$.

Following \S~\ref{sec:grp_alg}, the Weyl group of $\hbox{SU}_2$ is $\fS_2$ with the only generator denoted by $s_\alpha$. Its action on $\catA_T=E\otimes\bbX_*(T)\cong E$  is by taking inverse of an element in $E$. Write the coordinate of $\catA_{G}^{\PP(\Aff^2)}\cong\catA_{T}^{\pt}\cong E$ as $(x)$. The $s_\alpha$-action on $\catA_{G}^{\PP(\Aff^2)}$  sends $(x)$ to $(-x)$.
\begin{lemma}\label{lem:P1PullPush}
We have the following formula for $p_{1\sharp}p_2^{\sharp}:\Theta(p)\to \calO_{\catA_G^{\PP^1}}$: for any local section of $\Theta(p)$ on a $\fS_2$-invariant open subset of $E$,
\[p_{1\sharp}p_2^{\sharp}\left(\sigma\right)=\frac{s_\alpha\sigma(x)}{\vartheta(-\chi_{\alpha})}+\frac{\sigma(x)}{\vartheta(\chi_{\alpha})}.\]
\end{lemma}
\begin{proof}
The fact that the $T$-action on $\Aff^2$ has weights $-\alpha/2$ and $\alpha/2$
implies that the GKV projective  classifying map $c_{\PP(\Aff^2)}:E\to E\times E$ sends $x$ to $(-x,x)$.

Recall that  we have a natural isomorphism \[\catA_{G}^{\PP(\bbA^2)}\cong E \to (E/\fS_2)\times_{E^{(2)}}(E\times E)\] with the map $E \to (E/\fS_2)\times(E\times E)$  given by $x\mapsto (x,(-x,x))$. By our convention in Proposition~\ref{prop:Quillen-Weyl-Kac}, we write the coordinates of the $E\times E$-factor as $(y_1,y_2)$.
We consider  local sections $\sigma$ of $\Ell^0_{T}(\PP^1)$ on  $(E/\fS_2)\times(E\times E)$ as  $\calO_{(E/\fS_2)\times E}$-valued functions $\sigma(y_2)$ in the variable $y_2\in E$. 
Note that the map $E \to (E/\fS_2)\times_{E^{(2)}}(E\times E)$ induces an isomorphism from $E$ to the last $E$-factor in $(E/\fS_2)\times(E\times E)$. 
In particular, in the  Quillen-Weyl-Kac formula $y_2$ can be substituted by  $x$, and $y_1$ can be substituted by $-x$.
The Quillen-Weyl-Kac formula \eqref{Quill-formula} yields
\begin{eqnarray*}
p_{1\sharp}\left(p_2^{\sharp}\sigma\right)
&=&\frac{\sigma(-x)}{\vartheta(\chi_{-\alpha/2}-\chi_{\alpha/2})}+\frac{\sigma(x)}{\vartheta(\chi_{\alpha/2}-\chi_{-\alpha/2})}\\
&=&\frac{s_\alpha\sigma(x)}{\vartheta(-\chi_{\alpha})}+\frac{\sigma(x)}{\vartheta(\chi_{\alpha})}.
\end{eqnarray*}

\end{proof}
\subsection{Convolution with Lagrangian correspondences}
\label{subsec:ell_coh_conv}
Now we briefly recall that the  convolution with Lagrangian correspondences defines an associative algebra. For  more details, we refer to \cite[\S~2.8]{GKV95}.

Let $M_i$ be smooth complex quasi-projective $G$-varieties, for $i=1,2$. Let $N_i=T^*M_i$, endowed with $G$-actions such that $N_i\to M_i$ are equivariant. Note that we do not require the action on $N_i$ to be the induced action on from that on $M_i$. An example of such is when $G=H\times S$, where $H$ acts on $M_i$ non-trivially and acts on $T^*M_i$ via the induced action, while the action of $S$ on $M_i$ is trivial but on the fibers of $T^*M_i$ are non-trivial. Let $Z\subseteq T^*(M_1\times M_2)$ be a regular embedding of $G$-invariant subvariety which is proper over $T^*M_1$, not necessarily smooth or Lagrangian. Let $\proj_i:Z\to N_i$ be the composition $Z\overset{j}\longrightarrow T^*(M_1\times M_2)\overset{\pro_i}\longrightarrow N_i$. 
Define \[\Xi_G(Z):=\Theta(\proj_1)\]
 which is a coherent sheaf on $\catA_G^Z$.   There is  a natural map \[*_Z: \pi_{Z*}\Xi_G(Z)\to \sHom_{\catA_G}(\Ell^0_G(N_2),\Ell^0_G(N_1)),\]
 defined by pulling back via $\proj_2$, multiply, and then pushforward via $\proj_1$ \cite[(2.7.1)]{GKV95}.
\begin{remark}
In the case when $Z$ is Lagrangian, there is a twisted version of this construction in \cite{GKV95}, which we do not use in this paper, although in the case $M_1=M_2$ it is easy to find the relation between the two. 
With notations \[\xymatrix{M_2&T^*M_2\ar[l]_-{\wt b_2}&Z\ar[l]_-{\proj_2}\ar[r]^-{\proj_1}&T^*M_1&M_1\ar[l]_-{i_1},}\]where $i_1$ is the zero section, in \cite{GKV95}  the coherent sheaf $\Xi_G(Z)$ is defined to be 
\begin{equation}\label{eq:twistcon} \Theta(\proj_1)\otimes\Theta(\proj_2^*\wt b_2^*N_2)\otimes\Theta(\proj_1^*\wt b_1^*N_1)^{-1}
\end{equation}
 on $\catA_G^Z$. Here $\wt{b_i}:N_i\to M_i$ is the bundle projection $i=1,2$. With this twisted version, there is   a map \[*_Z: \pi_{Z*}\Xi_G(Z)\to \sHom_{\catA_G}(\pi_{{M_2}*}\Theta(N_2)^{-1},\pi_{{M_1}*}\Theta(N_1)^{-1})\]
defined by  pull-back, multiplying sections (functions),  and push-forward. We refer the interested readers to \cite[\S~2.8]{GKV95} for more details. 
\end{remark}

For simplicity, we write $b_2:N_1\times N_2\to N_1\times M_2$ as the identity on the first factor and the bundle projection on the second  factor, and  $i_1:M_1\times M_2\to N_1\times M_2$ as the zero section.

Again for simplicity, 
in the lemma below we write $\Theta(p_1^*TM_1)^{-1}$
 as $ \Xi_{G}(M_1\times M_2)$, where $p_i:M_1\times M_2\to M_i$ are the projections $i=1,2$. Note that there is a natural map \[*_\calM: \pi_{M_1\times M_2*}\Xi_{G}(M_1\times M_2)\to \sHom_{\catA_G}(\Ell^0_G(M_2),\Ell^0_G(M_1)))\] defined by pull-back, multiplying sections,  and push-forward.
\begin{lemma}\label{Lem:convol_zerosec}
\begin{enumerate}
\item
Assume  $Z$ is smooth. We have \[\Xi_G(Z)\cong \Theta(j)\otimes\Theta(\proj_2^*\wt{ b_2}^*TM_2)^{-1}\otimes\Theta(\proj_2^*\wt{ b_2}^*N_2)^{-1},\] where the right hand side is defined  in \eqref{eq:twistcon}. Note the subscripts of the second term in the tensor. 
\item 
Assume $Z$ is smooth and that $b_2\circ j$ is proper with image $Y$ lying in $M_1\times M_2$. 
Let $p:Z\to Y$ be the induced map, and $\overline p_1:Y\to M_1$ be the first projection $p_1:M_1\times M_2\to M_1$ restricted to $Y$.  Then, we also have
\[\Xi_{G}(Z)\cong \Theta(b_2\circ j) \otimes \Theta({\overline p_1})^{-1}.\]
\item 
Under the same assumption as in (2), we have $i_1^*\circ ( b_{2}\circ j)_*$ is a well-defined injective
map of sheaves $\pi_{Z*}\Xi_{G}(Z)\to\pi_{M_1\times M_2*}\Xi_{G}(M_1\times M_2)$. 
Moreover, the following diagram commutes 
\[\xymatrix{\pi_{Z*}\Xi_{G}(Z)\ar[r]^-{*_Z}\ar[d]^-{i_1^*\circ ( b_{2}\circ j)_*}& \sHom_{\catA_{G}}(\Ell^0_G(N_2),\Ell^0_G(N_1))\ar[d]^-{\cong}\\
\pi_{M_1\times M_2*}\Xi_{G}(M_1\times M_2)\ar[r]^-{*_M}&\sHom_{\catA_G}(\Ell^0_G(M_2),\Ell^0_G(M_1)).}\]

\end{enumerate}
\end{lemma}
\begin{proof}
We prove (1). By definition, $\proj_1$ is the composition of $j$ and $\pro_1:N_1\times N_2\to N_1$. Hence by Proposition~\ref{prop:push}(2)
 \begin{equation}\label{eqn:proj=jT}\Theta(\proj_1)\cong \Theta(j)\otimes \Theta(\proj_2^*TN_2)^{-1}.\end{equation}
 We have that $TN_2$ is an extension of $N_2$ by $TM_2$, considered as $G$-vector bundles on $M_2$. Therefore, $\Theta(TN_2)^{-1}\cong \Theta(TM_2)^{-1}\otimes \Theta(N_2)^{-1}$. Plugging it in to the definition of $\Xi_G(Z)$, we get the conclusion.
 
 Using Proposition~\ref{prop:push}(2) and the fact that $j^*T{b_2}\cong -\proj_2^*\wt{b_2}^*N_2$, we obtain 
  $\Theta(b_2\circ j)=\Theta(j)\otimes\Theta(\proj_2^*\wt{b_2}^*N_2)^{-1}$ as coherent sheaves on $\catA_G^Z$.  Plugging  into the proof above, and observe that $\Theta(\overline{p}_1)=\Theta(TM_2)^{-1}$ we obtain (2).

The proof of (3) is a word-for-word repetition of that of  \cite[Lemma~5.4.27]{CG}.
\end{proof}

The followings are about convolution algebra, their change of groups properties and localizations. 

\begin{lemma}\cite[(2.7.3)]{GKV95}
Let $M_3$ be a smooth $G$-manifold with $N_3=T^*M_3$, and $Z'\subseteq N_2\times N_3$ be a closed subvariety, with $Z\circ Z'\subset N_1\times N_3$ the image of $Z\times N_3\cap N_1\times Z'$ under the projection. Then 
\begin{enumerate}
\item The composition of actions lifts to a morphism \[\pi_{Z*}\Xi_G({Z})\otimes\pi_{Z'*}\Xi_G({Z'})\to \pi_{Z\circ Z'*}\Xi_G({Z\circ Z'})\]
which is associative in the usual sense.
\item When $M_1=M_2=M$ and $Z\circ Z=Z$, then $\pi_{Z*}\Xi_G(Z)$ is a sheaf of  algebras and $\pi_{M*}\calE ^0_G(N)$ on $\catA_G$ is a representation of it.
\end{enumerate} 
\end{lemma}

As a consequence of Corollary~\ref{cor:Thom_Restrict}, we have the following property about change of groups.
\begin{corollary}\label{cor:Xi_loc}
With notations as above, let $\phi:H\to G$ be a group homomorphism. Let $\catA_\phi:\catA_{H}\to \catA_G$ be the induced morphism. Then there is an isomorphism  $\pi^H_{Z*}\Xi_H(Z)\cong \catA^*_\phi(\pi^G_{Z*}\Xi_G(Z))$, which commutes with convolutions.
\end{corollary}

In particular, if $N_1$ and $N_2$ are both endowed with trivial $G$-actions, then there is an isomorphism $\pi^G_{Z*}\Xi_G(Z)\cong p_{\catA_G}^*(\Ell^*(Z))$ where $p_{\catA_G}$ is the projection $\catA_G\to \pt$.

%Following \cite[(2.2)]{GKV95}, for each embedding $Y\hookrightarrow X$ of complex manifolds such that $G$ acts by holomorphic transformations, let $\Theta^{2i}_Y(X)=\Theta(T_YX)\otimes \omega^{\otimes (-i)}$, and $\Theta^{2i+1}_Y(X)=0$. Define $\Theta_Y(X)=\oplus_i\Theta^i_Y(X)$. Then we have that $\pi_{Y*}\Theta_Y(X)=\Ell_G(X,X\backslash Y)$. 

\subsection{Convolution algebra with abelian groups and localization}

\begin{assumption}\label{assum:Lag}
Let  $\xi$ be a 1-dimensional representation of $G$, and  assume the $G$-actions on the fiber directions of $T^*M_i, i=1,2$, are  given by the induced action on  $M_i$ twisted by $\xi$. Assume $Z$ is a Lagrangian subvariety.
\end{assumption}

In the remaining part of this section, let $G=T$ be a torus. Let $a\in \catA_T$ be an $R$-point, and  $T(a)<T$ be as in \eqref{eq:fixeda}. 
Let $i_Z: Z^{T(a)}\hookrightarrow Z$ and $i_{N}:N_1^{T(a)}\times N_2^{T(a)}\to N_1\times N_2$ be the inclusions. Then the following diagram is Cartesian
\[\xymatrix{
Z^{T(a)}\ar[r]_-{j^T}\ar[d]_{i_Z}&N_1^{T(a)}\times N_2^{T(a)}\ar[d]_{i_N}\\
Z\ar[r]_{j}&N_1\times N_2.
}\]
By Proposition~\ref{prop:push}.(6), we have the following map 
\begin{equation}\label{eqn:Loc}
i^{\sharp}:\Theta_T(j)\to i_{\catA*}\Theta_{{T(a)}}(j^T).
\end{equation}

As a consequence of Lemma~\ref{lem:loc_Eul}, 
under Assumption~\ref{assum:Lag}, multiplication by $e({T_{N_2^{T(a)}}N_2})^{-1}$ is a well-defined rational morphism of sheaves on $\catA_T$:
\begin{equation}
\pi_{N_2^{T(a)}*} \Theta(TN_2)^{-1}\to \pi_{N_2*}\Theta(TN_2^{T(a)})^{-1},
\end{equation}
which is regular and  invertible on a Zariski open subset of $\catA_T$ containing $a$ \cite[(2.9.1)]{GKV95}. Here recall that $T_{N_2^{T(a)}}N_2$ is the normal bundle of $N_2^{T(a)}$ in $N_2$.
Hence, \begin{equation}
\label{eqn:Thom_Loc}
 \Theta(\proj_2^*TN_2)^{-1}\to i_{\catA*}\Theta(\proj_2^{T*}TN_2^{T(a)})^{-1},
\end{equation}
where $\proj_2^T: N_1^{T(a)}\times N_2^{T(a)}\to N_1^{T(a)}$ is the projection.

Factoring $ \pi^T_{Z*}\Xi_T(Z)$ and $\pi^T_{Z^{T(a)}*}\Xi_T({Z^{T(a)}})$ as \eqref{eqn:proj=jT}, multiplying \eqref{eqn:Loc} and \eqref{eqn:Thom_Loc}
defines a rational morphism 
\[\rho_a:\pi^T_{Z*}\Xi_T(Z)\to \pi^T_{Z^{T(a)}*}\Xi_T({Z^{T(a)}}).\] 
\begin{prop}\cite[Proposition~2.10.1]{GKV95}\label{prop:rho_loc}
Under Assumption~\ref{assum:Lag},
the following is true.
\begin{enumerate}
\item There is a Zariski open subset of $\catA_T$ containing $a$, on which the restriction of $\rho_a$ is regular and invertible. In particular, $\rho_a$ is invertible on the stalks over $a\in \catA_T$.
\item The map $\rho_a$ commutes with the convolution actions.
\end{enumerate}
\end{prop}
 Let $\phi:T\to T/T(a)$ be the quotient and let $\catA_\phi:\catA_T\to \catA_{T/T(a)}$ be the induced map.
By Corollary~\ref{cor:Xi_loc}, we have an isomorphism
 \[\pi^T_{Z^{T(a)}*}\Xi_{T}({Z^{T(a)}})\cong {\catA_\phi}^*\pi^{T/T(a)}_{Z^{T(a)}*}\Xi_{T/T(a)}(Z^{T(a)}),\]
which commutes with convolutions.
When $T(a)=T$, $\catA_{T/T(a)}\cong \Spec R$ and $\Xi_{T/T(a)}(Z^{T(a)})=\Xi_{1}(Z^{T(a)})$ is an $R$-module. Recall that by Lemma~\ref{Lem:convol_zerosec} as well as the convention from Proposition~\ref{prop:push}(6),  $\Xi_{T/T(a)}(Z^{T(a)})$ is $\Theta(j)=\calE^0_{1}((N_1\times N_2)^T,(N_1\times N_2)^T\setminus Z^{T(a)})$ twisted by the Thom classes of $(N_2)^T\to M_2^T$ and $T(M_2^T)$. A choice of the local coordinate trivializes the twist and hence we obtain a convolution algebra structure on $\calE^0_{1}((N_1\times N_2)^T,(N_1\times N_2)^T\setminus Z^{T(a)})$.
Summarizing Corollary~\ref{cor:Xi_loc} and Proposition~\ref{prop:rho_loc}, we get the following bivariant Riemann-Roch theorem. Let $R_a$ be the skyscraper sheaf at the $R$-point $a\in\catA_T$. 
\begin{theorem}\label{thm:RR}
Assume $a\in\catA_T$ is such that $T(a)=T$. Then
the morphism 
\[\pi_{Z*}\Xi_T(Z)\otimes_{\calO_{\catA_T}} R_a\to \Xi_{T/T(a)}(Z^{T(a)})\]
induced by $\rho_a$ is an isomorphism,
and commutes with convolutions. Moreover, if $R$ is a $\bbQ$-algebra, then $\Xi_{T/T(a)}(Z^{T(a)})\cong H_*(Z^{T(a)};R)$.  If $Z\circ Z=Z$, then the above maps are  isomorphisms of  algebras.
\end{theorem}
Here, $H_*(Z^{T(a)};R)$ is the Borel-Moore homology with $R$-coefficients, endowed with the convolution product. 
Recall Remark~\ref{rmk:ellBorelMoore}(1) that in the singular case we use smooth embedding where up to a twist, $\Xi_{T/T(a)}(Z^{T(a)})$  is isomorphic to $\calE^0_{1}((N_1\times N_2)^T,(N_1\times N_2)^T\setminus Z^{T(a)})$, which 
 in turn is isomorphic to  $H^*((N_1\times N_2)^T,(N_1\times N_2)^T\setminus Z^{T(a)};R)$ when $R$ is a $\bbQ$-algebra. The Poincar\'e duality induces an isomophism of the latter with $ H_*(Z^{T(a)};R)$ \cite[(2.6.1)]{CG}. Here the algebra structure of $ H_*(Z^{T(a)};R)$ comes from convolution. The usual bivariant Riemann-Roch theorem identifies this algebra structure with the algebra structure on $\Xi_{T/T(a)}(Z^{T(a)})$  (see also \cite[Corollary~5.8]{ZZ14}).

\subsection{Remarks on convolution product on sheaves of spectra}
\label{subsec:conv_spectra}
In the cohomological construction  of the elliptic affine Hecke algebra below, we use the convolution of elliptic cohomology for a connected, simply-connected, compact Lie group. However, as is remarked in \S~\ref{subsec:existence}, when using cohomology theory to study the representations of these algebras, we only use the convolution algebra for a compact abelian Lie group. 

Much of pushforward and convolution in equivariant elliptic cohomology, equivariant with respect to compact abelian Lie groups, can be carried out on the level of sheaves of spectra as in \S~\ref{sec:Lurie}. 
Indeed, let $E^{der}$ on $S^{der}$ be an oriented derived  elliptic curve, $T$ be a compact abelian group, and $X$ be a finite $T$-space, then $\Ell_T(X)$ is a coherent sheaf on $\catA_T^{der}$. Let $V$ be an $T$-equivariant vector bundle on $X$. Define the Thom sheaf $\Theta(V):=\Ell_T(V,V\setminus X)$.
For any $T$-equivariant proper map $f:X\to Y$ between smooth complex manifolds, by choosing a $T$-representation $V$ into which $X$ equivariantly embeds, we extend $f$ equivariantly in to a regular embedding $i:X\to Y_+\wedge SV$.  We define the Thom sheaf of $f$ to be
$\Theta(N_i)\otimes\Theta(V)^{-1}$ where $N_i$ is the normal bundle of the map $i$  and $\Theta(V)= \widetilde{\Ell}_T(Th(V))$ is the line bundle from Proposition~\ref{prop:equivThom}. 
We expect  it to be easy to verify all the properties from Proposition~\ref{prop:push} up to Lemma~\ref{Lem:convol_zerosec}. The associativity of the convolution algebra as a sheaf of spectra is however a subtle point, which we do not attempt to address in this paper.

We note that recently, for compact abelian $T$, the details in constructing $\Ell_T(X)$ as a coherent sheaf of algebras on $\catA_T^{der}$ for  elliptic curves in Theorem~\ref{prop:derived_Ell_Curv} has been established in \cite{GM} using Lurie's work \cite{Ell2}. Furthermore,   Gepner's work in preparation characterizes $\catA^{der}_{U_n}$ as the classifying space of effective divisors of degree $n$ on $E^{der}$. Based on this, it is reasonable to expect the GKV-classifying map to be defined as a map of sheaves of spectra,    and a version of the  Thom isomorphism theorem in a similar fashion as Theorem~\ref{thm:Thom}, the $\pi_0$ of which is given in \S~\ref{subsec:GKV_classifying}.  With this, we expect the explicit calculations in \S~\ref{sec:iso} to hold in this generality as well.

\section{The elliptic affine Hecke algebra: the algebraic construction}\label{sec:ellHecke}
In this section, we recall the definition of the elliptic affine Hecke algebra. We prove some basic properties about its structure. Throughout this section we assume that $R$ is an integral domain, unless otherwise stated.

\subsection{The elliptic affine Hecke algebra}
\label{sec:ell}

 We recall the notion of root datum following \cite[Exp. XXI, \S~1.1]{SGA3}. A root datum of rank $n$ consists of a lattice $\Lambda$ of rank $n$, a  non-empty subset $\Sigma\subset \La$, together with an embedding $\Sigma\hookrightarrow \La^\vee, \al\mapsto \al^\vee$ where $\La^\vee$ is the dual of $\La$. Elements in $\Sigma$ are called roots, and let $\Sigma=\Sigma^-\cup \Sigma^+$ be the decomposition of $\Sigma$ into union of positive roots and negative roots. Let $\La_r\subset \La$ denote the root lattice, and $\La_w$ be the weight lattice, with $\La_w^\vee$ the dual weight lattice. A root datum is irreducible if it is not a direct sum of two non-trivial root data, and is called semi-simple if $\La_r\otimes_\Z\bbQ=\La\otimes_\Z\bbQ$. We always assume that the root datum is semi-simple. A root datum is said to be adjoint (resp. simply connected) if $\La=\La_r$ (resp. $\La=\La_w$).

Let  $\Phi=\{\al_1,...,\al_n\}\subseteq \Sigma$ be the set of simple roots, with simple reflections $s_i:=s_{\al_i}$.  The Weyl group is denoted by $W$. For each $w\in W$, define $\Sigma(w)=w\Sigma^-\cap \Sigma^+$. 

We now recall the definition of the elliptic affine Hecke algebra  given in \cite{GKV97}.  Let $\Gamma$ be a free abelian group of rank one, and let $E$ be an elliptic curve over  $R$. Consider $\catA:=E\otimes_{\bbZ}(\Gamma\oplus\Lambda^\vee)\cong E^{n+1}$, which is an abelian variety. 
  Each group homomorphism $\la:\Gamma\oplus \La^\vee\to \Z$ extends to a map of abelian varieties $\chi_\la:\catA\to E=E\otimes_{\bbZ}\Z.$ In particular, each 
$\alpha\in \Sigma$ defines a map of abelian varieties $\chi_\al: \catA\to E$. Let  $D^\al$ be the divisor of $\catA$  defined by the kernel of $\chi_\al$. Let $\gamma:\Gamma\oplus\Lambda^\vee\to \bbZ$ be a group homomorphism which is zero on $\Lambda^\vee$ and isomorphism on $\Gamma$.  Let  $D^{\al, \ga}$ be the divisor  on $\catA$ defined by the kernel of $\chi_\al-\chi_\ga$. We have  $D^\al=D^{-\al}$ and $w(D^\al)=D^{w^{-1}(\al)}$. The action of the Weyl group $W$ on $\La^\vee$ naturally extends to an action on $\catA$. Let $\pi:\catA\to\UP$ be the quotient, and $\calS:=\pi_*\calO_{\catA}$ on $\UP$. There is a natural action of $W$ on the sheaf $\calS$. Define $\calS_W:=\calS\rtimes W$, a coherent sheaf of associative algebras. Note that $\calS$ is a sheaf of modules over $\calS_W$. On any open subset, sections of $\calS_W$ coming from elements in $W$ will be denoted by $\de_w$. For any root $\al\in \Sigma$, we write $\de_\al=\de_{s_\al}$.

Let $\catA^c:=(\catA\backslash (\cup_{\al\in \Sigma}D^\al))/W$, and let $j:{\catA^c}\inj \UP$ be the inclusion.   The action of $\calS_W$ on $\calS$ induces a morphism $\rho:\calS_W|_{{\catA^c}}\to \sEnd(\calS|_{{\catA^c}})$. Applying $j_*$, we get a morphism \[j_*\rho:j_*(\calS_ W|_{{\catA^c}})\to j_*\sEnd(\calS|_{{\catA^c}}).
\]
Let $p:\sEnd(\calS)\to j_*\sEnd(\calS|_{{\catA^c}})$ be the canonical map. 

\begin{definition} \label{def:res}We define $\widehat \calD$ to be $(j_*\rho)^{-1}\left(p(\sEnd(\calS)\right)$, as a quasi-coherent subsheaf of algebras of $j_*(\calS_W|_{{\catA^c}})$.  %Define $\widetilde\calH$ to be the subsheaf of left $j_*(\calS|_{{\catA^c/W}})$-modules of $\widetilde\calD$ generated by  $\calT_{I_w}, w\in W$. 
\end{definition}

We consider local sections of $j_*(\calS_W|_{{\catA^c}})$ written as $\sum_{w\in W}f_w\delta_w$ where $f_w$ are local sections of $j_*(\calS|_{\catA^c})$, satisfying the following conditions.

\begin{enumerate}
\item[\textbf{R}$1$] for any root $\al$ and $w\in W$, each $f_w$ has a pole of order at most one along the  divisor $D^\al$; 
\item[\textbf{R}$2$] for any root $\alpha$ and $w\in W$, the residues of $f_w$ and $f_{s_\al w}$ along the divisor $D^\al$ differ by a negative sign;
\item[\textbf{R}$3$] for any $\al\in \Sigma(w)$,  $f_w$ vanishes along the divisor $D^{\al, \ga}$.
\end{enumerate}
It follow from \textbf{R}$2$ that $f_w+f_{s_\al w}$, a priori defined in the complement of $D^\alpha$, can be extended across $D^\alpha$.

\begin{definition}\cite[Definition~1.3]{GKV97}
The elliptic affine Demazure algebra $\calD$ is defined to be the quasi-coherent subsheaf of $j_*(\calS_W|_{{\catA^c}})$ whose local sections  satisfy conditions  \textbf{R}1 and  \textbf{R}2. The elliptic affine Hecke algebra $\calH$ is defined to be the quasi-coherent subsheaf of $\calD$ whose local sections satisfy \textbf{R}3.
\end{definition}

\begin{remark}\label{Rmk:Formal}
Fix a local uniformizer of $E$, which defines an elliptic formal group law $F_e$. Then the completion of the stalk of $\calO_{\catA}$ at the origin of $E$ is precisely the formal group algebra defined in \cite[Definition~2.4]{CPZ}, and the completion of $\calD$ and $\calH$ at the origin of  $\UP$,  are  the elliptic formal affine Demazure algebra $\bfD_{F_e}$ and the elliptic formal affine Hecke algebra $\bfH_{F_e}$ considered in \cite{ZZ14}, respectively.   For general formal group law $F$, the dual of $\bfD_F$ is the algebraic model of the corresponding equivariant oriented cohomology of complete flag variety. For more details, please refer to \cite{CZZ1, CZZ2, CZZ3}. 
\end{remark}
The following proposition will be referred  as the structure theorem of the elliptic affine Demazure algebra and the elliptic affine Hecke algebra.
\begin{prop}\cite[Theorem~4.4]{GKV97}
 Let $\calJ$ be the sheaf of ideals of  $\calO_{\catA}$ corresponding to the divisor $\sum_{\alpha\in \Sigma}D^{\alpha,\gamma}$ on $\catA$. We have 
\[\calD\cong \widehat\calD\]
as sheaves of  algebras on $\catA/W$, and 
$\calH$ is isomorphic to the subsheaf of algebras of $\calD$ consisting of sections $u$ of $\calD$ such that \[u(\pi_*\calJ)\subset \pi_*\calJ.\]
\end{prop}

For any  rational function $f$ on  $E$, which vanishes of order 1 along the zero section of $E$, define $\Sing(f)$ to be the set of non-origin zeros of $f$ union with the set of poles of $f$.
Define \begin{equation}\label{eq:T}
T_\al^f=\frac{f(\chi_\gamma)}{f(\chi_\alpha)}+(1-\frac{f(\chi_\gamma)}{f(\chi_\alpha)})\delta_\alpha,
\end{equation}
 which is a regular section of $\calH$ on the open complement of $\chi_\gamma^{-1}(\Sing(f))\cup\chi_\alpha^{-1}(\Sing(f))$.
\begin{example}\cite[\S~4]{GKV97}\label{ex:ellip}  Let $E$ be a complex elliptic curve. Let $x$ be an order-2 torsion point on $E$, and consider the Jacobi-sign function $\sn$, which is the unique rational function that has simple zeros at the origin and $x$, poles of order 1 at the other two order-2 points, and derivative 1 at the origin.  According to Definition~\ref{def:loc_coord}, $\sn$ is a local coordinate of $E$. The sheaf of algebras $\calS$ is $\pi_*\calO_{E^{n+1}}$. Let
\[
V=[\catA\backslash (\cup_{\al\in \Sigma^+}\chi_\al^{-1}(x))]/W.
\]
Then $H^0(V, \calH)$ is the subalgebra of $\End(H^0(V,\calS))$ generated by $H^0(V,\calS)$ and $T_\alpha=\frac{\sn(\chi_\gamma)}{\sn(\chi_\alpha)}+(1-\frac{\sn(\chi_\gamma)}{\sn(\chi_\alpha)})\delta_\alpha$ for $\al\in \Phi$. 
\end{example}

\subsection{Filtration by the Bruhat order} \label{subsec:filH} 
For any $w\in W$, define $\calH_{\le w}\subseteq \calH$ to be the subsheaf whose local sections consist of $\sum_{y\le w}f_y\delta_y$. Denote $\calH_{w}=\calH_{\le w}/\calH_{<w}$.

For any point $p\in \catA$ and  any simple root $\alpha$, let $p_\alpha$ be the $\alpha$-coordinate of $p$. Let $f_{p,\alpha}$  be a rational function on $E$ such that $p_\alpha,-p_\alpha\notin\Sing(f_{p,\alpha})$. Note that such a function always exists (although non-unique). 
Let $U_{p,\alpha}\subseteq E$ be the open complement of $\Sing(f_{p,\alpha})\cup\Sing(-f_{p,\alpha})$. Then $U_p:=\prod_{\alpha}U_{p,\alpha}$ is a $W$-invariant open subset of $\catA$ that contains $p$. This subset $U_p$ depends on $p$ and $f_{p,\al}$.

\begin{lemma}\label{lem:basis}
Let $p\in \catA$. For each $w\in W$, we fix a reduced sequence $I_w=(i_1,...,i_l)$ of $w$, i.e., $w=s_{i_1}\cdots s_{i_l}$. 
\begin{enumerate}
\item Consider the subsheaf of the sheaf of rational sections of $\calS$, whose local sections consist of rational sections $f$ of $S$ that have simple poles at each divisor $D^\alpha$, and vanish  along the divisors $D^{\alpha,\gamma}$ for each $\alpha\in \Sigma(w)$. This subsheaf is a locally free sheaf of rank 1 over $\calS$.
\item On $U_p/W$, the sheaf in (1) is globally free, generated by $F_{I_w}:=(1-\frac{f_{\alpha_{i_1},p}(\chi_\gamma)}{f_{\alpha_{i_1},p}(\chi_{\alpha_{i_1}})})\cdots (1-\frac{f_{\alpha_{i_l},p}(\chi_\gamma)}{f_{\alpha_{i_l},p}(\chi_{\alpha_{i_l}})})$.
\item $\calH_w|_{U_p/W}$ is free of rank 1 as a module over $\calS|_{U_p/W}$, generated by $T_{I_w}:=T^{f_{\alpha_{i_1},p}}_{\alpha_{i_l}}\cdots T^{f_{\alpha_{i_l},p}}_{\alpha_{i_l}}$.
\end{enumerate}
\end{lemma}
\begin{proof}
Claim (1) is clear.
Claim (2) is a consequence of (1). 
An easy calculation as in \cite[Lemma~2.8.(ii)]{GKV97} shows that $T_{I_w}-F_{I_w}\delta_w\in(\calH_{<w})|_{U_p/W}$. Hence,  (3) follows from (2).
\end{proof}

The following theorem is essentially proved in \cite{GKV95}. 

\begin{theorem}\label{thm:ellHeckFree}
We have the following
\begin{enumerate}
\item $\calH_w$ is a locally free sheaf of modules over $\calS$ of rank 1;
\item $\calH$ is a locally free sheaf of modules over $\calS$ of rank $|W|$;
\item $\calH$ is a locally free sheaf of modules over $\calO_{\catA/W}$ of rank $|W|^2$.
\end{enumerate}
\end{theorem}

\begin{proof}
We prove, by induction on Bruhat order, that $\calH_{\leq w} $ is locally free of rank $\#\{y\in W\mid y\leq w\}$. The rest of the theorem follows easily from this, where the inductive step is taken care of by  Lemma~\ref{lem:basis}.(3) above.
\end{proof}  

\subsection{Elliptic Hecke algebra with dynamical parameters}\label{subsec:dyn}
In this section, we define a version of elliptic affine Hecke algebra with dynamical parameters. In this sheaf of algebras, we find global sections that satisfy the braid relations. We postpone to future publications the geometric study of this version of elliptic affine Hecke algebra, as well as its relation with the dynamical elliptic quantum group of Felder. 

Let $T$ be a torus. 
Recall that $\mathfrak  A_T\cong E\otimes \bbX_*(T)$, and let $\catA_T^\vee:=\Pic^0(\mathfrak A_T)\cong E\otimes \bbX^*(T)$ be the dual abelian variety. In this subsection only, we use $\la$ to denote elements in $\catA_T^\vee$.  The line bundle determined by $\la\in \catA_T^\vee$ is denoted by $\calO(\la)$.
Recall that by definition of $\Pic^0$, there is a universal line bundle $\calP$ on $\catA_T\times\catA_T^\vee$, called the Poincar\'e line bundle, determined by the property that for any $\lambda\in \catA_T^\vee$, the restriction    $\calP|_{\catA_T\times\{\lambda\}}$ coincides with $\calO(\lambda)$, and $\calP|_{0\times \catA_T^\vee}\cong \calO_{\catA^\vee_T}$.

The Weyl group actions on $\bbX^*(T)$ and on $\bbX_*(T)$ induce actions on $\catA_T$ and $\catA_T^\vee$, denoted by $w$ and $w^{\dyn}$, respectively. Denoting $w^{-1}:\catA_T\to \catA_T$, then for any  $\la\in \catA_T^\vee$, 
\[\calO(w^{\dyn}\la)=(w^{-1})^*\calO(\la). 
\]
In other words, the bundle $\calP$ is preserved by the diagonal Weyl group action. Moreover, $\calP$ is $W$-equivariant. 

Let  $\rho$ be the half sum of simple roots. 
Consider the  map  
\[vor:\mathfrak A_T\times\mathfrak A_T^\vee\times E\to \mathfrak A_T\times\mathfrak A_T^\vee\times E,  \quad (z,\lambda, \hbar)\mapsto (z,\lambda+\rho\hbar,\hbar). \]
 Let  $\mathbb L=vor^*(\calP\boxtimes\calO_{E})$ on $\mathfrak A_T\times\mathfrak A_T^\vee\times E$. Then for each $(\lambda,\hbar)\in \mathfrak A_T^\vee\times E$, the line bundle $\mathbb L$ restricted to the fiber $\mathfrak A_T\times\{(\lambda, \hbar)\}\cong \catA_T$ is $\mathbb L_{\lambda,\hbar}:=\calO(\lambda+\rho\hbar)$.

\begin{lemma} \label{lem:equivL}
There exists an action of  $W$ on $\mathfrak A_T\times\mathfrak A_T^\vee\times E$, denoted by $\star$,  satisfying the following two conditions:
\begin{enumerate}
\item the projection $\mathfrak A_T\times\mathfrak A_T^\vee\times E\to \catA_T$ is $W$-equivariant;
\item the line bundle $\mathbb L$ is preserved by the action.
\end{enumerate}
\end{lemma}
\begin{proof}
We first consider the diagonal action of $W$ on $\catA_T\times \catA_T^\vee\times E$. Since the map $vor:\mathfrak A_T\times\mathfrak A_T^\vee\times E\to \mathfrak A_T\times\mathfrak A_T^\vee\times E$ is an isomorphism, the $W$-action on the target  induces an action on the domain, denoted by $\star$. We have
\begin{align*}
w \star(z,\la, \hbar)&=vor^{-1}\circ w \circ vor (z,\la, \hbar)\\
&=vor^{-1}(wz, w^{\dyn}(\la+\rho \hbar), \hbar)\\
&=(wz, w^{\dyn}(\la+\rho \hbar)-\rho \hbar, \hbar).
\end{align*}
We then show that $\bbL$ is preserved by the $\star$-action. Let $U\subset \catA_T$ be an open subset. Then
\begin{align*}
\Gamma(U,((w\star\_)^*\bbL)_{\la, \hbar})&=\Gamma(wU, \bbL_{w^{\dyn}(\la+\rho\hbar)-\rho\hbar, \hbar})\\
%&=\Gamma(wU,\calO(w^{\dyn}(\la+\rho\hbar)-\rho \hbar+\rho\hbar))\\
%&=\Gamma(wU, \calO(w^{\dyn}(\la+\hbar w\rho))\\
%&=\Gamma (U, \calO(\la+\hbar \rho))\\
&=\Gamma(U, \bbL_{\la, \hbar}). 
\end{align*}
Therefore, $(w\star\_)^*\bbL=\bbL$. 
\end{proof}

The $\star$ action induces a new $W$-action on $\catA_T^\vee\times E$, denoted by $\bullet$ so that the projection $\catA_T\times\catA_T^\vee\times E\to \catA_T^\vee\times E$ is equivariant. More explicitly, we have
\begin{align}\label{eq:bullet}
(w^{-1})^*\bbL_{\la, \hbar}&=(w^{-1})^*\calO(\la+\rho \hbar)=\calO(w^{\dyn}\la+\hbar w\rho)=\bbL_{w\bullet \la, h},
\end{align}
where \[
w\bullet \la=w^{\dyn}\la+\hbar w\rho -\hbar \rho=w^{\dyn}(\la+\hbar \rho)-\hbar \rho. 
\]
This notation $(W,\bullet)$ is motivated by the usual dot-action of the Weyl group.
Note that $(w,\star)=(w,\bullet)\circ \delta_w=\delta_w\circ (w,\bullet)$ as maps $\catA_T\times\catA_T^\vee\times E\to\catA_T\times\catA_T^\vee\times E$ for any $w\in W$.

To avoid possible confusion, we spell out explicitly the effect of the group actions on the line bundles. 
As mentioned above, the Poincar\'e line bundle on $\catA_T\times\catA_T^\vee$ is equivariant with respect to the diagonal $W$-action, which has $w(z, \la)=(wz, w^{\dyn}\la)$ for any $w\in W$. Indeed, for any $w\in W$, consider $w^*(\calP)$ on $\catA_T\times\catA_T^\vee$ via the map $w:\catA_T\times\catA_T^\vee\to \catA_T\times\catA_T^\vee$. By the universality of the Poincar\'e line bundle, we get an isomorphism $w^*(\calP)\cong \calP$, denoted by $w$. 
The same is true for $\calP\boxtimes\calO_E$ on $\catA_T\times\catA_T^\vee\times E$. 
Note also that $vor$ is $W$-equivariant when the codomain has the diagonal action, and the domain the $\star$-action.
Pulling-back via the $W$-equivariant isomorphism $vor$ endows $\bbL$ an equivariant structure with respect to the $\star$-action. That is, we have isomorphisms $\bbL\cong (w\star\_)^*\bbL$, as mentioned in Lemma \ref{lem:equivL}. 

Now consider $\delta_\alpha:\catA_T\times\catA_T^\vee\times E\to \catA_T\times\catA_T^\vee\times E$, $(z,\la, \hbar)\mapsto (s_\al(z), \la, \hbar)$. The following triangle obviously commutes \[\xymatrix{\catA_T\times\catA_T^\vee\times E\ar[rr]^{\delta_\alpha}\ar[dr]_\pi&&\catA_T\times\catA_T^\vee\times E\ar[dl]^{\pi}\\
&\catA_T/W\times\catA_T^\vee\times E
}.\]
Hence, we have the obvious map of sheaves $\pi_*\bbL\to \pi_*\delta_\alpha^*\bbL$. Composing with the isomorphisms $\delta_\alpha^*\bbL=(s_\alpha,\bullet)^*(s_\alpha,\star)^*\bbL\cong (s_\alpha,\bullet)^*\bbL$, we get a map of sheaves $\pi_*\bbL\to\pi_*(s_\alpha,\bullet)^*\bbL$. Restricting to the fiber over $(\lambda,\hbar)\in \catA_T^\vee\times E$, we get a map of sheaves \begin{equation}\label{eq:de}\pi_*\bbL_{\lambda,\hbar}\to \pi_*\bbL_{s_\alpha\bullet\lambda,\hbar}.\end{equation} 
Here note that the operations of applying $\pi_*$ and taking restrictions to $(\la, \hbar)$ commute with each other. Without causing any confusion, we denote this map of sheaves still by $\delta_\alpha$. 

Similarly, we also have the commutative triangle \[\xymatrix{\catA_T\times\catA_T^\vee\times E\ar[rr]^{s^{dyn}_\alpha}\ar[dr]_\varpi&&\catA_T\times\catA_T^\vee\times E\ar[dl]^{\varpi}\\
&\catA_T\times\catA_T^\vee/W\times E
},\] and hence a map of sheaves $\varpi_*\bbL\to \varpi_*s^{dyn*}_\alpha\bbL$. However, unlike in the case of $\delta_\alpha$, here restriction to the fiber over  $(\lambda,\hbar)\in \catA_T^\vee\times E$ is not well-defined. Nevertheless, taking the orbits of $\lambda$ under $W$ gives 
\begin{equation}
\label{eq:sdyn}\bigoplus_{w^{dyn}\lambda=\lambda'}\pi_*\bbL_{\lambda',\hbar}\to \bigoplus_{w^{dyn}\lambda=\lambda'}\pi_*\bbL_{\lambda',\hbar}.\end{equation}
This map is denoted by $\de_\al^{\dyn}$. 
By definition, the map $\de_\al$ from \eqref{eq:de} is induced from the $W$-action on the variable $z\in \catA_T$ while the map $\de_\al^{\dyn}$ is induced from the $W$-action on the variable $\la\in \catA_T^\vee$, hence they commute with each other.

Recall from \S~\ref{sec:grp_alg} that for each simple root $\alpha$, we have the map  $\chi_\al:\catA_T \cong E\otimes \bbX_*(T)\to E$. For $z\in \catA_T$, denote $z_\al=\chi_\al(z)$. Similarly, $\al^\vee$ determines a map $\chi_{\al^\vee}:\catA^\vee_T\to E$, and for $\la\in \catA_T^\vee$, denote $\la_{\al^\vee}=\chi_{\al^\vee}(\la)$. 

Denote $\theta(x)=\vartheta(e^x)$. Note that our $\theta$ is related with the one in \cite[\S~2.1]{RW} by a factor
\[
\theta(x)=\frac{1}{2\pi \kappa i}\theta^{\text RW}(x), \quad \kappa:=\prod_{s>0}(1-q^s)^2. 
\]For any simple root $\al$ and element $\la\in \catA_T^\vee$, %we consider the local section of $\pi_*\calO_{\catA_T\times E}[W]$:  for any $(z,\hbar)\in \catA_T\times E$, 
define
\begin{equation}
\label{eqref:DL_dyn}
T^\lambda_\alpha=\frac{\theta(\hbar)\theta(z_\alpha+\lambda_{\alpha^\vee})}{\theta(z_\alpha)\theta(\lambda_{\alpha^\vee}+\hbar)}+\frac{\theta(z_\alpha-\hbar)\theta(\lambda_{\alpha^\vee})}{\theta(z_\alpha)\theta(\lambda_{\alpha^\vee}+\hbar)}\delta_\alpha, \quad T_\al^{\dyn}=T_\al^\la \de_\al^{\dyn}. 
\end{equation}

For each $(\lambda,\hbar)\in \catA_T^\vee\times E$, the function  $\theta(z+\la_{\al^\vee})/\theta(z)$ on $E$ has a pole at $0\in E$ and a zero at $-\la_{\al^\vee}\in E$, and hence is a rational section of the degree 0 line bundle coming from the Weil divisor class $[\{-\la_{\al^\vee}\}]-[\{0\}]$ on $E$. Pulling back to $\catA_T$ along the map $\chi_\al:\catA_T\to E$, we obtain the  degree 0 line bundle  $\calO(-\la_{\al^\vee} \al)$ together with a section $\frac{\theta(z_\al+\la_{\al^\vee})}{\theta(z_\al)}$ on $\catA_T$.  
Similarly, the function $\frac{\theta(z-\hbar)}{\theta(z)}$ is a section of the line bundle $\calO(\hbar  \al)$ over $\catA_T$.  One also observes that the operator $T^\la_{\al^\vee}$ satisfies all three properties of the residue definition (Definition \ref{def:res}), therefore, fixing $\la_{\al^\vee}\neq -\hbar$, then $T_\al^\la$ can be viewed as an element  of  a version of the elliptic affine Hecke algebra $\calH$ to be made precise in Definition~\ref{def:dyn_hecke}.

\begin{lemma}
Assume $\lambda\in\catA^\vee_T$ lies in the complement of  the divisor $\lambda_{\al^\vee}=-\hbar$, then 
$T^\lambda_{\alpha}$ is a well-defined morphism of sheaves \[\pi_*\mathbb L_{\lambda,\hbar}\to \pi_*\mathbb L_{\lambda-\lambda_{\alpha^\vee} \alpha,\hbar}.\]
\end{lemma}
\begin{proof}
Indeed, the first term of $T^\lambda_\alpha$ defines a rational map $\pi_*\mathbb L_{\lambda,\hbar}\to \pi_*(\mathbb L_{\lambda,\hbar}\otimes \calO(-\lambda_{\alpha^\vee} \alpha))$ with an order 1 pole at $z_\alpha=0$. By definition of the addition in $\catA_T^\vee=\Pic^0(\catA_T)$, the target of this map is $\pi_*\mathbb L_{\lambda-\lambda_{\alpha^\vee} \alpha,\hbar}$. 
Now  we consider the second term, where $\delta_\alpha:\pi_*\mathbb L_{\lambda,\hbar}\to \pi_*\mathbb L_{s_\alpha\bullet \lambda,\hbar}$ by \eqref{eq:de}. The coefficient is a rational section of $\calO(\hbar\alpha)$. So the second term defines a rational map 
\[\pi_*\bbL_{\la, \hbar}\to \pi_*\bbL_{s_\al\bullet \la, \hbar}\to \pi_*(\bbL_{s_\al\bullet \la, \hbar}\otimes \calO(\hbar \al))=\pi_*\bbL_{\la-\la_{\al^\vee} \al, \hbar}\] with an order 1 pole at $z_\alpha=0$.
The usual argument that the zeros of the numerator cancels the order one pole on the denominator at $z_\alpha=0$ now implies the statement. 
\end{proof}

Consequently, assuming $\lambda\in\catA^\vee_T$ lies in the complement of  the divisor $\lambda_{\alpha^\vee}=-\hbar$, then $T^{\dyn}_\alpha=T^\lambda_\alpha\circ s_\alpha^{\dyn}$ is a well-defined element of $\End(\oplus_{\lambda'=w^{dyn}\lambda}\pi_*\mathbb L_{\lambda',\hbar})$. Or equivalently, let $U\subseteq \catA_T^\vee/W\times E$ be the open subset where $\lambda_{\alpha^\vee}\neq-\hbar$  for any reflections $\alpha$, and let $(\pi\times\varpi)_*\bbL|_{\catA_T/W\times U}$  be the restriction of $(\pi\times\varpi)_*\bbL$. Then, we have the endomorphism $T^{\dyn}_\alpha$ for any simple root $\alpha$.

\begin{remark}
Recall that  Felder's elliptic $R$-matrices with dynamical parameters of $\mathfrak{sl}_n$  on $\bbC^n\otimes\bbC^n$ where $\bbC^n$ is the natural representation of $\mathfrak{sl}_n$ has the form 
\[
R(z,\lambda)=\sum_{i=1}^nE_{i,i}\otimes E_{i,i}+\sum_{i\neq j}a(z,\lambda_i-\lambda_j)E_{i,i}\otimes E_{j,j}+\sum_{i\neq j}b(z,\lambda_i-\lambda_j)E_{i,j}\otimes E_{j,i}.\]
Here \[a(z,\lambda)=\frac{\theta(z)\theta(z+\hbar)}{\theta(z-\hbar)\theta(\lambda)};\quad b(z,\lambda)=\frac{\theta(\hbar)\theta(z+\lambda)}{\theta(z-\hbar)\theta(\lambda)}.\]
Originally, the above definition \eqref{eqref:DL_dyn} is motivated by this formula. Here the coefficient in $T_\alpha^\la$ of $\delta_\alpha$ is $\frac{1}{a(z_\alpha,\lambda_{\alpha^\vee})}$ and the coefficient   of $\id$ is $\frac{b(z_\alpha,\la_{\al^\vee})}{a(z_\alpha,\la_{\al^\vee})}$. This $R$-matrix satisfies the dynamical Yang-Baxter equation 
\begin{small}
\[R^{12}(z_1-z_2, \lambda-\hbar \lambda(3))
R^{13}(z_1-z_3, \lambda)
R^{23}(z_2-z_3, \lambda-\hbar\lambda(1))=R^{23}(z_2-z_3, \lambda)
R^{13}(z_1-z_3, \lambda-\hbar \lambda(2))
R^{12}(z_1-z_2, \lambda).\]
\end{small}
Here we are following the convention as in \cite[\S~2]{FV2} which differs from that in \cite{F1} be a shift. We refer the readers to the former for the notations and the details. 
\end{remark}

\begin{prop}\label{prop:Tbraid}Assuming $E$ is an elliptic curve over $\bbC$, then the operators $T^{\dyn}_\alpha$ satisfy $(T^{\dyn}_\al)^2=1$, and the braid relations. 
\end{prop}
\begin{proof} The first identity follows from direct computation, using Fay's trisecant identity \cite[(10)]{RW}.

For the second part, we use the algebraic properties of the operators from \cite{RW}. In {\it loc. cit.},  the operator $C_\alpha$ from \S~5.4 acts on  the local elliptic classes defined in \S~4.2, while we would like to focus on the elliptic  classes defined in \S~3.4, and then normalized before Theorem 1.3, which is denoted by $\mathfrak{C}(X_w)$. For such classes, the operators that permute them  is the following
\[
C_\al':=\frac{\theta(z_\al+\la_{\al^\vee})\theta(\hbar)}{\theta(z_\al)\theta(\la_{\al^\vee}-\hbar)}\de_\al^{\dyn}+\frac{\theta(z_\al+\hbar)\theta(\la_{\al^\vee})}{\theta(z_\al)\theta(\hbar-\la_{\al^\vee})}\de_\al\de_\al^{\dyn},
\]
if we identify $L_\al$ with $z_\al$,  $h$ with $\hbar$, $h^{\langle\la, \al^\vee\rangle}$ with $\la_{\al^\vee}$, $s_\al^\gamma$ (shifting on the fixed points) with $\de_\al$, and $s^\mu_\al$ (action on $\bbX^*(T)\otimes_\bbZ\bbC\cong \mathfrak{t}^*$) with $\de_\al^{\dyn}$.

If we denote $\hbar'=-\hbar$, then we see that $T_\al^{\dyn}=-C_\al'|_{\hbar'=-\hbar}$. Therefore, the fact that elliptic classes satisfying braid relations implies that $C_\al'$ satisfy braid relations, then so do $T_\al^{\dyn}$. The proof is then finished.
\end{proof}

Let $U\subseteq \catA_T^\vee/W\times E$ be the open subset where $\lambda_{\alpha^\vee}\neq-\hbar$  for any positive root  $\alpha$. Let $\calS^{\dyn}$ be the sheaf of algebras $(\pi\times\varpi)_*\calO$ on $\catA_T/W\times U$, on which $W$ acts via $\de_w^{\dyn}$ for any $w\in W$. That is, for any open subset $V\subset \catA_T/W\times U$, we define
\[
\Gamma(V,\calS^{\dyn})=\Gamma((\pi\times \varpi)^{-1}(V), \calO)\to \Gamma(V,\calS^{\dyn}), \quad f\mapsto ((z,\la, \hbar)\mapsto f(z,(w^{dyn})^{-1}(\la), \hbar)).
\] Hence the twisted product $\calS^{\dyn}\rtimes W$ makes sense. The discussion in the paragraph containing \eqref{eq:sdyn} implies that $(\pi\times\varpi)_*(\id\times\varpi)^*(\id\times\varpi)_*\bbL|_U$ is a sheaf of modules over $\calS^{\dyn}\rtimes W$.
\begin{definition}\label{def:dyn_hecke}
On $\catA_T/W\times U$, we define the dynamical elliptic Hecke algebra $\calH^{\dyn}$ to be the sheaf of subalgebras of $\sEnd((\pi\times\varpi)_*(\id\times\varpi)^*(\id\times\varpi)_*\bbL)$ generated by $\calS^{\dyn}\rtimes W$ and the global sections $T_\al^{\dyn}$ for all simple roots $\al$.
\end{definition}
It follows from Proposition \ref{prop:Tbraid} that $\calH^{\dyn}$ is a free module over $\calS^{\dyn}\rtimes W$ of rank $|W|$. 

\begin{remark}We believe that Proposition \ref{prop:Tbraid} holds for elliptic curve over general base ring $R$. For example, it is possible to verify the braid relations of $T^{\dyn}_\alpha$ algebraically, 
 by direct computation based on Fay's trisecant identity, which holds if $R$ is an algebraically closed field \cite[Theorem 18.6]{Po05},  taking into account the following facts
\[
s_\al(\theta(z_\al))=\theta(z_{-\al})=-\theta(z_\al),~ s_\al^{\dyn}(\theta(\la_{\al^\vee}))=\theta(\la_{-\al^\vee})=-\theta(\la_{\al^\vee}).
\]
However, over complex numbers, 
reducing this to the properties of the $C_\alpha$-operators of Rimanyi and Weber makes the argument much easier. 
The fact that $C_\alpha'$ satisfies the braid relations is proven by Rimanyi and Weber topologically based on earlier work of Borisov and Libgober. This is another application of topology in the study of elliptic affine Hecke algebra, although of a different flavour, compared to that in \S~\ref{sec:repn}.

Observe also the symmetry between $z_\alpha$ and $\la_{\al^\vee}+\hbar$ in  the coefficients of $T^{\dyn}_\alpha$. This reflects the invariance of the operators $T^{\dyn}_\alpha$ under the Langlands duality \cite{RW2}.
\end{remark}

\subsection{The  elliptic  Demazure operators} We continue following notations from \S~\ref{sec:ell}.
Recall that $\calO(-\{0\})$ is the line bundle corresponding to  the divisor $\{0\}$ on $E$, and $\vartheta$  is the natural section of $\calO(-\{0\})^{-1}=\calO(\{0\})$. 
Moreover,  $\calL_\al$ is defined as $\calO(-D^\al)$ and  $\varth(\chi_\al)$ is the section of $\calL_\al^{-1}$, which is the pull-back of  $\vartheta$ along $\chi_\al:\catA\to E$. The Weyl group action on $\catA_T$ gives the identity $w(\calL_\la)=\calL_{w\la}$. Consider $\vartheta(\chi_\al)$ as a section of $\pi_*\calL_\al^{-1}$. Define 
\[\Delta_\alpha:=\frac{1}{\vartheta(\chi_\al)}-\frac{1}{\vartheta(\chi_\al)}s_\alpha.\] 
Then $\Delta_\alpha$ is a well defined element in \[H^0(\catA^c,\sHom(\calS, \pi_*\calL_\alpha^{-1}) |_{{\catA^c}}).\]
For $\alpha_i\in\Phi$ we write $\Delta_i=\Delta_{\al_i}$ for short.

\begin{lemma}\label{lem:XTglobal}
For any root $\alpha$,  the section 
$\Delta_\al$ of $ \sHom(\calS, \pi_*\calL_\alpha^{-1}) |_{{\catA^c}}$ extends to  a global  section of $\sHom(\calS, \pi_*\calL_\alpha^{-1}) $. 
\end{lemma}
\begin{proof}
It suffices to show that for each local section $\sigma$ of $\calS$ on an open set $U$, the  element $\Delta_\al( \sigma)$ in $H^0(U,\pi_*\calL_\alpha^{-1})$ is regular along the divisor $U\cap \pi(D^\al)$. This in turn amounts to show that on $U\cap{{\catA^c}}$, the rational section of $\pi_*\calL_\alpha^{-1}$ \[\frac{\sigma- s_\alpha(\sigma)}{\vartheta(\chi_\alpha)}\] has numerator vanishing along the divisor $\pi(D^\al)$. But this is clear from definition.
\end{proof}

\begin{lemma} \label{lem:ellipXTproperty}
The operators satisfies the following relations: 
\begin{enumerate} 
\item $\de_w \Delta_\al\de_{w^{-1}}=\Delta_{w(\al)}\in\Hom(\calS,\pi_*\calL_{w\alpha}^{-1})$;
\item for any open $U\subseteq {\catA^c}$ and $\sigma\in H^0(U, \calS)$, $\Delta_\al \sigma-s_\al(\sigma)\Delta_{\al}=\Dem_\al(\sigma)\in \Hom(\calS, \pi_*\calL_{\al}^{-1}).$
\end{enumerate}
\end{lemma}
\begin{proof}
(1) follows from the fact that $w(\vartheta(\chi_\al))=\vartheta(\chi_{w(\al)})$. (2) follows from calculation similar as in the twisted formal group algebra case. (3) follows from the fact that $\vartheta(\chi_\al)=-\vartheta(\chi_{-\al})$.
\end{proof}

\subsection{Geometric meaning of the elliptic Demazure operators}
In this subsection we are back to the assumption that $R$ is a $\bbQ$-algebra. 
Let $G$ be a connected, simply connected compact Lie group, $T$ be a maximal torus.  Note that $W$ naturally acts on $\bbX^*(T)=\La$, hence acts on $\catA_T$. The quotient $\catA_T /W$ is the moduli scheme of fiber-wise topologically trivial stable principal $G^{\alg}$-bundles over $E^\vee$. Let $\pi:\catA_T \to \catA_T /W$ be the natural projection. 

\begin{lemma}\label{lem:catA}
Let $P$ be a connected closed Levi subgroup of $G$ such that $T<P$. Let $W_P$ be the Weyl group of $P$.
Then, there are isomorphisms $\catA_P\cong\catA_G^{G/P}$, and $\catA_P\cong \catA_T/{W_P}$, making following diagram  commutative. 
\[\xymatrix{\catA_T\ar[r]\ar[d]&\catA_{P}\ar[dl]\\
\catA_G
}\]
All the maps can be identified with quotient maps by the Weyl group action.
\end{lemma}
\begin{proof}
This follow directly from Assumption~\ref{assum:G-equiv} \eqref{prop:ell_induction}. 
\end{proof}

For any simple root $\alpha$, let $P_\alpha$ be the corresponding minimal parabolic group. Let $p:G/T\to G/P_\alpha$ be the natural projection, which is a $\PP^1$-bundle. It induces a morphism $p^{\sharp}:\Ell_G^0(G/P_\alpha)\to \Ell_G^0(G/T)$, and after taking spectra, we have $p_\catA:\catA_G^{G/T}\to \catA_G^{G/P_\alpha}$. It induces a morphism $p^{\sharp}: p_{\catA}^*\calO_{\catA_G^{G/P_\alpha}}\to \calO_{\catA_G^{G/T}}$, and from Section~\ref{subsec:Thom}, we have the pushforward  $p_{\sharp}:\Theta(p)\to p_{\catA}^*\calO_{\catA_G^{G/P_\alpha}}=p_{\catA}^*\Ell^0_G(G/P_\alpha)$.
Note that from Lemma \ref{lem:catA},   $\pi:\catA\to \catA/W$ coincides with  $\catA_\phi:\catA_{T\times S^1}\to \catA_{G\times S^1}$ where $\phi:T\times S^1\to G\times S^1$. We then have
\[\calS=\pi_*\calO_{\catA}=\catA_{\phi*}\calO_{\catA_{T\times S^1}}=\catA_{\phi*}\Ell_{T\times S^1}^0(\pt)=\Ell_{G\times S^1}^0(G/T)=\calO_{\catA_G^{G/T}}.\]
Here the $S^1$-factor acts trivially on $G/T$. 

\begin{prop}
The following diagram commutes:
\[\xymatrix{
\pi_{G/T*}^{G\times S^1}\Theta(p)\ar[rr]^-{p^{\sharp}\circ p_{\sharp}}\ar[d]^{\cong}& &  \calO_{\catA_{G\times S^1}^{G/T}}\ar@{=}[d] & \\
\pi_*\calL_\alpha\ar[rr]^{\Delta_\alpha}& & \calO_{\catA_{G\times S^1}^{G/T}}\ar[r]^\cong&\calS.
}\]
\end{prop}
\begin{proof}
Without loss of generality, we assume $G=\hbox{PSU}_2\subset G^{\bbC}=\hbox{PGL}_2(\bbC)$ acts on $\Aff^2$ in the usual way, and $p:\PP(\Aff^2)\to \pt$ is the natural projection. The diagram \[\xymatrix{\PP^1\times\PP^1\ar[r]^{p_2}\ar[d]_{p_1}&\PP^1\ar[d]^p\\
\PP^1\ar[r]_{p}&\pt
}\]
is a transversal Cartesian diagram. By base change, $p^{\sharp}\circ p_{\sharp}:\Theta(p)\to \calO_{\catA_G^{\PP^1}}$ is naturally equivalent to the map 
$p_{1\sharp}p_2^{\sharp}:\Theta(p)\to \calO_{\catA_G^{\PP^1}}$. The latter is calculated in Lemma~\ref{lem:P1PullPush}. Therefore, we are done.
\end{proof}

\section{The elliptic affine Hecke algebra: the convolution construction}\label{sec:iso}
In this section, we prove the isomorphism between the elliptic affine Hecke 
algebra and  equivariant elliptic cohomology of the Steinberg variety. 
Throughout this section, $G$  is a connected, 
simply-connected, compact Lie group, with maximal torus $T$.  Let $B\subseteq 
G^{\bbC}$ be the Borel subgroup. Let $\calB= G^{\bbC}/B\cong G/T$ as variety 
over $\bbC$.

\subsection{Convolution construction of the Demazure-Lusztig operator}\label{subsec:DL_Conv}
Recall that $\calN\subseteq \fg^*$ is the cone of nilpotent matrices, and $\widetilde{\calN}:= T^*\calB$ with the moment map $\widetilde{\calN}\to \calN$ being a resolution of singularity, the Springer resolution.  The group $G^{\bbC}$ and the maximal compact subgroup $G$ acts in the obvious way on $\widetilde{\calN}$. Let $S^1$ acts on $T^*\calB$ by scaling the cotangent  fibers (weight 1 on each fiber).  

Recall that in $\calB\times\calB$, the orbits of the diagonal $G^{\bbC}$-action are in natural one-to-one correspondence with elements of the Weyl group $W$. Let $Y_\alpha$ be the orbit corresponding to the simple root $\alpha\in \Phi$ where $\Phi$ is the set of simple roots. Its closure $\overline{Y_\alpha}$ is the union of $Y_\alpha$ and $\calB_\Delta$, the latter being the diagonal. In general $\overline{Y_w}$  is singular as a variety of dimension $\dim \calB+l(w)$, called Schubert variety. Nevertheless $\overline{Y_\alpha}$  is smooth.
Let $Z_w:=T^*_{Y_w}(\calB\times\calB)$ be the conormal bundle, considered as a locally-closed Lagrangian subvariety of $\widetilde{\calN}\times\widetilde{\calN}$.  The Steinberg variety $Z:=\widetilde{\calN}\times_{\calN}\widetilde{\calN}$ has  a decomposition $Z:=\sqcup_{w\in W}Z_w$ \cite[Corollary 3.3.5]{CG}. Denote $\proj_j:Z\to \calN$ the $i$th projection, and without causing confusion, its restriction to $\overline{Z_\alpha}$ still by $\proj_i$ for $i=1,2$.  Although the closer of each $Z_w$ is singular in general, $\overline{Z_\alpha}$ is smooth with $p_\al:\overline{Z_\al}\to \overline{Y_\al}$ a line bundle, and the second projection $\proj_2:\overline{Z_\alpha}\to \widetilde{\calN}$ is proper. 
Via convolution and the embedding $\overline{Z_\alpha}\inj Z$, a local section $\eta$ of  
$\Xi_{G\times S^1}(\overline{Z_\alpha})$ defines locally a section of 
$\sEnd_{\catA_{G\times S^1}}(\calE ^0_{G\times S^1}(\widetilde{\calN}))$, which  
we denoted by $\eta*_Z\_$ following the notations as in \S~\ref{subsec:ell_coh_conv}.

We follow the same notations as in Lemma~\ref{Lem:convol_zerosec}, where $\overline p_1:\overline{Y_\alpha}\to \calB$ is the first projection $\calB\times\calB\to \calB$ restricted to $\overline{Y_\alpha}$. It is a fiber bundle with each fiber isomorphic to $\PP^1$ \cite[p395]{CG}. Let $\Omega^1_{\overline p_1}$ be the relative cotangent bundle of $\overline p_1$, which is a line bundle on $\overline{Y_\al}$. Define $\calJ_\alpha:=p_\alpha^*(\Omega^1_{\overline p_1})^{-1}$ on $Z_\alpha$.
Let $k_q$ be the natural 1-dimensional $S^1$-representation. By the Lagrangian property of $\overline{Z_\alpha}$ and the action of $G\times S^1$ on $T^*\calB$, one can see that the relative normal bundle $N_{b_2\circ j}$ of the map $b_2\circ j$ is isomorphic to $p_\alpha^*(\Omega^1_{\overline p_1})\otimes k_q\cong \calJ_\alpha^{-1}\otimes k_q$. 
We apply Lemma~\ref{Lem:convol_zerosec}, to get  $\Xi_{G\times S^1}(\overline{Z_\alpha})\cong \Theta(\calJ_\alpha^{-1}\otimes k_q)\otimes \Theta(\calJ_\alpha)^{{-1}}$.
Let $f$ be a rational function on $E$ vanishing along the zero section  of order 1. 
Define 
\begin{equation}\label{eqn:Q_alpha}
J^f_\alpha:=\frac{e(\calJ_\alpha)}{e(\calJ_\alpha^{-1}\otimes k_q)}\cdot \left(1-\frac{c^f_1( k_q)}{c^f_1(\calJ_\alpha)}\right),
\end{equation} 
which by the discussion above is a rational section of $\Xi_{G\times S^1}(\overline{Z_\alpha})$, the locus on which it is regular is determined later.
 Note that although $\overline{Z_\alpha}$'s are conormal bundles to orbits in $\calB$, Proposition~2.8.6 in \cite{GKV95} does not apply, since the $ S^1$-factor in $G\times S^1$ acts non-trivially on $T^*\calB$ but trivially on $\calB$. However, its action on $\Ell_{G\times S^1}^0(\wt\calN)$ can be calculated using Lemma~\ref{Lem:convol_zerosec}, which we carry out here. 

Note that the scheme $\catA$ from Section \ref{sec:ellHecke} coincides with $\catA_{T\times S^1}$, and we have
\[\Ell_{G\times S^1}^0(\wt\calN)\cong \Ell_{G\times S^1}^0(\calB)\cong \catA_{\phi*}\Ell_{T\times S^1}^0(\pt)=\catA_{\phi*}\calO_{\catA_{T\times S^1}}\cong   \calS,\]
where $\catA_{\phi}:\catA_{T\times S^1}\to \catA_{G\times S^1}\cong \catA_{T\times S^1}/W$  is induced by the embedding $\phi:T\times S^1\to G\times S^1$.
For any  local section $\sigma(x)$ of $\calS$, from the definition of $T_\alpha^f$ in \eqref{eq:T},  we have
\begin{equation}\label{eq:JFal}
(T^f_{-\al}-1)\cdot \sigma(x)=(\frac{f(\chi_\gamma)}{f(\chi_\alpha)}-1)(s_\alpha\sigma(x)-\sigma(x))=\vartheta(\chi_\alpha)(1-\frac{f(\chi_\ga)}{f(\chi_{\al})})(\frac{s_\al \sigma(x)}{\vartheta(\chi_{-\al})}+\frac{\sigma(x)}{\vartheta(\chi_{\al})}).
\end{equation}

\begin{theorem}\label{thm:ell_DemLusz}
With notations as above,
\[J^f_\alpha*_Z\_=(T^f_{-\al}-1)\] as rational sections of $\sEnd_{\catA_{G\times S^1}}(\calE ^0_{G\times  S^1}(\widetilde\calN))\cong \sEnd_{\catA_{G\times S^1}}(\calS)$.
\end{theorem}

The proof of this theorem is similar to that of Theorem~6.3 in \cite{ZZ14}. However, there are  differences due to the phenomenon of Thom bundles in  equivariant  elliptic cohomology. Therefore, we include the proof for convenience of the readers.

\subsection{Rank-1 case}

Now we assume $G$ has rank 1. We follow the same notations as in Lemma~\ref{lem:P1PullPush}.
The only  simple root is denoted by $\alpha$. In this case, $\calB\cong\PP^1$, and $Z_\alpha=\overline{Y_\alpha}=\PP^1\times\PP^1$. 
We identify $\PP^1$ with $\PP(\Aff^2)$. Let $T=S^1$ acts on $\Aff^2$ with weights $-\alpha/2$ and $\alpha/2$. For any character $\la$, the associated line bundle $\calL_\lambda$ is isomorphic to $\calO(\angl{\lambda,\alpha^\vee})$ on $\PP^1$, and $\Omega^1_{\PP^1}=\calO(-2)$.   The  line bundle $\calJ_\alpha=(\Omega^1_{p_1})^{-1}$ on $Z_\al$ is identified with $\calO_{\bbP^1}\boxtimes \calO_{\bbP^1}(2)$.
The map $i^{\sharp}\circ (b_{2}\circ j)_\sharp:\pi_{Z_\alpha*}\Xi_{G\times S^1}(Z_\al)\to \pi_{\calB\times\calB*}\Xi_{G\times S^1}(\calB\times\calB)$, according to Theorem~\ref{thm:thom}, is multiplication by $e(\calO_{\PP^1}\boxtimes\calO_{\PP^1}(-2) \otimes k_q)$, which is a section of $\Theta(\proj_2^*b_2^*N_2)^{-1}\cong \Theta(\calO_{\PP^1}\boxtimes\calO_{\PP^1}(-2) \otimes k_q)^{-1}$. Consequently,  we have 
\[J_\al^f=\frac{e(\calO_{\PP^1}\boxtimes\calO_{\PP^1}(2))}{e(\calO_{\PP^1}\boxtimes\calO_{\PP^1}(-2))\otimes k_q)}\left(1-\frac{c_1^f(k_q)}{c_1^f(\calO_{\PP^1}\boxtimes\calO_{\PP^1}(2)))}\right)\]
as a rational section of $\pi_{Z_\alpha*}\Xi_{G\times S^1}(Z_\al)$.

Let $p_i:\PP^1\times\PP^1\to \PP^1$ be the $i$-th projection for $i=1,2$.
The Euler class $e(\calJ_\alpha)=\vartheta(\chi_\alpha)$ is a global section of $ \pi_{\PP^1\times\PP^1*}\Xi_{G\times S^1}(\PP^1\times\PP^1)\cong  \pi_{\PP^1\times\PP^1*}\Theta(\calJ_\alpha)^{-1}$. Hence, we have the operator $e(\calJ_\alpha)*_\calB\in\End_{\catA_{G\times S^1}}(\calE ^0_{G\times S^1}(\PP^1))$.
\begin{lemma}\label{Lem:Quillen P1}
For any local section $\sigma$ of $\Ell^0_{G\times S^1}(\PP^1)$, we have
\[ e(\calJ_\al)*_\calB\sigma=\sigma-s_\al\sigma.\]
\end{lemma}

\begin{proof}
We have 
\begin{eqnarray*}
p_{1\sharp}\left(e(\calJ_\alpha)p_2^{\sharp}(\sigma)\right)
&=&
p_{1\sharp}\left(p_1^{\sharp}[e(\calO(-2))]\cdot p_2^{\sharp}(\sigma)\right)\\
&=&e(\calO(-2))\cdot p_{1\sharp}\left(p_2^{\sharp}\sigma\right)\\
&=&\vartheta(\chi_\alpha)\cdot p_{1\sharp}\left(p_2^{\sharp}\sigma\right).
\end{eqnarray*}
By Lemma~\ref{lem:P1PullPush} and Lemma~\ref{lem:tensor_c1}, $\vartheta(\chi_\alpha)\cdot p_{1\sharp}\left(p_2^{\sharp}\sigma\right)=\sigma-s_\al\sigma$. Therefore, we are done.
\end{proof}

By the projection formula and Lemma~\ref{Lem:Quillen P1}, we have the following more general formula.

\begin{lemma}\label{lem:conv}
Let $\sigma_1, \sigma_2$ be   local sections of  ${\pi_*}\calO_{E\times E}\cong \calE ^0_{G\times S^1 }(\PP^1)$. Then we have the following identity:
\[\left(e(\calJ_\alpha)\cdot p_1^{\sharp}(\sigma_1)\right)*_\calB\sigma_2=\sigma_1\cdot\left(\sigma_2-s_\alpha\sigma_2\right).\]
\end{lemma}

\begin{prop}\label{prop:Q_alpha}
Under the identification $\calS\cong {\pi_*}\calO_{E\times E}\cong \calE ^0_{G\times S^1 }(\PP^1)$, the operator $J^f_\al*_Z$ on any local section $\sigma$ of $\calE ^0_{G\times S^1 }(\PP^1)$ is given by  $T^{f}_{-\al}-1$.
\end{prop}

\begin{proof}
By definition and Lemma~\ref{Lem:convol_zerosec}, 
\[J^f_\alpha*_Z \sigma=p_{1\sharp}\left[\left(p_2^{\sharp}(\sigma)\right)\cdot e((\Omega^1_{\PP^1} \otimes k_q)\boxtimes\calO)\cdot J^f_\alpha\right]\in \calE ^0_{G\times S^1}(\bbP^1).\]
Recall that $G=\hbox{SU}_2$, by Lemma~\ref{lem:conv},
we have
\begin{eqnarray*}
p_{1\sharp}\left(\left(p_2^{\sharp}(\sigma)\right)\cdot e((\Omega^1_{\PP^1}\otimes k_q)\boxtimes\calO )\cdot J^f_\alpha\right)
%&=&p_{2*}\left(p_1^*\left((\sigma)c_1(\Omega^1_{\PP^1})c_1(\Omega^{1\vee}_{\PP^1})\right)\right)\\
&=&p_{1\sharp}\left(p_2^{\sharp}\left(\sigma\right)\cdot e(\calJ_\alpha)\cdot p_1^{\sharp}\left(1-\frac{c^f_1( k_q)}{c^f_1(\calJ_\alpha)}\right)\right)\\
&=&\left(1-\frac{f(\chi_\gamma)}{f(\chi_\alpha)}\right)\left(-s_\alpha\sigma+\sigma\right)\\
&=&\left(\frac{f(\chi_\gamma)}{f(\chi_\alpha)}-1\right)\left(s_\alpha\sigma-\sigma\right).
\end{eqnarray*}
Comparing with \eqref{eq:JFal} we know that the effect of $J^f_\al$ on any local section $\sigma$ coincides with that of $T^{f}_{-\al}-1$, so the conclusion follows.
\end{proof}
Proposition~\ref{prop:Q_alpha}
 together with Lemma~\ref{Lem:convol_zerosec} proves  Theorem~\ref{thm:ell_DemLusz}. 

\subsection{Convolution algebra of the Steinberg variety}\label{subsec:conv_Steinberg}
 The goal of this subsection is to prove the following.
\begin{theorem}\label{thm:main}
There is an isomorphism $\Upsilon:\calH\cong \pi_{Z*}\Xi_{G\times S^1}(Z)$ of sheaves of algebras on $\catA_{G\times S^1}$ extending the assignment sending the rational section $T_{-\al}^f-1$   (defined in \eqref{eq:T}) of $\calH$ to the rational section $J_\alpha^f$ of $\Xi_{G\times S^1}(Z)$ for any simple root $\alpha$ and any rational function $f$ of $E$.
\end{theorem}
  The proof goes essentially the same way as its classical counterpart in \cite{CG}. Nevertheless, the fact that $\calH$ has very few global sections makes differences. We outline the proof in the rest of this section, with emphasis on the parts that are different. For simplicity, in the remaining part of this section, we will omit the subscript $G\times S^1$ in $\Xi_{G\times S^1}(Z)$.

Recall that for any $w\in W$, $Y_w\subseteq \calB\times\calB$ be the $G^{\bbC}$-orbit corresponding to $w\in W$, and $Z_w=T^*_{Y_w}(\calB\times\calB)$. Let  $Z_{\leq w}=\sqcup_{v\le w}Z_v$ and $Z_{< w}=\sqcup_{v< w}Z_v$, and $i_w$, $i_{\leq w}$ and $i_{<w}$ be the natural embeddings into $Z$. 
Using the long exact sequence of cohomology and the vanishing of odd degree cohomology, we see that the induced map  $i_{\leq w\sharp}$ on cohomology is injective. Similarly, we have the short exact sequence $0\to \pi_{Z_{< w}*}\Xi(Z_{<w})\to\pi_{Z_{\leq w}*}\Xi(Z_{\leq w})\to \pi_{Z_w*}\Xi(Z_{w})\to 0$. Note here that $\Xi(Z_{w})$ is a line bundle on $\catA_{G\times S^1}^{Z_w}$. 

\begin{lemma}\label{lem:Xi_free}
The sheaf $\pi_{Z*}\Xi(Z)$ is locally free, and the action of $\pi_{Z*}\Xi(Z)$ on $\calE^0_{G\times S^1}(T^*\calB)$ is faithful, i.e., the morphism $\pi_{Z*}\Xi(Z)\to \sEnd_{\catA_{G\times S^1}}(\calE^0_{G\times S^1}(T^*\calB))$ is injective.
\end{lemma}
\begin{proof}
It can be proved exactly the same way as \cite[Lemma~4.6.1]{GKV95}.
\end{proof}

By definition of $\calH$, the natural morphism $\calH\to \sEnd_{\catA/W}(\calS)$ is injective. Therefore, we get the following consequence of Lemma~\ref{lem:Xi_free}.
\begin{corollary}
The assignment in Theorem~\ref{thm:main} extends to a well-defined morphism $\Upsilon:\calH\to \pi_{Z*}\Xi(Z)$ of coherent sheaves of algebras, which is injective.
\end{corollary}

Now we finish the proof of Theorem~\ref{thm:main}.
\begin{proof}[Proof of Theorem~\ref{thm:main}]
Clearly $\Upsilon$ is filtration preserving, with the filtration of the domain given by \S~\ref{subsec:filH}  and that of the codomain defined as above. Hence, $\Upsilon$ induces a morphism $\gr_w\Upsilon$ on associated graded piece $\calH_w\to \pi_{Z_w*}\Xi(Z_w)$ for any $w\in W$, which is a morphism of rank 1 locally free sheaves of module over $\calS$.

Let $I_w=(i_1,...,i_r)$ be a reduced sequence of $w$. For any $p\in \catA$, we fix a collection of rational functions $\{f_{\al_{i_j}, p}\mid j=1,...,r\}$, and an open neighbourhood  $U$ of $p$ as in \S~\ref{subsec:filH}. 
For each simple root $\al$, we say a local section $s$ of $\Xi(Z_\alpha)$  is invertible if there is a local section $s'$ of $\Xi(Z_\alpha)^{-1}$ on the same open set such that $s\otimes s'$ is $1\in \calS$.  In particular,  on $U$ the element $\frac{e(\calJ_\alpha)}{e(\calJ_\alpha^{-1}\otimes k_q)}\cdot \left(1-\frac{c^{f_{\al_i,p}}_1( k_q)}{c^{f_{\al_i,p}}_1(\calJ_\alpha)}\right)$ is invertible as a section of $\Xi(Z_{\alpha_i})$.

Following the proof of \cite[Proposition 7.6.12.(2)]{CG}, we define $\proj_{j,j+1}:(T^*\calB)^{r+1}\to T^*(\calB\times\calB)\cong T^*(\calB)\times T^*(\calB)$ to be the projection to the $(j,j+1)$ factor, and define $\calZ_{i_j}=\proj_{j,j+1}^{-1}(T^*_{Y_{s_{i_j}}}(\calB\times\calB))$, for $j=1,\dots,r-1$.  The projection
\[\calZ_{{i_1}}\times_{\tilde{\calN}} \calZ_{{i_2}}\times_{\tilde{\calN}}\cdots\times_{\tilde{\calN}}\calZ_{{i_r}}\to Z_w\] restricts to an isomorphism \[\calZ_{i_1}\cap\calZ_{i_2}\cap\cdots\cap\calZ_{i_r}\cong T^*_{\calY_w}(\calB\times\calB).\]
Consequently, the section $J_{I_w}:=J^{f_{\alpha_{i_1},p}}_{\alpha_{i_1}}\cdots J^{f_{\alpha_{i_r},p}}_{\alpha_{i_r}}$ is invertible in \[\Xi(\calZ_{i_1}\cap\calZ_{i_2}\cap\cdots\cap\calZ_{i_r})|_U\cong\Xi(Z_w)|_U.\]
By the same argument as \cite[Theorem~7.6.12.]{CG}, when restricted to $Z_w$, the convolution of $J_{I_w}$ is equal to $\gr_w\Upsilon (T_{I_w})$ (where $T_{I_w}$ is as in Lemma~\ref{lem:basis}). Hence, the restriction of $\Upsilon$ is a morphism $\calH_w|_U\to \Xi(Z_w)|_U$ that sends the generator $T_{I_w}$ of the domain to the generator $J_{I_w}$ of the codomain. So  $\gr_w\Upsilon$ is an isomorphism of  sheaves of modules over $\calS$. By the same argument as \cite[Proposition~2.3.20.(ii)]{CG}, the fact that $\gr\Upsilon$ is an isomorphism implies that $\Upsilon$ is an isomorphism.
\end{proof}
\subsection{Remarks on the integral form}\label{subsec:existence}
The algebra $\calH$ is well-defined for any elliptic curve on an integral domain $R$. All the properties in \S~\ref{sec:ellHecke} hold in this generality. The existence of an  elliptic cohomological construction  of $\calH$ in Theorem~\ref{thm:main} is proven under the Assumption \ref{assum:G-equiv}, which is known to be satisfied in many cases (see e.g., \S~\ref{subsec:conv_spectra} for a detailed discussion) and expected to hold in a greater generality. 
Nevertheless,  $\pi^*\calH$ as a sheaf of algebras on $\catA_T\times E$ always admits an elliptic cohomological construction without Assumption \ref{assum:G-equiv}.

By the discussion in \S~\ref{subsec:conv_spectra}, there exists a coherent sheaf $\Xi_T(Z)^{der}$ over $\catA_T^{der}$, and a well-defined map $\Xi_T(Z)^{der}\otimes \Xi_T(Z)^{der}\to \Xi_T(Z)^{der}$, the $\pi_0$ of which gives $\pi^*\calH$ and its multiplication, where $\pi:\catA_{T}\times E\to \catA$ is the projection. The associativity of the map of sheaves of spectra, however, is obvious only up to homotopy.
It is unknown to us whether this map can be lifted to an $A_\infty$-ring structure. 

For any central character of $\calH$, that is, a closed point $p:\Spec k\to \catA$, and any lifting $p':\Spec k\to \catA_T$, by the flatness of $\calH$ as coherent sheaf, the algebras $\calH\otimes_pk$  and $(\pi^*\calH)\otimes_{p'}k$ are isomorphic, and hence have equivalent categories of representations. 
In the next section we study the representation theory of the former using that of the latter.

From next section on, we only consider the algebra $\calH$ over a complex elliptic curve. We consider this as an illustration that the representation theory of the elliptic Hecke algebra is interesting even over the complex numbers. It is equally interesting, although beyond the scope of this paper, to consider the representation theory over a field of positive characteristic, a question well-posed due to the discussion above. 

\section{Geometric construction of representations at non-torsion points}\label{sec:repn}
In this section, we assume $E$ is an elliptic curve over $\bbC$ (although all the results in this section are true if more generally the base  field is characteristic zero). 
Again, $G$ is a connected, simply-connected, compact Lie group. The Borel-Moore homology we consider are with complex coefficients. 

As $\calH$ is a coherent sheaf of algebras on the Noetherian scheme $\catA/W$, any irreducible representation of $\calH$ is supported on a closed point of $\catA/W$. We generalize Kazhdan--Lusztig's classification of irreducible representations of affine Hecke algebra \cite{KL} to the case of elliptic affine Hecke algebra at closed points of $\catA/W$ whose $\gamma$-coordinate is non-torsion.

\begin{definition}
A representation of $\calH$ is a coherent sheaf on $\catA/W$ which  is a sheaf of modules over the coherent sheaf of algebras $\calH$.\end{definition}
In particular, the structure of coherent sheaf on $\catA/W$ is required to coincide with the action of the subsheaf $\calO_{\catA/W}\inj \calH$.
\subsection{Reminder on the decomposition theorem}\label{subsec:DecompThm}
We  recall briefly some basic facts about representations of convolution algebras, following \cite[\S~8.6]{CG}.

Let $f:M\to N$ be a projective morphism of quasi-projective complex varieties, with $M$  smooth. Let $D^b_c(N)$ be the derived category of constructible sheaves on $N$. Let $X=M\times_NM$. 
Then $H_*(X;\bbC)$ is endowed with a convolution product, and there exists an isomorphism $H_*(X;\bbC)\cong\Ext^*_{D^b_c(N)}(f_*\bbC_M, f_*\bbC_M)$ of  associative algebras. Applying the decomposition theorem \cite{BBD} to $f_*\bbC_M$, we get \[f_*\bbC_M\cong \bigoplus_{\phi,k}L_{\phi,k}\otimes IC_\phi[k],\] where $k$ runs through $\bbZ$, and $\phi$ runs through the set of isomorphism classes of simple perverse sheaves. 
Consequently, $\{L_{\phi}:=\oplus_kL_{\phi,k}\mid L_{\phi}\neq0\}$  is a complete set of pair-wise non-isomorphic simple modules over $H_*(X;\bbC)$ (see, e.g., \cite[Theorem~8.6.12]{CG}). 

If furthermore $f$ is a $G$-equivariant map between $G$-varieties, such that $N$ has only finitely many orbits, then we can label the set of simple perverse sheaves $\phi$ in the decomposition above  by pairs $(\bbO,\chi)$, where $\bbO$ is an orbit in $N$, and $\chi$ is an equivariant local system on $\bbO$ (see, e.g., \cite[Theorem~8.4.12]{CG}). Choosing a base point $x_\bbO$ for each orbit $\bbO$, and writing its isotropy subgroup as $G_{x_\bbO}$, then the local system $\chi$ is identified with an irreducible representation of the component group $G(x_\bbO)/G(x_\bbO)^0$ (see, e.g., \cite[8.4.13.(ii)]{CG}).

In the equivariant set-up, for each $\bbO$, let $x\in\bbO$ and let $H_*(M_x)$ be $H^*(i_x^!f_*\bbC_M; \bbC)$ , where $i_x:\{x\}\inj N$ is the embedding. We define the {\it standard modules} over $H^*(X;\bbC)$ associated to $\bbO$ as  $H_*(M_x)$
with the $H^*(X)$-module structure.
For each $\phi=(\bbO_\phi,\chi_\phi)$, let $H_*(M_x)_\phi$ be the $G(x_\bbO)/G(x_\bbO)^0$-isotropical component in $H_*(M_x)$ that transforms as the irreducible $G(x_\bbO)/G(x_\bbO)^0$-module $\chi_\phi$. Note that $H_*(M_x)_\phi$ is a submodule of $H_*(M_x)$ over the convolution algebra. We have by decomposition theorem an isomorphism
\[H_*(M_x)\cong \bigoplus_{\phi}L_\phi\otimes H^*(i^!_xIC_\phi),\]  which does not depend on the choice of $x_\bbO$.
Therefore, the top of standard module $H_*(M_x)_\phi$  is  irreducible, which is isomorphic to the irreducible module $L_{\phi,0}$.
For any two parameters $\psi=(\bbO_\psi,\chi_\psi)$ and $\phi=(\bbO_\phi,\chi_\phi)$, the multiplicity of the simple object $L_{\phi,0}$ in the standard module $H_*(M_x)_\psi$ is given by 
\[[H_*(M_x)_\psi:L_\phi]=\sum_k\dim H^k(i^!_x IC_\phi)_\psi\] (see, e.g., \cite[Theorem~8.6.23]{CG}). Here the simple perverse sheaf $IC_\phi$ is the intersection cohomology sheaf  associated to the local system $\chi_\phi$ on $\bbO_\phi$. 
We follow the standard notation $[M:S]$ for the multiplicity of an irreducible object $S$ in an object $M$ in any abelian category. 

\subsection{Pontryagin duality}\label{subsec:Pontryagin}
We recall some well-known facts about the moduli space $\catA_G$. For a simply-connected compact Lie group $G$, for any pair of commuting elements $s_1,s_2\in G$, there exists a maximal torus $T<G$ and some $g\in G$ such that 
$g\cdot s_i\cdot g^{-1}\in T$, $i=1,2$.  For  a fixed maximal torus $T<G$, for any two pairs of elements  $(g_1,g_2),(h_1,h_2)\in T^{2}$, if there is  $g\in G$ such that $g\cdot g_1\cdot g^{-1}=h_1$ and $g\cdot g_2\cdot g^{-1}=h_2$, then this $g$ can be chosen from the normalizer of $T$ in $G$.
It is well-known that a closed point in $\catA_G$ corresponds to an ordered  pair of semi-simple elements in $T$, up to simultaneous conjugation. Equivalently, we have the diffeomorphism $\catA_G\cong (T\times T)/W$. In particular, when $G=T$, this yields $\catA_T\cong T^2$.

This isomorphism can be made more explicit, after fixing an isomorphism $E\cong S^1\times S^1$. By definition of $\catA_T$, any closed point $a\in \catA_T$ defines a homomorphism of abelian groups $a:\bbX^*(T)\to E$,  hence defines an element in $\Hom(\bbX^*(T),S^1\times S^1)\cong T^2$, denoted by $DD(a)$. This defines a map \[DD:\catA_T\to T^2.\]

Recall that $T(a)=\cap_{a\in \catA_{T'}}T'$ for  any closed point $a\in \catA_T$. 
\begin{lemma} \label{lem:DDclosedgp}
Let $a\in \catA_T$ be any closed point, and let $DD(a)=(s_1,s_2)\in  T^{2}$. Then $T(a)\subseteq T$ is the minimal closed subgroup of $T$ containing $s_1$ and $s_2$.
\end{lemma}
\begin{proof}
Closed subgroups of $T$ form a lattice under inclusions, and so do  subgroups of $\bbX^*(T)$. There is an order-reversing one-to-one correspondence between these two lattices, sending any $T'<T$ to the kernel of the quotient $\bbX^*(T)\surj \bbX^*(T')$.

Let $T'$ be a closed subgroup of $T$. Then  $a\in \catA_{T'}\subseteq \catA_T$ if and only if the morphism $a: \bbX^*(T)\to E$ factors through $\bbX^*(T)\surj \bbX^*(T')$.  In turn, this happens if and only if $DD(a)$ is contained in the subgroup $T'\times T'\subseteq T\times T$. Therefore, the smallest closed subgroup $T'<T$ with the property that $a\in \catA_{T'}\subseteq \catA_T$ is also the smallest closed subgroup of $T'<T$ with the property that $DD(a)\in T'\times T'$. %The later is the closed subgroup of $T$ generated by $s_1$ and $s_2$.
\end{proof}

Let $T^{\bbC}$ be the corresponding algebraic torus containing $T$ as a maximal compact subgroup. Clearly, $\Hom( \bbX^*(T),\bbC^*)\cong T^{\alg}$, and $\Hom( \bbX^*(T),S^1)\cong T$. Intersection with $T$ defines an inclusion-preserving one-to-one correspondence between  algebraic subgroups of $T^{\bbC}$ and closed subgroups of $T$. In particular, for any subset $Z\subseteq T$, the smallest algebraic subgroup of $T^{\bbC}$ containing $Z$ corresponds to the closed subgroup of $T$ generated by $Z$. Composing $DD:\catA_T\to T^2$ with the natural embedding  $T^2\subseteq (T^{\bbC})^2$, we get a map $\catA_T\to  (T^{\bbC})^2$, which, without causing confusion, is also denoted by $DD$.

\subsection{A non-vanishing theorem} \label{subsec:nonvan}
In this subsection, we fix a quintuple $(s_1,q_1,s_2,q_2,x)$, where $s_i\in T$, $q_i\in \bbC^*$, and $x\in \calN$ such that $s_ixs_i^{-1}=q_ix$ for $i=1,2$. We assume that $q_1$ and $q_2$ are not simultaneously roots of unity. Without loss of generality, we assume $q_1$ has infinite order. Let $u=e^x\in G^{\bbC}$. In this subsection we prove a non-vanishing theorem (Proposition \ref{prop:non-zero_multiplicity}). The proof is similar to \cite[\S~7.1]{KL} (see also \cite[\S~8]{CG}). 

Following \cite[\S~7.1]{KL}, we fix (not-necessarily continuous) group homomorphisms $v_i:\bbC^*\to \bbR$ (with the codomain the additive group) for $i=1,2$ such that $v_i(q_i)>0$. The existence of such $v_i$ follows from \cite[Lemma~8.8.12]{CG}. Also fix $\phi:\SL_2(\bbC)\to G^{\bbC}$ such that $\phi(\begin{smallmatrix}1&1\\0&1\end{smallmatrix})=u$. For $i=1,2$, let $s_i^\phi=\phi(\begin{smallmatrix}q_i&0\\0&q_i^{-1}\end{smallmatrix})$, and let $s'_i=(s_i^{\phi})^{-1}\cdot s_i$. One can easily check that $s'_i$ commutes with $\phi(\SL_2(\bbC))$ for $i=1,2$, and $s'_1$ commutes with $s'_2$. 
Then, there is a decomposition of $\fg$ into simultaneous eigenspaces of $s'_1$, $s'_2$, and $\phi(\begin{smallmatrix}z&0\\0&z^{-1}\end{smallmatrix})$ for any $z\in \bbC^*$ as follows
\begin{equation}\label{eq:eigen}\fg=\bigoplus_{\alpha_1\in \bbC^*,\alpha_2\in \bbC^*,j\in\bbZ}\fg_{\alpha_1,\alpha_2,j},
\end{equation} where \[ \fg_{\alpha_1,\alpha_2,j}=\{y\in\fg\mid s'_i\cdot y\cdot (s'_i)^{-1}=\alpha_i y\hbox{ for }i=1,2,\hbox{ and }\phi(\begin{smallmatrix}z&0\\0&z^{-1}\end{smallmatrix})\cdot y\cdot\phi(\begin{smallmatrix}z&0\\0&z^{-1}\end{smallmatrix})^{-1}=z^jy\}.\]

Define $\fp_i=\bigoplus_{v_i(\alpha_i)\leq 0}\fg_{\alpha_1,\alpha_2,j}$ for $i=1,2$, and define $\fp=\fp_1\cap\fp_2$. Let the parabolic subgroup of $G$ corresponding to $\fp_i$ be $P_i$, and let $P$ be that corresponding to $\fp$. Let $L_i$ be the Levi subgroup of $P_i$, and $L$ be that of $P$.

Let  $G(s_1,s_2)$ (resp. $G(s_1,s_2,x)$) be the simultaneous centralizer of $s_1$ and $s_2$ (resp. $s_1,s_2$ and $x$) in $G^{\bbC}$,  and let $\fg(s_1,s_2)$ (resp. $ \fg(s_1,s_2,x)$) be its Lie algebra. Define $\fg^{s_1,s_2}=\{y\in\fg\mid s_i\cdot y\cdot s_i^{-1}=q_iy\hbox{ for }i=1,2\}$.  We write $\fg(s_1,s_2)\cap\fp$ as $\fp(s_1,s_2)$, and  $\fp\cap\fg^{s_1,s_2}$as $\fp^{s_1,s_2}$.

The following lemma is essential in the proof of the non-vanishing theorem. Its proof is almost the same as that of \cite[Lemma~8.8.22]{CG}. The key assumption is that $v_1(q_1)$ is strictly positive.

\begin{lemma}\label{lem:decomp_g_p}
The followings hold
\begin{enumerate}
\item $\fg(s_1,s_2,x)\subseteq \underset{\alpha_1=q_1^{-j},\alpha_2=q_2^{-j},j\geq0}\bigoplus\fg_{\alpha_1,\alpha_2,j};$
\item $ \fp(s_1,s_2)=\underset{\alpha_1=q_1^{-j},\alpha_2=q_2^{-j},j\geq0}\bigoplus\fg_{\alpha_1,\alpha_2,j};$
\item $ \fp^{s_1,s_2}=\underset {\alpha_1=q_1^{2-j},\alpha_2=q_2^{2-j},j\geq2}\bigoplus\fg_{\alpha_1,\alpha_2,j};$
\item $[x,\fp(s)]=\fp^{s_1,s_2}$.
\end{enumerate}
\end{lemma}
\begin{proof}
Recall that $s'_i$ commutes with the image of $\phi$ for $i=1,2$. Hence, the decomposition  \eqref{eq:eigen} induces  decompositions of $\fg(s_1,s_2,x)$, $\fp(s_1,s_2)$, and $\fp^{s_1,s_2}$. 

For part (1), note that for any $y\in \fg(s_1,s_2,x)\cap\fg_{\alpha_1,\alpha_2,j}$, we have 
\[y=s_i\cdot y\cdot s_i^{-1}=s_i^\phi\cdot \big(s'_i\cdot y\cdot (s'_i)^{-1}\big)\cdot (s_i^\phi)^{-1}=\alpha_i q_i^jy.
\]
 Hence $\alpha_i=q_i^{-j}$. 
Recall that $x\in \fg$ is nilpotent, and that $u=e^x$  extends to the $\SL_2$-triple $\angl{u,\phi(\begin{smallmatrix}z&0\\0&z^{-1}\end{smallmatrix}),\phi(\begin{smallmatrix}1&0\\1&1\end{smallmatrix})}$.
 Let $G(u)$ be the centralizer of $u$, and $\fg(u)$ its Lie algebra.  By \cite[Corollary~3.7.11]{CG}, the eigenvalues of $\phi(\begin{smallmatrix}z&0\\0&z^{-1}\end{smallmatrix})$ on  $ \fg(s_1,s_2,x)\subseteq\fg(u)$ are all non-negative. Therefore, if $ \fg(s_1,s_2,x)\cap\fg_{\alpha_1,\alpha_2,j}$ contains $y$, then we have $j\geq0$.

Similarly, considering part (2),  for $y\in \fp(s_1,s_2)\cap \fg_{\alpha_1,\alpha_2,j}$, we also have $\alpha_i=q_i^{-j}$ for $i=1,2$. We have $v_i(\alpha_i)=-jv_i(q_i)$ for $i=1,2$. By definition of $\fp$, we have $v_i(\alpha_i)\leq0$ for $i=1,2$. The condition $v_1(q_1)>0$ then implies that  $j\geq0$.

The  decomposition in (3) is proved in the same way as (2).

To prove (4), we only need to show the surjectivity of the adjoint action of $x$ considered as a  map
\[\ad x:\bigoplus_{\alpha_1=q_1^{-j},\alpha_2=q_2^{-j},j\geq0}\fg_{\alpha_1,\alpha_2,j}\to\bigoplus_{\alpha_1=q_1^{2-j},\alpha_2=q_2^{2-j},j\geq2}\fg_{\alpha_1,\alpha_2,j},\]
which is equivalent to the surjectivity of 
\[\ad x:\fg_{q_1^{-j},q_2^{-j},j}\to\fg_{q_1^{-j},q_2^{-j},j+2}, j\geq0.\]
Indeed, the operator $\ad x$ extends to an  $\fs\fl_2$-action,  which is the image of $\phi$. According to the general theory of $\fs\fl_2$-representation, $\ad x$ is surjective on positive weight spaces.
\end{proof}

Let $\calB^{s_1,s_2}_x$ be the variety of Borel subgroups $B$ such that $s_1,s_2\in B$ and $x\in \Lie B$, and let $\calB^{P,s_1,s_2}_x=\{B\in\calB_x^{s_1,s_2}\mid B\subseteq P\}$. Here and below we use $\Lie H$ to denote the Lie algebra of a complex algebraic group. Since any solvable subgroup of $G^{\bbC}$ is contained in at least one Borel subgroup, we have  $\calB^{P,s_1,s_2}_x\neq\emptyset$.

Let $\bbO\subseteq \calN$ be the $G(s_1,s_2)$-orbit of $x\in\calN$, and let $\overline{\bbO}$ be the closure of $\bbO$. The following lemma is an analogue of \cite[Proposition~8.8.19]{CG}.

\begin{lemma}
Let $B\in \calB^{P,s_1,s_2}_x$ be a Borel subalgebra, and let $\fn$ be the nil-radical of $\Lie B$. Then \[G(s_1,s_2)\cdot (\fn\cap\fg^{s_1,s_2})=\overline{\bbO}.\]
\end{lemma}
\begin{proof}
By definition, we have $x\in\fg^{s_1,s_2}$ and $x\in\fn$. Therefore, to prove this lemma, we only need to show that $G(s_1,s_2)x$ is dense in $G(s_1,s_2)\fp^{s_1,s_2}$.

Note that $x\in\fp^{s_1,s_2}$, and that $\fp^{s_1,s_2}$ is stable under the action of $P(s_1,s_2):=P\cap G(s_1,s_2)$. We have a proper map \[G(s_1,s_2)\times_{P(s_1,s_2)}\fp^{s_1,s_2}\surj G(s_1,s_2)\cdot \fp^{s_1,s_2}.\]
To show that $G(s_1,s_2)x$ is dense in $G(s_1,s_2)\fp^{s_1,s_2}$, we only need to show that the $P(s_1,s_2)$-adjoint orbit of $x$ is dense in $\fp^{s_1,s_2}$. This in turn follows directly from Lemma~\ref{lem:decomp_g_p} and \cite[Lemma~1.4.12]{CG}.
\end{proof}

Let $\widetilde{\calN}^{(s_1,q_1),(s_2,q_2)}$ be the subset of $\widetilde{\calN}$ consisting of elements fixed by both $(s_1,q_1)$ and $(s_2,q_2)$. Here the action of $(s,q)\in G^{\bbC} \times  \bbC^* $ on $n\in\widetilde{\calN}$ is given by $q^{-1}\cdot( s\cdot n\cdot s^{-1})$. Define $\hat{\bbO}$ to be the union of connected components of $\widetilde{\calN}^{(s_1,q_1),(s_2,q_2)}$ that intersect non-trivially with  $\calB^{P,s_1,s_2}_x$. Then $\hat{\bbO}$ is a $G(s_1,s_2)$-stable subvariety  of $\widetilde{\calN}^{(s_1,q_1),(s_2,q_2)}$,  and is both open and closed.

So far we have proved the following lemma, which is an analogue of \cite[Theorem~8.8.1]{CG}.

\begin{lemma}
Let $\mu:\wt\calN\to \calN$ be the canonical map. With notations as above, we have  $\mu(\hat{\bbO})=\overline{\bbO}$.
\end{lemma}

Let $L(s_1,s_2,x)$ be the simultaneous centralizer of $s_1,s_2,x$ in L.  Let $C(s_1,s_2,x)$ be the component group of $G(s_1,s_2,x)$, i.e., the quotient of $G(s_1,s_2,x)$ by its connected component containing the identity. Then,  $C(s_1,s_2,x)$ acts on $H^*(\calB_x^{s_1,s_2})$ and  $H^*(\widehat{\calB_x^{s_1,s_2}})$, where $\widehat{\calB_x^{s_1,s_2}}=\calB_x^{s_1,s_2}\cap \hat{\bbO}$.

\begin{lemma} 
\label{lem:component_surjective} We have \begin{enumerate}
\item $G(s_1,s_2,x)\subseteq P$;
\item the map of component groups $L(s_1,s_2,x)/L(s_1,s_2,x)^0\to C(s_1,s_2,x) $ is surjective.
\end{enumerate}
\end{lemma}
\begin{proof}
For $i=1,2$, let 
\[
M(s_i,x)=\{(g,q)\in G\times \bbC^*\mid gs_ig^{-1}=s_i, gxg^{-1}=q^2x\},
\] and let  $M(s_1,s_2,x)=M(s_1,x)\cap M(s_2,x)$. It is shown in 
\cite[Lemma~7.2(a)]{KL} 
(see also \cite[Lemma~8.8.23]{CG}) that $M(x,s_i)\subseteq P_i\times\bbC^*$ for $i=1,2$. Hence, $M(x,s_1)\cap M(x,s_2)=M(x,s_1,s_2)\subseteq P\times  \bbC^* =(P_1\times \bbC^* )\cap(P_2\times \bbC^* )$.
In particular, $G(s_1,s_2,x)\subseteq P$.

Claim (2) follows similarly as \cite[Lemma~8.8.25]{CG}.
\end{proof}

The following non-vanishing theorem is an analogue of \cite[Proposition~8.8.2]{CG}. For  convenience of the readers, we include the proof.
\begin{prop}\label{prop:non-zero_multiplicity} Any simple $C(s_1,s_2,x)$-module having non-zero multiplicity in $H^*(\calB_x^{s_1,s_2})$   also has non-zero multiplicity in $H^*(\widehat{\calB_x^{s_1,s_2}})$.
\end{prop}
\begin{proof}
Let $Z^0(L)$ be the identity component of the center of $L$.
We have \[H^*(\calB_x^{s_1,s_2})\cong H^*((\calB_x^{s_1,s_2})^{Z^0(L)}).\]
In particular, for any simple $C(x,s_1,s_2)$-module $\chi$, the multiplicity $[H^*(\calB_x^{s_1,s_2}):\chi]$ is non-zero if and only if $[H^*((\calB_x^{s_1,s_2})^{Z^0(L)}):\chi]$ is non-zero.

Let $\calB^P$ be the variety  consisting of  all the Borel subgroups contained in $P$. By \cite[(8.8.28)]{CG}, the flag variety $\calB(L)$ of $L$ is isomorphic to $\calB^P$. Also, by \cite[Proposition~8.8.2]{CG}, each connected component of $\calB^{Z^0(L)}$ is $L$-equivariantly isomorphic to $\calB(L)$. This in turn implies that  $(\calB_x^{s_1,s_2})^{Z^0(L)}$ is a disjoint union subsets, each piece a union of connected components, $L(s_1,s_2,x)$-equivariantly isomorphic to $\calB(L)_x^{s_1,s_2}\cong \calB^{P,s_1,s_2}_x$.
Therefore, $H^*((\calB_x^{s_1,s_2})^{Z^0(L)})\cong H^*(\calB^{P,s_1,s_2}_x)^{\oplus m}$ for some $m>0$ as $L(s_1,s_2,x)/L(s_1,s_2,x)^0$-modules. 

By Lemma \ref{lem:component_surjective}, for any simple $C(s_1,s_2,x)$-module $\chi$ such that $[H^*((\calB_x^{s_1,s_2})^{Z^0(L)}):\chi]\neq 0$, we also have $[H^*(\calB^{P,s_1,s_2}_x):\chi]\neq0$. 
Recall that by definition we have $\calB^{P,s_1,s_2}_x\subseteq \widehat{\calB_x^{s_1,s_2}}$, and hence $(\widehat{\calB_x^{s_1,s_2}})^{Z^0(L)}$ contains $\calB^{P,s_1,s_2}_x$ as a union of connected components. 
Therefore, 
 we then also have $[H^*(\widehat{\calB_x^{s_1,s_2}})^{Z^0(L)}):\chi]\neq0$. 
This in turn implies that $[H^*(\widehat{\calB_x^{s_1,s_2}}):\chi]\neq0,$ which finishes the proof.
\end{proof}

\subsection{Classification of irreducible representations at non-torsion points}\label{subsec:irrep}
Let $(a,t)\in\catA_T\times E$ be a closed point, such that $t\in E$ is not a torsion point. 

Let $DD(a,t)=((s_1,q_1),(s_2,q_2))\in (T\times S^1)^2$ be as defined in \S~ \ref{subsec:Pontryagin}. The condition that $t$ is non-torsion is equivalent to saying that $q_1$ and $q_2$ are not simultaneously roots of unity. Let $T(a,t)<T\times S^1$ be the closed subgroup generated by $(s_1,q_1)$ and $(s_2,q_2)$.
Let $x\in \calN^{T(a,t)}$ be a nilpotent element fixed by the subgroup $T(a,t)$. Then, we have $s_ixs_i^{-1}=q_ix$ for $i=1,2$. 
Let $T(a)<T$ be defined as in \eqref{eq:fixeda}, and let $G(a)$ be the centralizer of $T(a)$ in $G^{\alg}$. By Lemma \ref{lem:DDclosedgp}, $T(a)$ is generated by $s_1$ and $s_2$, hence $G(a)=G(s_1,s_2)$ where the latter is defined in \S~\ref{subsec:nonvan}. Let $G(a,x)<G(a)$ be the centralizer of $x$ in $G(a)$, then we have $G(a,x)=G(s_1,s_2,x)$. The component group $C(a,x)$ of $G(a,x)$ is equal to $C(s_1,s_2,x)$ in \S~\ref{subsec:nonvan}.

We need some basic properties about  equivariant elliptic cohomology of the Steinberg variety $Z$.
\begin{lemma}\label{lem:G&A}
Let $\pi:=\catA_\pi^Z:\catA^Z_{T\times S^1}\to \catA^Z_{G\times S^1}$ be the natural projection induced by $\catA_\pi:\catA_{T\times S^1}\to \catA_{G\times S^1}$.  Then \[\pi^*\Xi_{G\times S^1}(Z)\cong \Xi_{T\times S^1}(Z)\] as sheaves of  algebras on $\catA^Z_{T\times S^1}$.
\end{lemma}
The $K$-theory analogue of this lemma is \cite[(6.2.(6))]{CG}. In the elliptic case, it follows from the fact that $\pi_{Z*}\Xi_{T\times S^1}(Z)$ is locally free on $\catA_{G\times S^1}$, together with faithfully flat descent for the map $\catA_\pi:\catA_{T\times S^1}\to \catA_{G\times S^1}$, similar to its $K$-theory analogue.

For any $t\in E$, let $\calH_t$ be the pull-back of $\calH$ to the subvariety $\catA_G\times\{t\}\subseteq \catA_G\times E$. Similarly, for closed point $(a',t)\in \catA_G\times E$, let $\calH_{(a',t)}$ be the pull-back of $\calH$ to $(a',t)$.

\begin{prop}[\cite{CG}, (8.1.6)]
Let $a\in\catA_T$ and   $a'=\pi(a)\in\catA_G$.  Then, we have an isomorphism of  algebras \[\calH_{(a',t)}\cong H_*(Z^{T(a,t)};\bbC).\]
\end{prop}
\begin{proof}
Let $\iota:T(a,t)\inj T\times S^1\subseteq G\times S^1$ be the natural embedding, and let $\catA_\iota:\catA_{T(a,t)}\to \catA_{G\times S^1}$ be the induced map.
By Lemma~\ref{lem:G&A}, we have $\pi^*\Xi_{G\times S^1}(Z)\cong \Xi_{T\times S^1}(Z)$ as vector bundles on $\catA^Z_{T\times S^1}$, hence $\iota_\catA^*\pi_{Z*}\Xi_{T\times S^1}(Z)\cong \pi_{Z*}\Xi_{T(a,t)}(Z)$. In particular, \[\calH_{(a',t)}\cong (\pi_{Z*}\Xi_{T(a,t)}(Z))\otimes_{\catA_{T(a,t)}}\bbC_{(a,t)}.\]
By Theorem~\ref{thm:RR}, we have an isomorphism $\pi_{Z*}\Xi_{T(a,t)}(Z)\otimes_{\catA_{T(a,t)}}\bbC_{(a,t)}\cong H_*(Z^{T(a,t)};\bbC)$.
\end{proof}

Apply \S~\ref{subsec:DecompThm} to $\mu:\widetilde{\calN}^{T(a,t)}\to \calN^{T(a,t)}$ with the action of $G(a)$, we get the following corollary of Proposition~\ref{prop:non-zero_multiplicity}.
\begin{corollary}\label{cor:classify}
Assume $t\in E$ is a non-torsion point. 
Then we have the following properties. 
\begin{enumerate}
\item For $a\in \catA_G$,  $x\in \calN^{T(a,t)}$, and $\chi$  an irreducible representation of $C(a,x)$ so that $[H^*(\calB_x^{T(a)});\chi]\neq0$, the $\bbC$-vector space  $H_*(\calB^{T(a)}_x)_\chi$ has a natural $\calH_t$-module structure.
\item For each triple $(a,x,\chi)$ with $a, x$, and $\chi$ as in (1), the $\calH_t$-module $H_*(\calB^{T(a)}_x)_\chi$ has an irreducible top, denoted by $L_{a,x,\chi}$. Moreover, two such irreducible representations are isomorphic iff the two triples are conjugate, and each irreducible representations of $\calH_t$ is isomorphic to one of these. In particular, the set of irreducible representations of $\calH_t$ up to isomorphism is in one-to-one correspondence with the set of triples $(a,x,\chi)$ up to conjugation;
\item Let $(a,y,\kappa)$ be another triple. The multiplicity of the simple object $L_{a,x,\chi}$ in $H_*(\calB^{T(a)}_y)_\kappa$ is given by $\sum_k\dim H^k(i_x^!IC_{x,\chi})_{y,\kappa}$.
\end{enumerate} 
\end{corollary}
Here $IC_{x,\chi}$ is the intersection cohomology sheaf on $\calN^{T(a)}$ associated to the local system $\chi$ on the orbit $G(a)x$, and $H^k(i_y^!IC_{x,\chi})_{\kappa}$ is the $C(a,y)$-isotypical component transforms as the irreducible $C(a,y)$-module $\kappa$.
This corollary follows directly from Proposition~\ref{prop:non-zero_multiplicity} by the same argument as in \cite[\S~8.8]{CG}.

The following remark was communicated to the authors by Sasha Braverman at the Park City Mathematics Institute 2015.
\begin{remark} As a direct application of \cite[Corollary~1.10]{Kal}, $H^*(\calB_x^{T(a)})$ is generated by algebraic cycles. In particular, the Euler characteristic of $\calB_x^{T(a)}$ is the same as its dimension.
\end{remark}

\subsection{Higgs bundle interpretations}\label{subsec:Higgs}

For any principle $G^{\alg}$-bundle $\calP$ on $E^\vee$, let $\ad(\calP)$ be the Lie algebra bundle $\calP\times_{G^{\alg}}\fg$. For any $t\in E$, denote the line bundle $\calO(\{0\}-\{t\})$ on $E^\vee$ by $\calL_t$. A $\calL_t$-valued {\it Higgs field} is a holomorphic section of the vector bundle $\ad(\calP)\otimes_{\calO_{E^\vee}}\calL_t$. (For a $\fg$-bundle, we use the terminology {\it holomorphic section} instead of {\it regular section} to avoid confusion with regularity of elements in $\fg$.) The pair $(\calP,x)$ is called a {\it Higgs bundle}. 
A $B$-structures on the Higgs bundle $(\calP,x)$ is a principal $B$-bundle $\calP'$ on $E^\vee$, together with a holomorphic section $x'$ of $\ad(\calP')\otimes_{\calO_{E^\vee}}\calL_t$, such that $(\calP,x)$ is induced from $(\calP',x')$.  
Let $\bfB_{\calP,x}$ be the variety of all  $B$-structures on the Higgs bundle $(\calP,x)$.

\begin{theorem}\label{thm:Higgs}
For  $t\in E$ and  $a\in \catA_G$, let $\calL_t$ be the line bundle corresponding to $t$, and let $\calP_a$ be the principal $G^{\alg}$-bundle corresponding to $a$. Then: 
\begin{enumerate}
\item Holomorphic sections of $\ad(\calP_a)\otimes_{\calO_{E^\vee}}\calL_t^{-1}$ are in one-to-one correspondence with $\fg^{T(a,t)}$;
\item There is an isomorphism of algebraic varieties $\bfB_{\calP_a,x}\cong \calB_x^{T(a)}$, where we use $x$ to denote both the holomorphic section and the Lie algebra element;
\item There is an isomorphism of algebraic groups $Aut(\calP_a,x)\cong G(a,x)$, which intertwines the action of   $Aut(\calP_a,x)$ on $\bfB_{\calP_a,x}$ and the action of $G(a,x)$ on $\calB_x^{T(a)}$.
\end{enumerate}
\end{theorem}
We prove this theorem in the end of this section. First, let us mention an application. 
Let $Higgs^{nil}_t(E^\vee)$ be the set of isomorphism classes of triples $(\calP,x,\chi)$, where $\calP$ is a semi-simple semi-stable topologically trivial principle $G^{\alg}$-bundle on $E^\vee$, $x$ is a nilpotent holomorphic section of $\ad(\calP)\otimes_{\calO_{E^\vee}}\calL_t^{-1}$, and $\chi$ is an irreducible representation of $Aut(\calP,x)/Aut^0(\calP,x)$ with non-trivial multiplicity in $H^*(\bfB_{\calP,x})$. Here, a nilpotent section of $\ad(\calP)\otimes_{\calO_{E^\vee}}\calL_t$ is a section such that the image at each fiber is an nilpotent element of $\fg$.

The following is a direct corollary of Theorem~\ref{thm:Higgs}.
\begin{corollary}\label{cor:Higgs_Irrep}
Let $t\in E$ be a non-torsion point, then the irreducible objects in the category Mod-$\calH_t$ are in one-to-one correspondence with  $Higgs^{nil}_t(E^\vee)$.
\end{corollary}

The proof of Theorem~\ref{thm:Higgs} will occupy the rest of this section. As an intermediate step, we interpret $Higgs^{nil}_t(E^\vee)$ in terms of principal bundles with flat connections. 

Let $\pi_1:=\pi_1(E,0)$ be the fundamental group of $E$ with base point $0\in E$, and let $\tilde{E}\to E$ be the universal cover. Denote the two generators of $\pi_1$ by $\gamma_1$ and $\gamma_2$. 
For $a\in\catA_G$, let $(s_1,s_2)=DD(a)\in (T^{\bbC})^2$ as in \S~\ref{subsec:Pontryagin},  and for $t\in E$ we associate $(q_1,q_2)\in (S^1)^2\subseteq \bbC^{*2}$. Note that by construction $s_i$ is unitary, in the sense that $s_i$ lies in the maximal compact subgroup $T\subset T^{\bbC}$.  Let $\rho:\pi_1\to T\subset G^{\bbC}$ be the group homomorphism sending $\gamma_i$ to $s_i$, and  let $\eta:\pi_1\to S^1\subset \bbC^*$  be the group homomorphism sending $\gamma_i$ to $q_i$, for $i=1,2$.

The group $\pi_1$ acts on $\tilde{E}$ by deck transform,  and acts on $\fg$ via $\rho:\pi_1\to T\subset G^{\bbC}$ and the adjoint action of $G$ on $\fg$. Define $\tilde{a_\rho}$ to be the local $\fg$-system on $\tilde{E}$, endowed with the diagonal $\pi_1$-action. The quotient $\tilde{a_\rho}/\pi_1$ is a local $\fg$-system on $E$, denoted by $a_\rho$. Note that as a vector bundle on $E$, $a_\rho$ is the adjoint bundle $\ad(a)$. The construction above endows $\ad(a)$ with a unitary flat connection. 

Similarly, we have the rank-1 local system $\tilde{t}_\eta$ on $\tilde{E}$, whose descent to $E$ is denoted by $t_\eta$. This local system $t_\eta$ can be considered as a unitary flat line bundle, whose underlying line bundle   is $\calL_t$.
\begin{prop}\label{prop:flat_bundle}
The followings hold.
\begin{enumerate}
\item Flat sections of $a_\rho\otimes_\bbC t_\eta^{-1}$ are in one-to-one correspondence with $\fg^{T(a,t)}$.
\item There is an isomorphism of algebraic varieties  between $ \calB_x^{T(a)}$ and $\bfB_{a_\rho,x}$, the latter being the variety of $B$-structures on the pair $(a_\rho,x)$;
\item There is an isomorphism of algebraic groups $Aut(a_\rho,x)\cong G(a,x)$, which intertwines the action of   $Aut(a_\rho,x)$ on $\bfB_{a_\rho,x}$ and the action of $G(a,x)$ on $\calB_x^{T(a)}$.
\end{enumerate}
\end{prop}
\begin{proof}
(1). Note that  $\tilde{a}_\rho\otimes_\bbC\tilde{t}_\eta$, as a local $\fg$-system on $\tilde{E}$, is trivial. Hence a flat section of $\tilde{a}_\rho\otimes_\bbC\tilde{t}_\eta$ is canonically identified with an element $x\in \fg$. On the other hand, a flat section of the local $\fg$-system  $a_\rho\otimes_\bbC t_\eta$ on $E$ is a $\pi_1$-invariant flat section of $\tilde{a}_\rho\otimes_\bbC\tilde{t}_\eta$.
Therefore, a flat section of $a_\rho\otimes_\bbC t_\eta^{-1}$ is canonically identified with an element $x\in\fg$ with the property that $q_i^{-1}\ad(s_i)x=x$ for $i=1,2$. Recall that the $T^{\bbC}\times\bbC^*$-action on $\fg$ is by $(t,q)\cdot x=q^{-1}\ad(t)x$ for $(t,q)\in T^{\bbC}\times\bbC^*$ and $x\in\fg$. Hence, $x\in\fg$ such that $q_i^{-1}\ad(s_i)x=x$ for $i=1,2$ is the same as an element $x\in \fg^{T(a,t)}$.

(2). A $B$-structure on $(a_\rho,x)$ is a Borel subgroup $B\subseteq G$ such that $s_i\in B$ for $i=1,2$, and $x\in\Lie(B)$. Note that $s_i\in B$ if and only if $B\in\calB^{s_i}$. Therefore, the $B$-structure on $(a_\rho,x)$ are parametrized by $\calB_x^{s_1,s_2}$.

(3). It is well-known that the group of automorphisms of a local system is the centralizer of the image of $\pi_1$ in $G$. Hence, $\Aut(\rho_a)\cong G(a)$. This in term implies that the group of automorphisms of the pair $(\rho_a,x)$ is $G(a,x)$.
\end{proof}

\begin{proof}[Proof of Theorem~\ref{thm:Higgs}]
It is well-known that a holomorphic $G$-bundle on $E$ admits a  unitary flat connection if and only if it is semi-stable of degree zero \cite{NS}. For the adjoint bundle of such a $G$-bundle, a section is flat if and only if it is holomorphic. Therefore, Theorem~\ref{thm:Higgs} follows directly from Proposition~\ref{prop:flat_bundle}.
\end{proof}

\section{Miscellaneous discussions}\label{sec:discussion}
In this section we discuss various mathematical and mathematico-physical contexts in which $Higgs^{nil}_{t}(E^\vee)$ plays a role. We offer some  speculations based on the correspondence Corollary~\ref{cor:Higgs_Irrep}.

\subsection{DAHA and Hitchin system on an elliptic curve}
The set $Higgs^{nil}_{t}(E^\vee)$ is introduced in \cite{BEG} in the study of representations of the double affine Hecke algebra (DAHA) of Cherednik.
\begin{remark}\label{rmk:BEG_Vass}
\begin{enumerate}
\item For $t=0$, the set $Higgs^{nil}_{t=0}(E^\vee)$ is shown to parametrize irreducible representations of the quantum torus  in \cite{BEG}.
\item Moreover, for non-torsion $t\in E$,  it was conjectured in {\it loc. cit.} that    the set $Higgs^{nil}_t(E^\vee)$ parametrizes  irreducible representations of DAHA. 
\item However, in \cite{Vas} a classification of irreducible DAHA representations has been achieved, which shows more DAHA representations  than the ones predicted by this conjecture.
\end{enumerate}\end{remark}
Therefore, Remark~\ref{rmk:BEG_Vass} gives raise to two natural problems.
\begin{problem} \label{que:BEG_Vass}
\begin{enumerate}
\item Find a family of algebras, parametrized by $t\in E$, whose irreducible representations are parametrized by $Higgs^{nil}_t(E^\vee)$ for non-torsion point $t\in E$.
\item Find a property  $\fP$, such that the irreducible integrable representations of DAHA having property $\fP$ are parametrized by $Higgs^{nil}_t(E^\vee)$. 
\end{enumerate} 
\end{problem}

Corollary~\ref{cor:Higgs_Irrep} answers Problem~\ref{que:BEG_Vass}(1).  Problem~\ref{que:BEG_Vass}(2) is to be addressed in work in progress of the first named author in collaboration with Valerio Toledano Laredo. The property $\fP$ is expected to be {\it non-torsion} when considered as a module over one of the subalgebras isomorphic to the affine Hecke algebra.

For simplicity, from now on until the end of this section, we focus on the case when $G$ is of type $A_{n-1}$.
Non-trivial local systems $\chi$ in the classification do not show up, hence $Higgs^{nil}_{t}(E^\vee)$ becomes the moduli space of nilpotent meromorphic Higgs bundles. More precisely, let $Hitch_t(n)$ be the moduli space parametrizing 
\[\left\{(\calV,X)\mid {\begin{matrix}\calV: \rank=n, \deg=0 \hbox{ vector bundle on }E\\  X:\calV\to\calV(\{t\})=\calV\otimes\calO(\{t\}), \hbox{ semi-stable}\end{matrix}}\right\}.\] It is shown in \cite[Main Theorem]{Markman} that $Hitch_t(n)$ is a Poisson variety with Poisson structure induced from the natural section $\vartheta(z-t)$ in $H^0(E;\calO(\{t\})\otimes \Omega_E^{-1})$. The symplectic leaves are determined by coadjoint orbits.

Let $L_t$ be the subvariety of $Hitch_t(n)$ in which the Higgs field $X$ is fiber-wise nilpotent. This is an analogue of the global nilpotent cone, hence we expect   $L_t$ is  to be Lagrangian. Note that due to the failure of $Hitch_t(n)$ being symplectic in general,  $L_t$ being Lagrangian means its intersection with any symplectic leaf is Lagrangian in the leaf. For special symplectic leaves, the fact that $L_t$ is Lagrangian is already proved in \cite[\S~9.3]{Markman}.

The closed points on $L_t$ are in bijection with the set $Higgs^{nil}_{t}(E^\vee)$. Indeed, composition with the natural section $\vartheta(z)\in \Hom_E(\calL_t,\Omega_E(t))$ defines a bijection $Higgs^{nil}_{t}(E^\vee)\to L_t$,  sending a Higgs field $X:\calV\to \calV\otimes\calL_t^{-1}$ to $X':\calV\to\calV(\{t\})$, the latter of  which is $X\otimes\vartheta(z)$.

The set of closed points on $L_t$ in turn can be interpreted as the set of irreducible objects in the abelian category $\Coh_{L_t}Hitch_t(n)$ of coherent sheaves on $Hitch_t(n)$ which are set-theoretically supported on $L_t$.
In particular, Corollary~\ref{cor:Higgs_Irrep} yields that the irreducible objects in the abelian categories $\Mod \calH_t$ and $\Coh_{L_t}Hitch_t(n)$ are in natural bijection.
However, the abelian categories $\Mod \calH_t$ and $\Coh_{L_t}Hitch_t(n)$ per  se are not Morita equivalent, since the former has non-trivial extensions between different simple objects while the latter does not. This opens the following question.
\begin{question}\label{conj:KoszulDual}
Assume $G$ is of type $A_{n-1}$, are there graded liftings of the abelian categories $\Mod \calH_t$ and $\Coh_{L_t}Hitch_t(n)$ which are Koszul dual?
\end{question}
As the only evidence we have so far is Corollary~\ref{cor:Higgs_Irrep}, which is a statement about irreducible objects, it is possible that in the above question the category $\Coh_{L_t}Hitch_t$ needs to be replaced by its analogue for local systems instead of Higgs bundles. Due to the focus of the present paper, we postpone the study of this duality question to future investigations. 

\subsection{Indications in 4d $\calN=2$ gauge theories}
There is a string theory $ST_{IIB}$ of type $IIB$, whose dimension reduction gives a $6d$ $\calN=(2,0)$ supersymmetric quantum field theory denoted by $S_\Gamma$. 
This dimension reduction depends on a quiver $\Gamma$, which for the purpose of the present paper, will be restricted to be of type $A_{n-1}$.
Although mathematically the theory $S_\Gamma$ has not been understood completely, various aspects of its dimension reduction to $4d$ has been of interest to mathematicians \cite{Nak}. 

More concretely, let $\fg_\Gamma$ be the Lie algebra associated to the Dynkin diagram $\Gamma$. Let $(C,x_i,\rho_i)_{i=1,\dots,k}$ be a Riemann surface $C$ with marked points $x_1,\dots,x_k$, and $\fs\fl_2$-triples $\rho_i:\fs\fl_2\to\fg_\Gamma$ for $i=1,\dots,k$. There is a $4d$ $\calN=2$ supersymmetric quantum field theory denoted by $S_\Gamma[(C,x_i,\rho_i)_{i=1,\dots,k}]$, obtained from $S_\Gamma$ by compactification over $C$ with defects at the marked point $x_i$ specified by $\rho_i$. 
The Coulomb branch of the moduli space of vacua is the moduli space of meromorphic Higgs bundles. In particular, let $C=E$ be an elliptic curve, with one marked point $t\in E$, for each $\rho: \fs\fl_2\to \fg_\Gamma$, the Coulomb branch of the corresponding theory $S_\Gamma[(E,t,\rho)]$ is a symplectic leaf in the space $Hitch_t(n)$.

Now we compactify the theory $S_\Gamma[(E,x,\rho)]$ along a circle $S^1_R$ with radius $R$, and take limit $R\to 0$, we get a $3d$ $\calN=4$ quantum field theory. 
We formally replace $E$ by  a sphere $\bbP^1$, the  $3d$ theory obtained this way is mirror dual to a quiver gauge theory, where the quiver and dimension vectors are determined by $\rho:\fs\fl_2\to\fg_\Gamma$. 
It is well-known that the Higgs branches of type-$A$ quiver gauge theories are  the closure of nilpotent orbits.
 Theorem \ref{thm:main} then yields that the convolution algebra in elliptic cohomology of its resolution is the elliptic affine Hecke algebra. 

To summarize the above discussion, in terms of quantum field theories, Corollary~\ref{cor:Higgs_Irrep} and Question~\ref{conj:KoszulDual} are about the mirror dual $\bbD$ of the $3d$ theory $S_\Gamma[\bbP^1,x,\rho][S^1]$ and the $4d$ theory $S_\Gamma[E,t,\rho]$.
Schematically this duality can be summarized as a correspondence 
\[\Mod \Ell_{G_\Gamma}\calM_{Higgs}(\bbD S_\Gamma[\bbP^1,x,\rho][S^1]) \longleftrightarrow \Coh \calM_{Coulomb}(S_\Gamma[E,x,\rho]),\]
which we speculate to be a Koszul duality  Question~\ref{conj:KoszulDual}.

Now we make a few comments about some features observed in the above.
\begin{remark}
\begin{enumerate}
\item It has been expected that singular cohomology of a modification of the Higgs branch $\calM_{Higgs}(\bbD S_\Gamma[\bbP^1,x,\rho][S^1]) $ of the mirror dual of $S_\Gamma[\bbP^1,x,\rho][S^1]$  is isomorphic to the space of functions on a torus fixed locus of $\calM_{Coulomb}(\bbD S_\Gamma[\bbP^1,x,\rho][S^1])$ \cite[\S~1(viii)]{Nak}. The nature of the above duality suggests that if the singular cohomology is replaced by the elliptic cohomology, than it is related to functions on the Coulomb branch of a $4d$ theory that is obtained from $S_\Gamma$ via a different compactification.
\item We postpone to the future the study of Schur-Weyl duality between the elliptic affine Hecke algebras $\calH_t$ and the elliptic quantum groups $E_{\tau,\hbar}(\fs\fl_d)$. Under this duality, the quantization parameter $\hbar$ of the quantum group corresponds to the parameter $t$ of the elliptic affine Hecke algebra, which in turn is the marked point in the theory $S_\Gamma[E,t,\rho]$.
\end{enumerate}
\end{remark}

\section{Representations at torsion points in type-$A$}
\label{sec:torsion_points}
In this section we still work under the assumption that $E$ is an elliptic curve over $\bbC$. We study the combinatorics related to representations of the elliptic affine Hecke algebra corresponding to $U_n$, when the parameter $\gamma$ is evaluated at a torsion point $t\in E$.

\subsection{Reminder on quiver Hecke algebras}\label{subsec:quiverHecke}
Let $\Gamma$ be an arbitrary finite symmetric quiver, with the set of vertices denoted by $I$, and the corresponding Cartan matrix denoted by $C$. Following the conventions of Rouqier \cite[\S3.2.4]{Rouq08}, for any pair of vertices $i,j\in I$, we define a polynomial in two variables $P_{i,j}(u,v)=(v^{h/j\cdot j}-u^{h/i\cdot i})^{d_{i,j}}$ where $h=\lcm(i\cdot i,j\cdot j)$ if $i\neq j$, and $P_{i,i}(u,v)=0$.

Associated to $\Gamma$, there is a quiver Hecke algebra  (also known as the KLR-algebra) $H_n(\Gamma)$ for any $n\geq0$. See \cite[\S3.2.1]{Rouq08}  for a presentation of this algebra, which we will not  use  in this paper. Instead, we recall the following fact (proved in \cite[Proposition~3.12]{Rouq08} and \cite[\S~2.3]{KhoLau}), which can be taken as the definition.

Let $\calO'=\bigoplus_{\nu\in I^n}\bbC[x_1,\dots,x_n][\{(x_i-x_j)^{-1}\}_{i\neq j,\nu_i=\nu_j}]$ with the component-wise multiplication. With the unity in the direct summand labelled by $\nu\in I^n$ denoted by $1_\nu$, the set $\{1_\nu\}$ consists of pair-wise orthogonal idempotents with $1=\sum_\nu 1_\nu$. Let $A_n(I)=\bbC^{(I)}[x]\wr \fS_n$, the wreath product. The variable $x$ in the $i$-th $\bbC^{(I)}[x]$-tensor factor of $A_n(I)$ is denoted by $x_i$. Define $\tau_i\in \calO'\otimes_{\Z^{(I)}[x]^{\otimes n}}A_n(I)$ as follows
\[\tau_{i,\nu}=\left\{ \begin{array}{ll}  \frac{s_i-1}{x_i-x_{i+1}}1_\nu, & \text{if } \nu_i=\nu_{i+1}, \\
P_{\nu_i,\nu_{i+1}}(x_{i+1},x_i)s_i1_\nu, &\text{ otherwise.}\end{array}\right.\]

\begin{prop}\label{prop:KLR_poly_rep}
The algebra $H_n(\Gamma)$ is the subalgebra of $\calO'\otimes_{\Z^{(I)}[x]^{\otimes n}}A_n(I)$ generated by  
$1_\nu$ for $ \nu\in I^n$,
$x_i$  and $\tau_i$ for $i=1,\dots,n$.

\end{prop}
In particular, the algebra $H_n(\Gamma)$ admits a faithful representation on $\bigoplus_{\nu\in I^n}\bbC[x_1,\dots,x_n]1_\nu$.

We recall some well-known facts about  representations of $H_n(\Gamma)$. The idempotents $1_\nu$'s define a direct sum decomposition 
\begin{equation}\label{eq:dec}
H_n(\Gamma)=\bigoplus_{\nu\in I^n}H_{\nu}(\Gamma).
\end{equation}
For each $\nu$, the algebra $H_\nu(\Gamma)$ has a natural grading (see \cite[\S~2.1]{KhoLau}), which is compatible with the decomposition \eqref{eq:dec}. 
Let $\Mod_0^{\gr}H_n(\Gamma)$ be the abelian category of finite dimensional graded $H_n(\Gamma)$-modules, and let $\Proj^{\gr}H_n(\Gamma)$ be the exact category of graded projective modules. Let $K^0(\Mod_0^{\gr}H_n(\Gamma))$ (resp. $K^0(\Proj^{\gr}H_n(\Gamma))$) be  the Grothendieck group of $\Mod_0^{\gr}H_n(\Gamma)$ (resp. $\Proj^{\gr}H_n(\Gamma)$).
They are  $\ZZ[q^\pm]$-algebras where $q$ acts by degree shifting. Note that the Euler pairing induces an isomorphism $K^0(\Proj^{\gr}H_n(\Gamma))_{\bbC}\cong K^0(\Mod_0^{\gr}H_n(\Gamma))^*_{\bbC}$. Under this pairing, the basis on the left hand side formed by the classes of indecomposable projective objects are mapped to the dual basis to the classes of simple objects on the right hand side.

Observe that the natural embeddings $H_k(\Gamma)\otimes H_l(\Gamma)\inj H_{k+l}(\Gamma)$ for any $l,k\in\bbN$ define induction functors \[\Ind_{k,l}:\Proj^{\gr}H_k(\Gamma)\otimes \Proj^{\gr}H_l(\Gamma)\to \Proj^{\gr}H_{k+l}(\Gamma),\] and hence induce a multiplication on $\bigoplus_nK^0\left(\Proj^{\gr}H_n(\Gamma)\right)$, making it into an associative algebra.

We summarize the basic theory of quiver Hecke algebras as the following.
\begin{theorem}\label{thm:quiverHecke}
With notations as above, we have the followings.
\begin{enumerate}
\item There is an isomorphism of $\bbC[q^\pm]$-algebras 
\[\bigoplus_nK^0\left(\Proj^{\gr}H_n(\Gamma)\right)\cong U_{q^\pm}(\fn^-),\]
 where $\fn^-$ is the negative half of the Kac-Moody Lie algebra  associated to the quiver $\Gamma$.
\item For any $i\in I$, let $P(i)=H_1(\Gamma)\cdot 1_i$. Then the isomorphism above sends $[P(i)]$ to $f_i\in U_{q^\pm}(\fn^-)$, the Chevalley generator of the quantum group.
\item The basis of the left hand side formed by the classes of indecomposable projective objects are mapped to the Lusztig canonical basis of the right hand side.
\end{enumerate}
\end{theorem}
The first two statements in Theorem~\ref{thm:quiverHecke} are proved in \cite{KhoLau} and \cite{Rouq08}. The third statement is 
conjectured by Khovanov and Lauda and proved by Varagnolo and Vasserot in \cite{VV}.

There is also a version of Theorem~\ref{thm:quiverHecke} without grading. Let $\Mod_0H_n(\Gamma)$ be the category of finite dimensional $H_n(\Gamma)$-modules on which $x_i$ acts nilpotently for any $i=1,\dots,n$. The inclusion \[H_{n-1}(\Gamma)\otimes H_1(\Gamma)\inj H_n(\Gamma)\] induces a restriction of scalars $\Mod_0H_n(\Gamma)\to \Mod_0( H_{n-1}(\Gamma)\otimes H_1(\Gamma))$.
For any $i\in I$, the  right multiplication by the idempotent $1_i\in H_1(\Gamma)$ defines $\Mod_0 (H_{n-1}(\Gamma)\otimes H_1(\Gamma))\to \Mod_0H_{n-1}(\Gamma)$. Let \[\Res_i:\Mod_0H_n(\Gamma)\to \Mod_0H_{n-1}(\Gamma)\] be their composition.
 Then there is a $\bbC$-algebra isomorphism \[\bigoplus_nK^0\left(\Mod_0H_n(\Gamma)\right)^*_{\bbC}\cong U(\fn^-).\] 
This isomorphism intertwines the operation $[\Res_i]^*$ on the left hand side and  multiplication by $f_i$ on the right hand side.

\subsection{Completion of $\calH_n$ and the quiver Hecke algebra}
From now on $\calH_n$ is the elliptic affine Hecke algebra of $G=U_n$. We drop the lower subscript $n$ if it is understood from the context. The simple roots of $U_n$ are given by $\al_i=e_i-e_j$, $i=1,...,n-1$. We have $\catA=\catA_{T\times S^1}=E^{n+1}, \catA_G=E^{(n)}$, $\catA/W=E^{(n)}\times E$, and  $\calS=\pi_*\calO_{E^{n+1}}$ where $\pi:E^{n+1}\to E^{(n)}\times E$ is the symmetrization map.  For any $\nu\in \catA=E^{n+1}$, we write it in coordinates as $\nu=(\nu_1,\dots,\nu_n,\nu_\gamma)$. The following lemma is a direct consequence of the structure theorem. Let $\fl$ be an arbitrary local parameter of the elliptic curve $E$. 

Recall from \S~\ref{sec:ell} that the divisor $D^i$ is given by the equation $\nu_i=\nu_{i+1}$, and $D^{i,\gamma}$ is given by $\nu_{i+1}-\nu_i=\nu_\gamma$.
\begin{lemma}\label{lem:loc_complete_hecke}
\begin{enumerate}
\item Let $\widetilde{\catA}^c=[\catA\backslash(\cup_{1\le i\le n-1}(D^i\cup D^{i,\gamma}))]/W$. The sheaf of algebras $\calH|_{\widetilde{\catA}^c}$ is isomorphic to $\calS_W|_{\widetilde\catA^c}$.
\item Let $\nu\in E^{n+1}$ be a closed point. Let $\calO^\wedge_\nu$ be the completion of $\calO_{E^{n+1}}$ at $\nu$. Then, the algebra structure on $\calH$ induces an algebra structure on 
\[\calH^\wedge_{\fS_n\cdot\nu}:=(\bigoplus_{\mu\in \fS_n\cdot \nu}\pi_*\calO_\mu^\wedge)\otimes_{\calS}\calH.\]
\item The algebra $\calH^\wedge_{\fS_n\cdot\nu}$ is a subalgebra of $\End(\underset{\mu\in \fS_n\cdot \nu}\bigoplus\pi_*\calO_\mu^\wedge)$, generated by $\underset{\mu\in \fS_n\cdot \nu}\bigoplus\pi_*\calO_\mu^\wedge$ and  the operators $T_i$ with $i=1,...,n-1$ such that 
\[(T_{i})_\mu:=\left\{ \begin{array}{ll}
\frac{\fl(\gamma)}{\fl(x_{i+1}-\mu_{i+1})-\fl(x_i-\mu_i)}(1_\mu-s_i)+s_i, &\text{ if }\mu\in D^{i} ;\\
 \big(\fl(x_{i+1}-\mu_{i})-\fl(x_i-\mu_i)-\fl(\gamma-\mu_\gamma)\big)s_i, &\text{if }\mu\in D^{i,\gamma};\\
s_i, & \text{otherwise}.\end{array} \right.\]
\item Evaluating $\gamma=\nu_\gamma$, we get that the algebra $(\calH^\wedge_{\fS_n\cdot\nu})|_{\gamma=\nu_\gamma}$ is generated by \[(\underset{\mu\in \fS_n\cdot \nu}{\bigoplus}\pi_*\calO_\mu^\wedge)/(\gamma=\nu_\gamma)\] and the  operators $T_i$:
\[(T_i)_\mu:=\left\{\begin{array}{ll}
\frac{\fl(\nu_\gamma)}{\fl(x_{i+1})-\fl(x_i)}(1_\mu-s_i)+s_i,& \hbox{ if }\mu\in D^{i};\\
\big(\fl(x_{i+1}-\mu_{i+1})-\fl(x_i-\mu_i)\big)s_i, & \hbox{ if }\mu\in D^{i,\gamma};\\
s_i, &\hbox{ otherwise}.\end{array}\right.\]
\end{enumerate}
\end{lemma}
In (3) and (4), the idempotent element in $\underset{\mu\in \fS_n\cdot \nu}\bigoplus\pi_*\calO_\mu^\wedge$ corresponding to the multiplicative identity in $\pi_*\calO_\mu^\wedge$ is denoted by $1_\mu$. The operators $s_i$ is understood as going from $\pi_*\calO_\mu^\wedge$ to $\calO_{s_i\mu}^\wedge$.

Let $t\in E$ be a torsion point. 
We naturally identify  $E^n$ with $E^n\times\{t\}\subseteq E^{n+1}$. Then $\calH_n/(\nu_\gamma=t)$ is a sheaf of algebras on $E^{(n)}$.
Let  $(q_1,q_2)\in (\bbC^*)^2$ be $DD(t)$, where $DD$ is the map defined in \S~\ref{subsec:Pontryagin}.  Assume that $q_1$ is of order $n_1$, and $q_2$ is of order $n_2$, where $n_1,n_2$ are integers strictly grater than 1.
Let $S_t\subset E$ be the subset consisting of $z\in E$  such that $DD(z)$ is of the form $(q_1^u,q_2^v)\in \bbC^*\times\bbC^*$ for $u,v\in \bbZ$. Let $S_t^n\subseteq E^n$ be the subset of $E^n$ consisting of points whose coordinates are in $S_t$.
Let $\Mod_{t}\calH_n$ be the subcategory of finite dimensional $\calH_n-$modules, whose restriction to the action of  $\calS$, considered as a coherent sheaf on $E^{(n)}$, is set-theoretically supported on $\pi(S_t^n)$.

We have the following commutative diagram 
\[\xymatrix{E^k\times E^l\ar@{=}[r]\ar[d]&E^{k+l}\ar[d]\\
E^{(k)}\times E^{(l)}\ar[r]^{\pi_{k,l}}&E^{(k+l)}.}\]
Consider $\calH_k\boxtimes\calH_l$ as a coherent sheaf of algebras on $E^{(k)}\times E^{(l)}$.
By definition of the elliptic affine Hecke algebras, $\pi_{k,l*}\calH_k\boxtimes\calH_l$ is a subsheaf of algebras of $\sEnd_{E^{(k+l)}}(\pi_*\calO_{E^{k+l}})$. It is  easy  to verify that $\pi_{k,l*}\calH_k\boxtimes\calH_l$ lies in the subsheaf $\calH_{k+l}$. Therefore, we have an  injective map $\pi_{k,l*}\calH_k\boxtimes\calH_l\to\calH_{k+l}$, which induces induction and restriction functors on the category of representations.
One can easily check that when restricting to subcategories of representations, the following functors \[\xymatrix{\Mod_{t}\calH_k\otimes \Mod_{t}\calH_l\ar@/^/@<1ex>[rr]^-{Ind}&&\Mod_{t}\calH_{k+l}\ar@/^/@<1ex>[ll]^-{\Res}}\] are well-defined

Let $d=\lcm(n_1,n_2)$ and let $l=\frac{n_1n_2}{d}$.
Let $\Gamma=\Gamma_{d,l}$ be the disjoint union of $l$ quivers, each of type $A^{(1)}_{d-1}$. Note that $\Gamma$ naturally embeds into $E$ as follows. Let $I:=S_t$, which is a subset of $E$. For $z, z'\in I$ with $DD(z)=(q_1^u,q_2^v)$ and $DD(z')=(q_1^{u+1},q_2^{v+1})$, there is a single arrow from $z$ to $z'$. In particular, the set of vertices $I$ can be relabelled by the set $\{(i,j)\mid i=0,\dots,d-1,\hbox{ and }j=1\dots,l\}$. There is an arrow from $(i,j)$ to $(i',j')$ if $i'=i+1\mod d$ and $j'=j\mod l$. No arrows otherwise. One can simply verify that the quiver obtained this way is isomorphic to $\Gamma_{d,l}$.

\begin{remark}
It follows from \S~\ref{subsec:ell_coh_conv} and Theorem~\ref{thm:main} that $(\calH_n)_{S_t^n/\fS_n}^\wedge\cong \bigoplus_{\mu\in S_t^n/\fS_n}H^*_{G(\mu,t)}(Z^{T(\mu,t)})$.   It is not hard to show that the right hand side is isomorphic to the quiver Hecke algebra for the quiver $\Gamma_{d,l}$. The induction and restriction functors can also be described geometrically similar to \cite[\S~4]{AJL}.  However, thanks to Lemma~\ref{lem:loc_complete_hecke}, it turns out to be easier to follow the purely algebraic approach of Rouquier in \cite{Rouq08}.
\end{remark}

Let $\calO^\wedge=\bigoplus_{\nu\in I^n}\bbC[x_1,\dots,x_n]^\wedge_{x=\nu}$, and let the idempotent corresponding to $\nu$ be denoted by $1_\nu$.
Define an algebra structure on 
\[\calO^\wedge H_n(\Gamma):=\calO^\wedge \otimes_{(\bigoplus_{\nu\in I^n}\bbC[x_1,\dots,x_n]1_\nu)}H_n(\Gamma)\] by setting \[\tau_i1_\nu-1_{s_i\nu}\tau_i=(x_{i+1}-x_i)^{-1}(1_\nu-1_{s_i\nu}).\]
Here $s_i$ is the simple reflection exchanging $i$ and $i+1$. Let the operators $T_i$ be defined as in Lemma~\ref{lem:loc_complete_hecke}.
\begin{lemma}
The following assignment induces an isomorphism of algebras
\[\phi:\calO^\wedge H_n(\Gamma_{d,l})\to \bigoplus_{\mu\in S_t^n/\fS_n}(\calH_n)^\wedge_\mu\]
where 
\begin{align*}
\phi(x_i1_\nu)&=\fl(z_i-\nu_i)1_\nu&\\
\phi(\tau_i)&=\frac{T_i-1}{\fl(z_{i+1}-\nu_{i+1})-\fl(z_i-\nu_i)-\fl(t)}1_\nu=\frac{s_i-1}{\fl(z_{i+1}-\nu_{i+1})-\fl(z_i-\nu_i)}1_\nu,&\hbox{ if }\nu_{i+1}=\nu_i,\\
\phi(\tau_i)&=T_i1_\nu=\big(\fl(z_{i+1}-\nu_{i+1})-\fl(z_i-\nu_i)\big)s_i1_\nu,&\hbox{ if }\nu_{i+1}=\nu_i+t,\\
\phi(\tau_i)&=\big(\fl(z_{i+1}-\nu_{i+1})-\fl(z_i-\nu_i)\big)s_i1_\nu,&\hbox{ otherwise.}
\end{align*}
\end{lemma}
\begin{proof}
This Lemma follows directly from \cite[Proposition~3.12]{Rouq08} (which is recalled above as Proposition~\ref{prop:KLR_poly_rep}) and Lemma~\ref{lem:loc_complete_hecke}.
\end{proof}

Due to this Lemma, there is a $\Gm$-action on $\bigoplus_{\mu\in S_t^n/\fS_n}(\calH_n)^\wedge_\mu$ with non-negative weights coming from the $\Gm$-action on $\calO^\wedge H_n(\Gamma_{d,l})$. Hence, there is a well-defined notion of finite dimensional graded modules over $\calH_n$ supported on $S_t^n$. We denote the category of such modules by $\Mod^{\gr}_{t}\calH_n$. For each pair of integers $k,l\in \bbN$, the induction and restriction functors \[\xymatrix{\Mod^{\gr}_{t}\calH_k\otimes \Mod^{\gr}_{t}\calH_l\ar@/^/@<1ex>[rr]^-{\Ind}&&\Mod^{\gr}_{t}\calH_{k+l}\ar@/^/@<1ex>[ll]^-{\Res}}\] are well-defined, and are compatible with the forgetful functors $\Mod^{\gr}_{t}\calH_n\to \Mod_{t}\calH_n$.

Summarizing the discussion above, we have the following.
\begin{theorem}\label{thm:ellHecke_quiverHecke}
With notations as above, we have
\begin{enumerate}
\item There is an equivalence of abelian categories
$\Mod_0^{\gr}H_n(\Gamma_{d,l})\cong\Mod^{\gr}_{t}\calH_n$
which is compatible with $\Ind$ and $\Res$ on both sides.
\item There is an equivalence of abelian categories
$\Mod_0H_n(\Gamma_{d,l})\cong\Mod_{t}\calH_n$
which is compatible with $\Ind$ and $\Res$ on both sides.
\item The equivalences in (1) and (2) intertwines with the forgetful functors $\Mod^{\gr}_{t}\calH_n\to \Mod_{t}\calH_n$ and $\Mod_0^{\gr}H_n(\Gamma_{d,l})\to \Mod_0H_n(\Gamma_{d,l})$.
\end{enumerate}
\end{theorem}

For each $(i,j)\in I$, we define functors $\Res_{(i,j)}:\Mod^{\gr}_{t}\calH_n\to\Mod^{\gr}_{t}\calH_{n-1}$ as follows. For any module $M\in \Mod^{\gr}_{t}\calH_n$, by definition, $M$ decomposes into $\bigoplus_{(u,v)\in S_t}M_{(u,v)}$ as a coherent sheaf on $E^n$. Here for each $(u,v)\in S_t$,  the submodule $M_{(u,v)}$  is the direct summand of $M$ as coherent sheaf on $E^n$, whose support has the $n$-th coordinate equal to $(u,v)\in S_t\subset E$. Consider $M$ as a module over $\calH_{n-1}$ via the map $\calH_{n-1}\boxtimes1\subseteq\calH_{n-1}\boxtimes\calH_1\to \calH_n$.
Obviously, each  $M_{(u,v)}$ is a $\calH_{n-1}$-submodule of $M$.
We define $\Res_{(i,j)}(M)$ to be the direct summand of $M$ as $\calH_{n-1}$-module whose support has the $n$-th coordinate equal to $(i,j)\in I\subseteq E$.
Similarly to the graded situation, the functors $\Res_{(i,j)}:\Mod_{t}\calH_n\to\Mod_{t}\calH_{n-1}$ are also well-defined.

Let $U_q(\widehat{\fs\fl_d})$ be the quantized enveloping algebra of $\widehat{\fs\fl_d}$, and $U^-_q(\widehat{\fs\fl_d})$ its negative nilpotent subalgebra. Let \[f_{ij}= 1\otimes\cdots\otimes1\otimes f_j\otimes1\otimes\cdots\otimes1\in U^-_q(\widehat{\fs\fl_d})^{\otimes l}\] where $f_j$ is in the $i$-th tensor factor.

\begin{corollary}\label{cor:ellHeck_Quant}
\begin{enumerate}
\item There is a $\bbC$-linear isomorphism $U^-_q(\widehat{\fs\fl_d})^{\otimes l}\to \bigoplus_nK(\Mod^{\gr}_{t}\calH_n)^*$.
\item This isomorphism intertwines multiplication of $f_{ij}$ on the left and $[\Res_{i,j}]^*$ on the right.\item Under this isomorphism, the  basis on  $K(\Mod^{\gr}_{t}\calH_n)^*$ dual to the classes of simple objects corresponds to the Lusztig canonical basis on $U^-_q(\widehat{\fs\fl_d})^{\otimes l}$.
\item This isomorphism induces an isomorphism $U^-(\widehat{\fs\fl_d})^{\otimes l}\cong \bigoplus_nK(\Mod_{t}\calH_n)^*$, which also intertwines $[\Res_{i,j}]^*$ and $f_{ij}$.
\end{enumerate}  
\end{corollary}
This is a direct corollary of Theorem~\ref{thm:quiverHecke} and Theorem~\ref{thm:ellHecke_quiverHecke}.

\newcommand{\arxiv}[1]
{\texttt{\href{http://arxiv.org/abs/#1}{arXiv:#1}}}
\newcommand{\doi}[1]
{\texttt{\href{http://dx.doi.org/#1}{doi:#1}}}
\renewcommand{\MR}[1]
{\href{http://www.ams.org/mathscinet-getitem?mr=#1}{MR#1}}


\begin{thebibliography}{00}
\bibitem[AO16]{AO} M. Aganagic and A. Okounkov, {\em Elliptic stable envelope},  J. Amer. Math. Soc. 34 (2021), no. 1, 79-133. \MR{4188815}

\bibitem[And00]{And00} M. Ando, {\em Power operations in elliptic cohomology and representations of loop groups}. Trans. Amer. Math. Soc., 352(12), (2000), 5619--5666. \MR{1637129}

\bibitem[And03]{And03} M. Ando, {\em The Sigma-orientation for circle-equivariant elliptic cohomology},  Geom. Topol., {\bf 7}, (2003), 91--153. \MR{1988282}

\bibitem[Ar95]{Arik} S.~Ariki, {\em Representation theory of a Hecke algebra of $G(r,p,n)$}. J. Algebra {\bf 177} (1995), no. 1, 164--185. \MR{1356366}
 

\bibitem[AJL08]{AJL} S.~Ariki, N.~Jacon, and C.~Lecouvey, {\em The modular branching rule for affine Hecke algebras of type A},  Adv. Math. 228 (2011), no. 1, 481-526. \MR{2822237}

\bibitem[AS69]{AS} M.~F.~Atiyah and G.~B.~Segal, {\em Equivariant $K$-theory and completion}, J. Differential Geometry {\bf 3} (1969 )1--18. \MR{0259946}


\bibitem[BEG03]{BEG} V.~Baranovsky,  S.~Evens,  and V.~Ginzburg, {\em Representations of quantum tori and G-bundles on elliptic curves.} The orbit method in geometry and physics (Marseille, 2000), 29--48, 
Progr. Math., {\bf 213}, Birkhauser Boston, Boston, MA, 2003. \MR{1995373}

\bibitem[BG96]{BG} V.~Baranovsky and V.~Ginzburg, {\em Conjugacy classes in loop groups and G-bundles on elliptic curves}, International Mathematics Research Notices, Volume 1996, Issue 15, 1996, Pages 733--751. \MR{1413870}

\bibitem[BBD82]{BBD} A.A.~Beilinson, J.N.~Bernstein, P.~Deligne, {\em Faisceaux pervers}, Ast\'erisque {\bf 100}, Paris, Soc.
Math. Fr. 1982. \MR{751966}


%\bibitem[BE09]{BezE} R.~Bezrukavnikov, P.~Etingof, {\em Parabolic induction and restriction functors for rational Cherednik algebras}. Selecta Math. (N.S.) {\bf 14} (2009), no. 3--4, 397--425. \MR{2511190}


%\bibitem[BE90]{BE} P. Bressler and  S. Evens, {\it The Schubert calculus, braid relations, and generalized cohomology}, Trans. Amer. Math. Soc. 317 (1990), no. 2, 799--811.  \MR{0968883}


%\bibitem[BV97]{BrVerg} M.~Brion and M.~ Vergne, {\em On the localization theorem in equivariant cohomology}, preprint, (1997). \arxiv{dg-ga/9711003}


\bibitem[BK08]{BK} J.~Brundan and A.~Kleshchev, {\em Blocks of cyclotomic Hecke algebras and Khovanov-Lauda algebras.} Invent. Math. {\bf 178} (2009), no. 3, 451--484. \MR{2551762}





\bibitem[CPZ13]{CPZ}
B. Calm\`es, V. Petrov, and K. Zainoulline,
\textit{Invariants, torsion indices and oriented cohomology of complete
  flags,} Ann. Sci. \'Ecole Norm. Sup. (4) 46 (2013), no. 3, 405--448. \MR{3099981}



\bibitem[CZZ12]{CZZ1}
B. Calm\`es, K. Zainoulline,  and C. Zhong, 
\textit{A coproduct structure on the formal affine Demazure algebra}, Math. Zeitschrift, 282 (2016) (3), 1191-1218. \MR{3473664} 

\bibitem[CZZ13]{CZZ2}
B. Calm\`es, K. Zainoulline,  and C. Zhong, 
{\it Push-pull operators on the formal affine Demazure algebra and its dual},   Manuscripta Math., 160 (2019), no. 1-2, 9-50. \MR{3983385}

\bibitem[CZZ14]{CZZ3} 
B. Calm\`es, K. Zainoulline,  and C. Zhong, 
\textit{Equivariant oriented cohomology of flag varieties}, Doc. Math., Extra Volume: Alexander S. Merkurjev's Sixtieth Birthday (2015), 113-144.  \MR{3404378}


\bibitem[Ch10]{Chen10} H.-Y. Chen, {\em Torus equivariant elliptic cohomology and sigma orientation}. Ph.D. Thesis, University of Illinois at Urbana-Champaign, 109 pp, (2010). \MR{2873496}

\bibitem[CG97]{CG} N.~Chriss and V.~Ginzburg, {\sl Representation theory and complex
geometry}, Birkh\"auser, Boston-Basel-Berlin, 1997. \MR{2838836}


\bibitem[D73]{De} P. Deligne, {\em Cohomologie \'a supports propres, Th\'eorie des topos et cohomologie \'etale des sch\'emas}. Tome 3. S\'eminaire de G\'eom\'etrie Alg\'ebrique
du Bois-Marie 1963--1964 (SGA 4). Lecture Notes in Mathematics, Vol.
305, Springer-Verlag, Berlin, 1973, pp. 250--462.

\bibitem[DW96]{DW} R.~Donagi and E. Witten, {\em Supersymmetric Yang-Mills theory and integrable systems}. Nuclear Phys. B 460 (1996), no. 2, 299--334. \MR{1377167}

\bibitem[D62]{Dy} E. Dyer.{\em  Relations between cohomology theories}. Col. Alg. Topology(1962), 1--10.

\bibitem[F95]{F1} G. Felder, {\em Elliptic quantum groups}, XIth International Congress of Mathematical Physics (Paris, 1994), 211--218, Int. Press, Cambridge, MA, (1995).  \MR{1370676}

\bibitem[FRV17]{FRV} G. Felder, R. Rimanyi, and A. Varchenko, {\em Elliptic dynamical quantum groups and equivariant elliptic cohomology}, SIGMA Symmetry Integrability Geom. Methods Appl. 14 (2018), Paper No. 132, 41 pp. \MR{3893437}



\bibitem[FV97]{FV2}  G.~Felder  and A.~Varchenko, {\em Elliptic quantum groups and Ruijsenaars models}, J.  Statistical Phys., Vol. 89, Nos. 5/6, (1997) \MR{1606760}


\bibitem[FMW98]{FMW} Robert Friedman, John W. Morgan, Edward Witten, {\em Principal G-bundles over elliptic curves}, Math.Res.Lett. 5 (1998) 97--118
\MR{1618343}


\bibitem[SGA3]{SGA3} 
M.~Demazure, A.~Grothendieck, 
{\em Sch\'emas en groupes I, II, III}, 
Lecture Notes in Math 151, 152, 153, Springer-Verlag, New York, 1970,
and new edition :  Documents Math\'ematiques 7, 8,  Soci\'et\'e Math\'ematique de France, 2003.  


\bibitem[Ga12]{Gan} N. Ganter, {\em The elliptic Weyl character formula},  Compos. Math. 150 (2014), no. 7, 1196-1234. \MR{3230851 }

\bibitem[Ge06]{Gep} D.~Gepner, {\em Equivariant elliptic cohomology and homotopy topoi}, Ph.D. thesis, University of Illinois, (2006).\MR{2708962}

\bibitem[GM20]{GM} D.~Gepner, L. Meier, {\em On equivariant topological modular forms}, arXiv preprint, (2020). \arxiv{2004.10254}

%\bibitem[Gi85]{Gin85} V. Ginzburg, {\em Deligne-Langlands conjecture and representations of affine Hecke algebras}, Preprint, Moscow (1985).

\bibitem[GKV95]{GKV95}V.~Ginzburg, M.~Kapranov, and E.~Vasserot, {\em Elliptic algebras and equivariant elliptic cohomology},
Preprint, (1995). \arxiv{9505012}

\bibitem[GKV97]{GKV97}
V.~Ginzburg, M.~ Kapranov, and E.~Vasserot,
{\em Residue construction of Hecke algebras}, Adv. Math. {\bf 128} (1997), no. 1, 1--19. \MR{1451416}


\bibitem[Gr94a]{Gr94} I.~Grojnowski, {\em Representations of affine Hecke algebras (and affine quantum $GL_n$) at roots of unity}, Internat. Math. Res. Notices (1994), no. 5, 215 ff., approx. 3 pp. \MR{1270135}

\bibitem[Gr94b]{Gr} I.~Grojnowski, {\em Delocalized equivariant elliptic cohomology}, Elliptic cohomology, London Math. Soc. Lecture Note Ser., {\bf 342}, Cambridge Univ. Press, (2007), 111--113. \MR{2330509}


\bibitem[Har77]{Hart} R.~Hartshorne, {\em Algebraic geometry}. Graduate Texts in Mathematics, No. 52. Springer-Verlag, New York-Heidelberg, 1977. xvi+496 pp. \MR{0463157}

\bibitem[HMSZ12]{HMSZ}
A. Hoffnung, J. Malag\'{o}n-L\'{o}pez, A. Savage, and K. Zainoulline, \textit{Formal Hecke algebras and algebraic oriented cohomology
theories,} Selecta Math., 20 (2014), no. 4, 1213--1245. \MR{3273635}

\bibitem[Kal08]{Kal} D.~Kaledin, {\em Derived equivalences by quantization},  Geom. Funct. Anal. 17 (2008), no. 6, 1968-2004. \MR{MR2399089 }

\bibitem[KW07]{KW} A. Kapustin and E. Witten, {\em Electric-magnetic duality and the geometric Langlands program.} 
Commun. Number Theory Phys. 1 (2007), no. 1, 1--236. \MR{2306566}

\bibitem[KL87]{KL} D. Kazhdan and G. Lusztig, {\em Proof of the Deligne-Langlands conjecture for Hecke algebras}, Invent. Math. 87 (1987), no. 1, 153--215.  \MR{0862716}


\bibitem[KL09]{KhoLau} M. Khovanov and A. Lauda, {\em A diagrammatic approach to categorification of quantum groups, \uppercase\expandafter{\romannumeral1}}, Represent. Theory {\bf 13} (2009), 309--347. \MR{2525917}

\bibitem[Ko16]{Konno} H. Konno, {\em Elliptic Quantum Groups $U_{q, p}(gl_N)$ and $E_ {q, p}(gl_N)$}, Representation theory, special functions and Painlev\' equations, RIMS 2015, 347-417, Adv. Stud. Pure Math., 76, Math. Soc. Japan, Tokyo, 2018. \MR{3837927}

\bibitem[Laz98]{Laz} Y. Laszlo, {\em About G-bundles over elliptic curves}
Annales de l'Institut Fourier, tome 48, no 2 (1998), p. 413-424 \MR{1625614}

\bibitem[Lo77]{Loo} E. Looijenga, {\em Root systems and elliptic curves}, Invent. Math., {\bf 38} (1), (1976), 17--32. \MR{0466134}

\bibitem[Lu09]{Lur} J.~Lurie, {\em A survey of elliptic cohomology}, Algebraic topology, 219--277, Abel Symp., 4, Springer, Berlin, 2009. \MR{2597740}

\bibitem[Lu17]{HighAlg} J.~Lurie, {\em Higher Algebra}, 2017. \url{https://www.math.ias.edu/~lurie/papers/HA.pdf}

\bibitem[Lu18a]{SAG} J.~Lurie, {\em Spectral Algebraic Geometry}, 2018. 

\url{http://people.math.harvard.edu/~lurie/papers/SAG-rootfile.pdf}
%\bibitem[Lu85]{Lus85} G.~Lusztig, {\em Equivariant $K$-theory and representations of Hecke algebras}. Proc. Amer. Math. Soc. {\bf 94} (1985), no. 2, 337-342. \MR{0784189}


\bibitem[Lu18b]{AV}  J.~Lurie,  {\em Elliptic Cohomology I: Spectral Abelian Varieties.} 2018.

 \url{https://www.math.ias.edu/~lurie/papers/Elliptic-I.pdf}


\bibitem[Lu18c]{Ell2} J.~Lurie,  {\em Elliptic Cohomology II: Orientations}. 2018. 

\url{https://www.math.ias.edu/~lurie/papers/Elliptic-II.pdf}


\bibitem[Mar94]{Markman}  E.~Markman, {\em Spectral curves and integrable systems}. Compositio Math. 93 (1994), no. 3, 255--290. \MR{1300764}

\bibitem[May99]{May} J.P. May,  {\em A Concise Course in Algebraic Topology}. Chicago Lectures in Mathematics. University of Chicago Press, Chicago, IL, 1999. x+243 pp. ISBN: 0-226-51182-0; 0-226-51183-9. \MR{1702278}

\bibitem[Nak16]{Nak} H.~Nakajima, {\em Towards a mathematical definition of Coulomb branches of 3-dimensional $\calN$=4 gauge theories, I},  Adv. Theor. Math. Phys. 20 (2016), no. 3, 595–669. \MR{3565863}

\bibitem[NS65]{NS} M. S.~Narasimhan and C.S.~Seshadri, {\em Stable and unitary vector bundles on a compact Riemann surface}. Ann. of Math. (2) {\bf 82} (1965) 540--567. \MR{0184252}

\bibitem[P04]{P04} I. A. Panin, {\em Riemann-Roch theorems for oriented cohomology}, Axiomatic, enriched and motivic homotopy theory, 261-333, NATO Sci. Ser. II Math. Phys. Chem., 131, Kluwer Acad. Publ., Dordrecht, 2004. \MR{2061857}

\bibitem[Po05]{Po05} A. Polishchuk, {\em Abelian varieties, theta functions and the Fourier transform}.  Cambridge Tracts in Mathematics, 153. Cambridge University Press, Cambridge, 2003. xvi+292 pp. ISBN: 0-521-80804-9. \MR{1987784}
\bibitem[RW19]{RW} R. ~ Rimanyi, A. ~Weber, {\em Elliptic classes of Schubert varieties via Bott-Samelson resolution}.   J. Topol. 13 (2020), no. 3, 1139-1182. \MR{4100129}

\bibitem[RW20]{RW2} R. ~ Rimanyi, A. ~Weber, {\em Elliptic classes on Langlands dual flag varieties}, to appear in Comm.  Contemp. Math.,  \arxiv{2007.08976}

\bibitem[R08]{Rouq08} R.~Rouquier, {\em 2-Kac-Moody algebras}, preprints, (2008). \arxiv{0812.5023}


\bibitem[S80]{Sie} C.L.~Siegel, {\em Advanced Analytic Number Theory}, Second edition. Tata Institute of Fundamental Research Studies in Mathematics, 9. Bombay, 1980. v+268 pp. \MR{0659851}

\bibitem[VV11]{VV} M.~Varagnolo and E.~Vasserot, {\em Canonical bases and KLR-algebras}, J. Reine Angew. Math. {\bf 659} (2011), 67--100. \MR{2837011}

\bibitem[Vas05]{Vas} E.~Vasserot, {\em Induced and simple modules of double affine Hecke algebras}. 
Duke Math. J. {\bf 126} (2005), no. 2, 251--323. \MR{2115259}

\bibitem[Vish07]{Vish} A. Vishik, {\em Symmetric operations in algebraic cobordisms}, Adv. Math. 213 (2007), no. 2, 489--552.
\MR{2332601}

\bibitem[YZ17]{YZ}Y. Yang and G. Zhao, {\em Quiver varieties and elliptic quantum groups}, preprints, (2017). \arxiv{1708.01418}

\bibitem[ZZ14]{ZZ14}
G. Zhao and C. Zhong, {\em Geometric representations of the formal affine Hecke algebra},     Adv. Math., 317 (2017), 50-90. \MR{3682663}


\bibitem[Zh13]{Zho} C. Zhong, {\em On the formal affine Hecke algebra}, J. Inst. Math. Jussieu,  14 (2015), no. 4, 837-855. \MR{3394129}


\end{thebibliography}
\end{document}